\documentclass{amsart}

\RequirePackage{amsmath} \RequirePackage{amssymb}
\usepackage{amscd,latexsym,amsthm,amsfonts,amssymb,amsmath,amsxtra}
\usepackage[all]{xy}
\usepackage[colorlinks=true, urlcolor=blue,bookmarks=true,bookmarksopen=true,
citecolor=blue,]{hyperref}

\newcommand{\vphi}{\varphi}
\newcommand{\ra}{\rightarrow}

    \newcommand{\BA}{{\mathbb {A}}} 
    \newcommand{\BC}{{\mathbb {C}}} 
     
    \newcommand{\BG}{{\mathbb {G}}}

    \newcommand{\BQ}{{\mathbb {Q}}} \newcommand{\BR}{{\mathbb {R}}}

     \newcommand{\BZ}{{\mathbb {Z}}}
    \newcommand{\CA}{{\mathcal {A}}} 
     
    \newcommand{\CE}{{\mathcal {E}}} 
    \newcommand{\CG}{{\mathcal {G}}} \newcommand{\CH}{{\mathcal {H}}}

    \newcommand{\CO}{{\mathcal {O}}} 
     \newcommand{\CR}{{\mathcal {R}}}
    \newcommand{\CS}{{\mathcal {S}}} 
    \newcommand{\CU}{{\mathcal {U}}} 
    \newcommand{\CW}{{\mathcal {W}}}

    \newcommand{\Ind}{{\mathrm{Ind}}}

    \newcommand{\ord}{{\mathrm{ord}}} \newcommand{\rank}{{\mathrm{rank}}}
     
    \renewcommand{\Re}{{\mathrm{Re}}} 
    \newcommand{\Res}{{\mathrm{Res}}} 
      
      \newcommand{\fp}{{\mathfrak {p}}}

     \newcommand{\tr}{{\mathrm{tr}}}
      
    \newcommand{\vol}{{\mathrm{vol}}}

    \newcommand{\pair}[1]{\langle {#1} \rangle}

    \newcommand{\bs}{\backslash}
    \theoremstyle{plain}
    \newtheorem{thm}{Theorem}[section] 
    \newtheorem{lem}[thm]{Lemma}  \newtheorem{prop}[thm]{Proposition}

    \newcommand{\SHer}{{\mathrm{SHer}}}

  \numberwithin{equation}{section}

\begin{document}

\title{A weak second term identity of the regularized Siegel-Weil formula for unitary groups}
\author{Wei Xiong}
\address{College of Mathematics and Econometrics, Hunan University, Changsha 410082, China}
\email{weixiong@amss.ac.cn}
\thanks{}
\date{}
\begin{abstract}
Following W.T.Gan and S.Takeda, we obtain a weak second term identity of the regularized Siegel-Weil formula for the unitary dual pair $(U(n,n),U(V))$, where $V$ is a split hermitian space of dimension $2r$ with $r+1\leq n \leq 2r-1$. As an application, we obtain a Rallis inner product formula for theta lifts from $U(W)$ to $U(V)$ for a skew-hermitian space $W$ of dimension $n$.
\end{abstract}
\maketitle
\tableofcontents
\pagestyle{headings}
\markboth{}{}

\section{Introduction}
Let $F$ be a totally real number field with adele ring $\BA$, and let $E$ be a totally imaginary quadratic field extension of $F$. Let $G=U(n,n)$ be the quasi-split unitary group in $2n$ variables and let $H=U(V)$ be the unitary group of a split hermitian space $V$ of dimension $2r$ over $E$. Then $G$ and $H$ form a unitary dual pair. Fix a non-trivial additive character $\psi: \BA/F \ra \BC^{\times}$, then $G(\BA)\times H(\BA)$ acts on the Schwartz space $\CS(V^n(\BA))$ via the Weil representation $\omega=\omega_{\psi}$. Let $S(V^n(\BA))$ be the subspace of $\CS(V^n(\BA))$ consisting of Schwartz functions on $V^n(\BA)$ which correspond to polynomials in the Fock model at every archimedean place of $F$.

Let $P$ be the Siegel parabolic subgroup of $G$ and let $K_G$ be the standard maximal compact subgroup of $G(\BA)$. For $\vphi \in S(V^n(\BA))$, define the Siegel Eisenstein series on $G(\BA)$ by
$$E^{(n,n)}(g,s;\Phi_{\vphi})=\sum_{\gamma \in P(F)\bs G(F)}\Phi_{\vphi}(\gamma g,s)$$
for $g\in G(\BA)$, where $\Phi_{\vphi}\in \Ind_{P(\BA)}^{G(\BA)}|\cdot|^s$ is the standard section associated to $\vphi$.

On the other hand, for $\vphi\in S(V^n(\BA))$, define the theta integral on $G(\BA)$ by
$$I(g,\vphi)=\int_{H(F)\bs H(\BA)}\theta(g,h;\vphi)dh,$$
where
$$\theta(g,h;\vphi)=\sum_{x\in V^n(F)}\omega(g,h)\vphi(x).$$
Weil \cite{W} showed that the theta integral converges absolutely for all $\vphi$ if and only if $n<r$. Moreover, if $n<r$, then the Siegel Eisenstein series $E^{(n,n)}(g,s;\Phi_{\vphi})$ is holomorphic at $s=\rho_{n,r}=(2r-n)/2$, and Weil \cite{W} proved the following so-called Siegel-Weil formula which identifies the value at $s=\rho_{n,r}$ of the Siegel Eisenstein series with the theta integral:
$$ E^{(n,n)}(g, \rho_{n,r};\Phi_{\vphi})=I(g,\vphi).$$

On the other hand, if $n\geq r$, then the theta integral may not converge. Following the ideas of Kudla-Rallis \cite{KR2} and Ichino \cite{Ich}, we will define a regularized theta integral
$\CE^{(n,r)}(g,s;\vphi)$ whose behavior at $s=r/2$ recovers the theta integral for good $\vphi$.

Suppose that $n\geq 2r$. Then the regularized theta integral has a simple pole at $s=\rho_r=r/2$, while the Siegel Eisenstein series is holomorphic at $s=\rho_{n,r}=(2r-n)/2$. Ichino \cite{Ich} and Yamana \cite{Yam} proved the following regularized Siegel-Weil formula:
$$ E^{(n,n)}(g,\rho_{n,r};\Phi_{\vphi})=2~ \Res_{s=\rho_r}\CE^{(n,r)}(g,s;\vphi).$$

Now suppose that $r+1 \leq n \leq 2r-1$ and fix $\vphi \in S(V^n(\BA))$. Then the regularized theta integral $\CE^{(n,r)}(g,s;\vphi)$ has a double pole at $s=\rho_r=r/2$, and we write its Laurent expansion at $s=\rho_r$ as
$$\CE^{(n,r)}(g,s;\vphi)=\frac{B_{-2}^{(n,r)}(\vphi)(g)}{(s-\rho_r)^2}+\frac{B_{-1}^{(n,r)}(\vphi)(g)}{s-\rho_r}+ B_0^{(n,r)}(\vphi)(g)+O(s-\rho_r).$$

On the other hand, the Siegel Eisenstein series $E^{(n,n)}(g,s;\Phi_{\vphi})$ has a simple pole at $s=\rho_{n,r}=(2r-n)/2$, and we write its Laurent expansion there as
$$E^{(n,n)}(g,s;\Phi_{\vphi})=\frac{A^{(n,r)}_{-1}(\vphi)(g)}{s-\rho_{n,r}}+A^{(n,r)}_0(\vphi)(g)+O(s-\rho_{n,r}).$$
Note that each $B^{(n,r)}_d$ and each $A^{(n,r)}_d$ defines a map from $S(V^n(\BA))$ to the space of automorphic forms on $G$ respectively.
The regularized Siegel-Weil formula predicts that there is a first term identity which relates $B_{-2}^{(n,r)}(\vphi)$ with $A_{-1}^{(n,r)}(\vphi)$, and more importantly, there is a second term identity which relates $A^{(n,r)}_0(\vphi)$ with $B^{(n,r)}_{-1}(\vphi)$. In this paper, we prove such identities for some good Schwartz functions. The results are as follows.

We first obtain a weak first term identity. See Section \ref{Section 2.5} for the definition of the spherical Schwartz function.\\
 
{\bf Weak first term identity} (see Propsition \ref{Prop 5.3}):
Suppose $r+1\leq n \leq 2r-1$. If $\vphi\in S(V^n(\BA))$ is in the $G(\BA)$-span of the spherical Schwartz function, then
$$ A^{(n,r)}_{-1}(\vphi)=b_{n,r} B^{(n,r)}_{-2}(\vphi)$$
for some non-zero constant $b_{n,r}$.\\

We also obtain a weak second term identity. See Section \ref{Section 2.6} for the definition of Ikeda map.\\

{\bf Weak second term identity} (see Theorem \ref{Thm 5.4}):
Suppose $r+1\leq n \leq 2r-1$. If $\vphi\in S(V^n(\BA))$ is in the $G(\BA)$-span of the spherical Schwartz function, then
$$\begin{aligned}
&B^{(n,r)}_{-1}(\vphi)+|D_E|^{\frac{n(n-2r)}{2}} \prod_{i=0}^{2r-n-1}\left(\frac{\xi_E(-i)}{\xi_F(2r-2i-1,\eta)\xi_F(2r-2i)}\right) B^{(n,n-r)}_0(\mathrm{Ik}(\vphi))\\
& \equiv A^{(n,r)}_0(\vphi) \mod \mathrm{Im}(A^{(n,r)}_{-1}),
\end{aligned}$$
where $\mathrm{Ik}(\vphi) \in S(V_c^n(\BA))$ is the image of $\vphi$ under the Ikeda map, $V_c$ is the split hermitian space of dimension $2(n-r)$, and $\mathrm{Im}(A^{(n,r)}_{-1})$ is the image of the $A^{(n,r)}_{-1}$ map.\\

As an application of the weak second term identity of the regularized Siegel-Weil formula, we obtain a Rallis inner product formula for theta lifts from unitary groups.

Consider the theta lifting from $U(W)$ to $U(V)$, where $W$ is a skew-hermitian space of dimension $n$. Suppose the Weil representation of $U(W)(\BA)\times U(V)(\BA)$ acts on the space $\CS(X(\BA))$, where $X$ is a vector space of dimension $nr$ over $E$. Let $\pi$ be a cuspidal automorphic representation of $U(W)(\BA)$. For $f\in \pi$ and $\phi\in \CS(X(\BA))$, the theta lift $\theta_{2r}(f,\phi)$ of $f$ to $U(V)(\BA)$ (with respect to $\phi$) is the function on $U(V)(\BA)$ given by
$$\theta_{2r}(f,\phi)(h)=\int_{U(W)(F)\bs U(W)(\BA)}\left(\sum_{x\in X(F)}\omega(g,h)\phi(x)\right)f(g)dg $$
for $h\in U(V)(\BA)$. The global theta lift of $\pi$ to $U(V)$, denoted by $\Theta_{2r}(\pi)$, is the space of all the functions $\theta_{2r}(f,\phi)$ for all $f\in \pi$ and $\phi \in \CS(X(\BA))$.

We obtain the following inner product formula for $r+1\leq n \leq 2r-1$. See Section \ref{Section 6} for more notation.\\

{\bf Inner product formula} (see Theorem \ref{Thm 6.7}):
Suppose $r+1\leq n \leq 2r-1$. Let $\pi$ be a cuspidal automorphic representation of $U(W)(\BA)$ with $\Theta_{2(r-1)}(\pi)=0$. Let $\vphi\in S(V^n(\BA))$ be such that $\Phi_{\vphi}$ is factorizable as $\otimes'\Phi_v$. Write $\vphi=\sum_{i}\sigma(\phi_{1,i}\otimes \overline{\phi_{2,i}})$, where each $\phi_{1,i}$ or $\phi_{2,i}$ is in $\CS(X(\BA))$. Then for $f_1, f_2\in \pi$, we have
$$\begin{aligned}
&\quad \sum_{i}\pair{\theta_{2r}(f_1, \phi_{1,i}), \theta_{2r}(f_2,\phi_{2,i})}\\
&=\frac{1}{d_n^S(\rho_{n,r})}\mathrm{Val}_{s=\rho_{n,r}}\left(L^S(s+\frac{1}{2},\pi)\prod_{v\in S}Z_v(s,f_{1,v},f_{2,v},\Phi_v,\pi_v) \right),
\end{aligned}$$
where $\mathrm{Val}_{s=s_0}$ means the constant term of the Laurent expansion at $s=s_0$, and $S$ can be chosen to be any finite set of places of $F$ such that if $v\notin S$, then $v$ is finite and everything is unramified at $v$.\\

These results are analogs of the results of Gan-Takeda \cite{GT} for the orthogonal-symplectic dual pair. As early as in 1989, Kudla \cite{K2} found a spherical second term identity for the symplectic-orthogonal dual pair, but the proof has never appeared. Subsequently, the second term identities of the regularized Siegel-Weil formula have been established for some small groups. For example, Kudla-Rallis-Soudry \cite{KRS} obtained a second term identity for the symplectic-orthogonal dual pair $(Sp(4),O(4))$, and V.Tan \cite{Tan1} obtained a second term identity for the unitary dual pair $(U(2,2), U(3))$. In 2009, Gan and Takeda \cite{GT} proved a weak second term identity for the orthogonal-symplectic dual pair. Motivated by studying theta lifting from $U(1)$ to $U(1,1)$, we proved a second term identity for the unitary dual pair $(U(1,1), U(1,1))$ in December 2011 (see \cite{X1}). Then we started to study the second term identity for a general unitary dual pair, and we succeeded in proving a weak second term identity for split unitary groups by following the method of Gan-Takeda \cite{GT}. After we posted the first version of this paper on arXiv in May 2012, W.T.Gan and S.Takeda informed us that, together with Y.Qiu, they have established the second term identity for all classical groups! \footnote{Their results have appeared as \cite{GQT}.} Thus our results are just special cases of their results. Nevertheless, we hope that our results, being independent, may serve as a check on their general results.

The contents of this paper are as follows. In Section \ref{Section 2}, we give some notation and preliminaries. In Section \ref{Section 3}, we define the regularized theta integral and review the regularized Siegel-Weil formula in the case $n\geq 2r$. In Section \ref{Section 4}, we prove the second term identity of the regularized Siegl-Weil formula for the spherical Schwartz function in the case $r+1\leq n \leq 2r-1$. In Section \ref{Section 5}, we prove the weak second term identity. In Section \ref{Section 6}, we prove the inner product formulae. The last section, Section \ref{Appendix}, is an appendix in which we prove Proposition \ref{Prop 4.1}.

Before concluding this Introduction, we emphasize that this paper is greatly indebted to the works of Gan-Takeda \cite{GT}, Kudla-Rallis \cite{KR2}, and Ichino \cite{Ich}. Precisely, we follow Gan-Takeda's inductive method to establish the weak second term identity, and we follow Kudla-Rallis and Ichino to define the regularized theta integral.\\

{\bf Acknowledgements}. I would like to thank Professor Ye Tian, Professor Song Wang, and Professor Linsheng Yin for their kind advice and constant encouragement. And I would like to thank Professor Wee Teck Gan for encouragement and suggestions after the first version of this paper appeared. In addition, I would like to thank an anonymous referee of Math.Z. for a very careful reading of the manuscript and many comments leading to a much improved exposition. During the preparation of the final version of this paper, I was supported by The Young Teachers Program of Hunan University.

\section{Notation and preliminaries}\label{Section 2}
\subsection{Spaces and groups}
We assume throughout this paper that $F$ is a totally real number field and $E$ is a totally imaginary quadratic field extension of $F$, as in Ichino's paper \cite{Ich}, in order to use several results in \cite{Ich} directly. Let $x\mapsto \bar{x}$ be the non-trivial automorphism of $E$ over $F$.

Let $\CW$ be the space of row vectors $M_{1, 2n}(E)$ with a skew-hermitian form $\pair{,}: \CW\times \CW \ra E$ given by
$$\pair{w_1,w_2}=w_1 \begin{pmatrix}  & 1_n\\ -1_n & \end{pmatrix} {}^t\!\bar{w}_2$$
for $w_1, w_2\in \CW$, where ${}^t\!w$ denotes the transpose of the matrix $w$.

The unitary group $U(\CW)$ of $\CW$ is given by
$$ U(\CW)=\left\{g\in \mathrm{R}_{E/F}GL_{2n}| g \begin{pmatrix}  & 1_n\\ -1_n&  \end{pmatrix} {}^t\!\bar{g}=\begin{pmatrix} & 1_n\\-1_n & \end{pmatrix} \right\}.$$

Let $V=M_{2r,1}(E)$ be the space of column vectors with a hermitian form $(,): V\times V \ra E$ given by
$$ (v_1,v_2)={}^t\!\bar{v}_1 \begin{pmatrix} & 1_r\\ 1_r & \end{pmatrix} v_2$$
for $v_1, v_2\in V$. Then the unitary group $U(V)$ of $V$ is given by
$$ U(V)=\left\{h\in \mathrm{R}_{E/F}GL_{2r}| {}^t\! \bar{h} \begin{pmatrix} & 1_r\\1_r& \end{pmatrix} h=\begin{pmatrix} & 1_r\\1_r& \end{pmatrix} \right\}.$$

For simplicity, we often write $G=U(\CW)=U(n,n)$ and $H=U(V)$.

Let $P=M_n N_n$ be the Siegel parabolic subgroup of $G=U(n,n)$, where the Levi factor $M_n$ is
$$ M_n=\left\{m(a)=\begin{pmatrix} a & \\ & {}^t\!\bar{a}^{-1}\end{pmatrix}| a\in R_{E/F}GL_{n} \right\},$$
and the unipotent radical $N_n$ is
$$ N_n=\left\{u(b)=\begin{pmatrix}1_n & b\\ & 1_n\end{pmatrix}| b\in \mathrm{Her}_n \right\},$$
where $\mathrm{Her}_n=\{x\in R_{E/F}M_n|  {}^t\!\bar{x}=x\}$ is the space of $n\times n$ hermitian matrices.

For each place $v$ of $F$, let $F_v$ be the completion of $F$ at $v$, and let $E_v=E\otimes_F F_v$. Then $E_v$ is a quadratic field extension of $F_v$ if $v$ is inert, and $E_v=F_v\oplus F_v$ if $v$ is split. For a finite place $v$ of $F$, let $\CO_v$ be the ring of integers of $F_v$, $\fp_v$ the maximal ideal of $\CO_v$, and $q_v$ the cardinality of the residue field of $F_v$. If $E_v$ is a quadratic field extension of $F_v$, let $\CO_{E_v}$ be the ring of integers of $E_v$, $\fp_{E_v}$ the maximal ideal of $\CO_{E_v}$, and $q_{E_v}$ the cardinality of the residue field of $E_v$. If $E_v=F_v\oplus F_v$, put $\CO_{E_v}=\CO_v\oplus \CO_v$, $\fp_{E_v}=\fp_v\oplus \fp_v$, and $q_{E_v}=q_v^2$.

Let $\BA$ be the adele ring of $F$, and $\BA_E$ the adele ring of $E$. Fix a non-trivial additive character $\psi=\otimes' \psi_v: \BA/F \ra \BC^{\times}$. We choose various measures as follows. For any place $v$ of $F$, fix the Haar measure $dx_v$ on $F_v$ to be self-dual with respect to $\psi_v$. If $v$ is finite, then $\vol(\CO_v)=q_v^{-c_v/2}$, where $c_v=\ord(\psi_v)$ is the order of $\psi_v$ (i.e. the largest integer $d$ such that $\psi_v$ is trivial on $\fp_v^{-d}$). If $E_v$ is a quadratic field extension of $F_v$, fix the Haar measure on $E_v$ to be self-dual with respect to $\psi_{E_v}:=\psi_v \circ \tr_{E_v/F_v}$. If $E_v=F_v\oplus F_v$, fix the Haar measure on $E_v$ to be the product of the fixed measures on $F_v$. We take the Haar measures on $\BA$ and on $\BA_E$ to be the product measure. Then the measure on $\BA$ is self-dual with respect to $\psi$, and the measure on $\BA_E$ is self-dual with respect to $\psi_E:=\psi \circ \tr_{E/F}$.

For a place $v$ of $F$, let $|~|_v$ be the standard absolute value on $F_v$. If $v$ is a finite place and $\varpi_v$ is a uniformizer of $F_v$, then $|\varpi_v|_v=q_v^{-1}$; if $v$ is a real place, then $|~|_v$ is the usual absolute value on $\BR$. Define an absolute value on $E_v$ by $\|x\|_{E_v}=|x \bar{x}|_{v}$. Denote the corresponding absolute value on $\BA$ by $|~|$, and the absolute value on $\BA_E$ by $\|~\|$.

For each place $v$ of $F$, let $K_{G_v}$ be the standard maximal compact subgroup of $G_v:=G(F_v)$ given by
$$ K_{G_v}=
\begin{cases}
G_v\cap GL_{2n}(\CO_{E_v}) & \text{if $v$ is finite,}\\
G_v\cap U(2n) & \text{if $v$ is real,}
\end{cases}$$
where $U(m)=\{g\in GL_m(\BC): {}^t\! \bar{g} g=1_m\}$ is the usual unitary group of rank $m$.

Let $K_G=\prod_v K_{G_v}$ be the standard maximal compact subgroup of $G(\BA)$. Then there is Iwasawa decomposition $G(\BA)=P(\BA)K_G$.

For $g=pk \in G(\BA)$ with $p=m(a)u(b)\in P(\BA)$ and $k\in K_G$, put
$$ |a(g)|=\|\det a\|.$$

For an algebraic group $\CG$ over $F$, we often write $[\CG]$ for the quotient group $\CG(F)\bs \CG(\BA)$. We write $\CA(\CG)$ for the space of automorphic forms on $\CG$ (without the $K_G$-finite condition).

\subsection{Zeta functions}
For a number field $k$, let $$\zeta_k(s)=\Gamma_{\BR}(s)^{r_1}\Gamma_{\BC}(s)^{r_2} \prod_{v\nmid \infty} (1-q_v^{-s})^{-1}$$ be the complete zeta function of $k$, where $q_v$ is the cardinality of the residue field of $k$ at a finite place $v$, $r_1$ (resp. $r_2$) is the number of real (resp. complex) places of $k$,
$$\Gamma_{\BR}(s)=\pi^{-s/2}\Gamma(s/2),$$
and
$$\Gamma_{\BC}(s)=2(2\pi)^{-s}\Gamma(s).$$
Let $$\xi_k(s)=|D_k|^{s/2}\zeta_k(s)$$ be the normalized zeta function of $k$, where $D_k$ is the discriminant of $k$.

Then there is functional equation
$$\xi_k(s)=\xi_k(1-s).$$

Note that $\xi_k(s)$ is holomorphic in $s$ except for $s=0$ and $s=1$, at which it has simple poles. By abuse of notation, we write $\xi_k(0)=\Res_{s=0}\xi_k(s)$ and $\xi_k(1)=\Res_{s=1}\xi_k(s)$.

A {\sl Hecke character} $\chi$ of $k$ is a (unitary) character $\chi: \BA_k^{\times}/k^{\times}\ra \BC^1$, where $\BA_k$ is the adele ring of $\BA$ and $\BC^1=\{z\in \BC:|z|=1\}$.

Return to the quadratic extension $E/F$. Let $\eta=\eta_{E/F}:\BA^{\times}/F^{\times}\ra \{\pm 1\}$ be the quadratic Hecke character of $F$ associated to $E/F$ by class field theory. Let $L_F(s,\eta)=\prod_{v\leq \infty}L(s,\eta_v)$ be the complete Hecke L-function of $\eta$. Put $\xi_F(s,\eta)=|D_F N_{F/\BQ}(\mathfrak{f}_{\eta})|^{s/2}L_F(s,\eta)$, where $\mathfrak{f}_{\eta}$ is the conductor of $\eta$. Then $\xi_E(s)=\xi_F(s)\xi_F(s,\eta)$.

\subsection{Weil representation}
We use the trivial Hecke character of $E$ to construct the Weil representation of $U(\CW)\times U(V)$, since both $\CW$ and $V$ are of even dimension (see \cite{HKS,K1}). Then the Weil representation $\omega$ of $G(\BA)\times H(\BA)$ can be realized in the Schrodinger model $\CS(V^n(\BA))$, or in the mixed model $\CS(\CW^r(\BA))$. Note that $V^n=M_{2r,n}(E)$ and $\CW^r =M_{r,2n}(E)$.

These two models of the Weil representaion are related by the partial Fourier transform
$$\begin{aligned}
&\CS(V^n(\BA)) \ra \CS(\CW^r(\BA)),\\
&\quad \quad \quad \quad \vphi \mapsto \hat{\vphi},
\end{aligned}$$
where
$$\hat{\vphi}(a,b)=\int_{M_{r,n}(\BA_E)}\vphi\begin{pmatrix}x\\a \end{pmatrix} \psi(\tr_{E/F}\tr(b {}^t\!\bar{x}))dx$$
for $(a,b)\in M_{r,2n}(\BA_E)$ with $a,b \in M_{r,n}(\BA_E)$ (see \cite{Ich} p.250), and we take the self-dual Haar measure on $\BA_E$ with respect to the character $\psi_E:=\psi \circ \tr_{E/F}$.

\subsection{Degenerate principal series}
For $1\leq r \leq n$, let $P_r$ be the parabolic subgroup of $G=U(\CW)$ stabilizing the totally isotropic subspace $\{(0_n,a_1,\ldots,a_r,0_{n-r}|a_i \in E\}$ of $\CW$. Then $P_r$ has Levi decomposition $P_r=N_r M_r$ with Levi factor $M_r\cong R_{E/F}GL_r \times U(n-r,n-r)$.
For $p=u m_r(a,g)\in P_r(\BA)$, let $|a_r(p)|=\|\det a\|$. We have $G(\BA)=P_r(\BA) K_G$, where $K_G$ is the standard maximal compact subgroup of $G(\BA)$.

For $s\in \BC$, the {\sl degenerate principal series} $I^{n,r}(s)=\Ind_{P_r(\BA)}^{G(\BA)}|\cdot|^s$ consists of smooth functions $f(-,s): G(\BA) \ra \BC$ such that
\begin{itemize}
\item $f(pg,s)=|a_r(p)|^{s+\frac{2n-r}{2}}f(g,s)$ for $p\in P_r(\BA), g\in G(\BA)$;

\item $f(-,s)$ is right $K_G$-finite.
\end{itemize}
An element in $I^{n,r}(s)$ is called a {\sl section}. A {\sl standard section} is one whose restriction to $K_G$ is independent of $s$. A {\sl spherical section} is one whose restriction to $K_G$ is a constant. A {\sl holomorphic section} is one which is holomorphic with respect to $s\in \BC$, i.e. $f(g,s)$ is holomorphic in $s$ for every $g\in G(\BA)$.

\subsection{Spherical Schwartz function}\label{Section 2.5}
Following \cite{GT} (see also \cite{GI} p.39), we define the {\sl spherical Schwartz function}
$$\vphi^0=\otimes_v \vphi_v^0\in \CS(V^n(\BA))=\CS(M_{2r,n}(\BA_E))$$
to be such that the partial Fourier transform  $\hat{\vphi}_v^0$ is
\begin{itemize}
\item the characteristic function of $M_{r,2n}(\CO_{E_v})$ if $v$ is finite,

\item the Gaussian if $v$ is archimedean, i.e. $\hat{\vphi}_v^0(x)=\exp(-2\pi \tr(x{}^t\!\bar{x}))$ for $x\in M_{r,2n}(\BC)$.
\end{itemize}
\medskip
It is easy to see that if $v$ is finite, then $\vphi^0_v$ is equal to
$$q_{E_v}^{-rn \cdot c_{E_v}/2}\times \text{the characteristic function of $M_{r,n}(\mathfrak{p}_{E_v}^{-c_{E_v}})\oplus M_{r,n}(\CO_{E_v})$},$$
where $c_{E_v}$ is the order of $\psi_{E_v}:=\psi_v \circ \tr_{E_v/F_v}$ is $v$ is finite inert, and $c_{E_v}=c_v$ is the order of $\psi_v$ if $v$ is finite split;
and if $v$ is archimedean, then $\vphi^0_v$ is the Gaussian given by
$$\vphi_v^0(x)=\exp(-2\pi \tr({}^t\!\bar{x}x)), \quad \quad \text{for $x\in M_{2r,n}(\BC)$}.$$

Note that $\vphi^0(0)=\mathfrak{D}_{E/F}^{-rn/2}|D_F|^{-rn}=|D_E|^{-rn/2}$.

Following \cite{Yam}, let $S(V^n(\BA))$ be the subspace of $\CS(V^n(\BA))$ consisting of Schwartz functions on $V^n(\BA)$ which correspond to polynomials in the Fock model of the Weil representation at every archimedean place of $F$. Then $S(V^n(\BA))=\otimes' S(V_v^n)$, where $S(V_v^n)=\CS(V_v^n)$ is the space of locally constant functions on $V_v^n$ if $v$ is non-archimedean, and $S(V_v^n)$ is the space of polynomials in the Fock model if $v$ is real.

For convenience, we explain the space $S(V_v^n)$ of polynomials in the Fock model if $v$ is real. See \cite{Ad} for relevant material on the Fock model. Suppose $V_v$ is a hermitian space over $\BC$ of signature $(p,q)$. Let $\{e_1,\ldots,e_p, f_1,\ldots, f_q\}$ be a standard basis of $V_v$ with $(e_i,e_j)=\delta_{ij}$, $(f_i,f_j)=\delta_{ij}$, and $(e_i,f_j)=0$, and write $V_v=V_{+}\oplus V_{-}$, where $V_{+}=\mathrm{span}\{e_1,\ldots, e_p\}$ and $V_{-}=\mathrm{span}\{f_1,\ldots, f_q\}$. Define
$$(x,x)_{+}=
\begin{cases}
(x,x) & \text{if $x\in V_{+}$}\\
-(x,x) &\text{if $x\in V_{-}$}
\end{cases},$$
and let $\vphi_v^0$ be the Gaussian given by $\vphi_v^0(x)=\exp({-\pi \tr((x,x)_{+})})$ for $x\in V_v^n$. Then $S(V_v^n)$ is the space of functions of the form $P(x)\vphi_v^0(x)$, where $P$ is a polynomial on $V_v^n$ and $\vphi_v^0$ is the Gaussian. See \cite{KRS} p.489 for the case of a quadratic space. In our case where $V_v$ is split (so $p=q=r$), $\tr((x,x)_{+})=2\tr(^t\!\bar{x}x)$ for $x\in V_v^n$, so $\vphi_v^0(x)=\exp(-2\pi\tr(^t\!\bar{x}x))$ for $x\in V_v^n$.

Note that $G(\BA)$ acts on $\vphi^0$ via the Weil representation. Define
$$\text{$ S(V^n(\BA))^{\circ}$= the $G(\BA)$-span of $\vphi^0$ in $S(V^n(\BA))$}.$$

\subsection{The complementary space and Ikeda's map}\label{Section 2.6}
Assume that $n\geq r$ as before. Recall that $V$ is the split hermitian space of dimension $2r$ over $E$. The {\sl complementary space} $V_c$ of $V$ with respect to $\CW$ is defined to be the split hermitian space of dimension $2(n-r)$ over $E$, i.e. $V_c$ is the space of column vectors $M_{2n-2r,1}(E)$ with the hermitian form given by the matrix
$$\begin{pmatrix} & 1_{n-r}\\ 1_{n-r}& \end{pmatrix}.$$

For $H=U(V)$, we choose a maximal compact subgroup $K=\prod_v K_v$ of $H(\BA)$ as follows:
$$ K_v=
\begin{cases}
\text{the stabilizer of $(\fp_{E_v}^{-c_{E_v}})^r \oplus (\CO_{E_v})^r $} & \text{if $v$ is finite,}\\
H_v\cap U(2r) &\text{if $v$ is infinite.}
\end{cases}$$
Then $\omega(k)\vphi^0=\vphi^0$ for all $k\in K$.\\

Now assume that $\dim V_c>\dim V$, namely $n>2r$.

For $\vphi \in \CS(V_c^n(\BA))$ and $x\in V^n(\BA)$, put
$$\mathrm{Ik}^{(n,r)}(\vphi)(x)=\int_{M_{n-2r,n}(\BA_E)}\vphi_{K_{c}} \begin{pmatrix}y\\ x\\ 0\end{pmatrix}dy,$$
where the Haar measure $dy$ on $M_{n-2r,n}(\BA_E)$ is the product of the self-dual Haar measures on $\BA_E$ with respect to $\psi \circ \tr_{E/F}$, $K_{c}$ is a maximal compact subgroup of $U(V_c)(\BA)$ defined for the split hermitian space $V_c$ as above, and
$$ \vphi_{K_c}(v)=\int_{K_c}\vphi(kv)dk$$
for $v\in V_c(\BA)$,
where the Haar measure $dk$ on $K_c$ is such that $\vol(K_c)=1$.

The map $\mathrm{Ik}^{(n,r)}$ is defined by Ichino in \cite{Ich} pp.252-253, and we call it {\sl Ikeda's map} following Gan-Takeda \cite{GT}.

It is easy to see that Ikeda's map
$$\mathrm{Ik}=\mathrm{Ik}^{(n,r)}: \CS(V_c^n(\BA)) \ra \CS(V^n(\BA))$$
is $G(\BA)$-intertwining, i.e.,
$$\omega(g)\mathrm{Ik}(\vphi)=\mathrm{Ik}(\omega(g)\vphi), \quad \text{for $g\in G(\BA)$}.$$

If $\vphi_c^0$ is the spherical Schwartz function in $\CS(V_c^n(\BA))$, then $\mathrm{Ik}(\vphi_c^0)=\vphi^0$ is the spherical Schwartz function in $\CS(V^n(\BA))$.\\

If $\dim V> \dim V_c$, i.e. $n<2r$, we can also define Ikeda's map
$$\mathrm{Ik}=\mathrm{Ik}^{(n,n-r)}: \CS(V^n(\BA))\ra \CS(V_c^n(\BA)),$$
which is given by
$$\mathrm{Ik}(\vphi)(x)=\int_{M_{2r-n,n}(\BA_E)}\vphi_{K} \begin{pmatrix}y\\ x\\ 0\end{pmatrix}dy$$
for $\vphi \in \CS(V^n(\BA))$ and $x\in V_c^n(\BA)$, where
$$\vphi_K(v)=\int_K \vphi(kv)dk$$
for $v\in V^n(\BA)$.

\section{The regularized Siegel-Weil formula}\label{Section 3}
\subsection{Eisenstein series}
Suppose $1\leq r\leq n$. For a holomorphic section $f \in I^{n,r}(s)$, define the associated Eisenstein series on $G(\BA)$ by
$$E^{(n,r)}(g,s;f)=\sum_{\gamma \in P_r(F)\bs G(F)}f(\gamma g,s), \quad \quad \text{for $g\in G(\BA)$}. $$
It is known that this series converges absolutely for $\Re(s)\gg 0$ and admits a meromorphic continuation to the whole $s$-plane.

We call $E^{(n,r)}(g,s; f)$ a {\sl Siegel Eisenstein series} if $r=n$, and a {\sl non-Siegel Eisenstein series} otherwise.
We write $E^{(n,r)}(g,s)$ for the Eisenstein series corresponding to the spherical section $f^0$ in $I^{n,r}(s)$ with $f^0(1)=1$, and call it the {\sl spherical Eisenstein series}.

For a holomorphic section $\Phi\in I^{n,n}(s)$, the Siegel Eisenstein series $E^{(n,n)}(g,s;\Phi)$ has a functional equation
$$E^{(n,n)}(g,s;\Phi)=E^{(n,n)}(g,-s;M(s)\Phi),$$
where
$$M(s)\Phi(g,s)=\int_{N_n(\BA)}\Phi(w_n ug,s)du$$
for $\Re(s)\gg 0$ and is given by meromorphic continuation in general, and
$$w_n=\begin{pmatrix} 0 & 1_n\\ -1_n & 0 \end{pmatrix}.$$

The location of the poles of $E^{(n,n)}(g,s;\Phi)$ is determined by V.Tan \cite{Tan2}.
\begin{prop}
Let $\Phi$ be a standard section in $I^{n,n}(s)$. Then $E^{(n,n)}(g,s;\Phi)$ is holomorphic in $\Re(s)\geq 0$ except for simple poles at
$$ s_0\in \{\frac{n}{2}-j|j\in \BZ, 0\leq j < \frac{n}{2}\}.$$
\end{prop}
Note that the same is true for a holomorphic section since it is a finite linear combination of standard sections with holomorphic coefficients.\\

The behavior of $E^{(n,n)}(g,s;\Phi)$ at $s=\rho_{n,r}=(2r-n)/2$ is of great interest to us. The Siegel Eisenstein series $E^{(n,n)}(g,s;\Phi)$ has at most a simple pole at $s=\rho_{n,r}$, and we write its Laurent expansion at $\rho_{n,r}$ as
$$E^{(n,n)}(g,s;\Phi)=\frac{A_{-1}^{(n,r)}(\Phi)(g)}{s-\rho_{n,r}}+A_{0}^{(n,r)}(\Phi)(g)+O(s-\rho_{n,r}).$$

If $\vphi \in S(V^n(\BA))$, we can define a standard section $\Phi_{\vphi}\in I^{n,n}(s)$ by
$$\Phi_{\vphi}(g,s)=|a(g)|^{s-\rho_{n,r}}\omega(g)\vphi(0), \quad \quad \text{for $g\in G(\BA)$}. $$

The Eisenstein series $E(g,s;\Phi_{\vphi})$ has at most a simple pole at $s=\rho_{n,r}$, and we write its Laurent expansion there as
$$ E^{(n,n)}(g,s;\Phi_{\vphi})=\frac{A_{-1}^{(n,r)}(\vphi)(g)}{s-\rho_{n,r}}+A_{0}^{(n,r)}(\vphi)(g)+O(s-\rho_{n,r}).$$
Then we have a map
$$ A^{(n,r)}_d: S(V^n(\BA)) \ra \CA(G).$$
Note that $A^{(n,r)}_d$ is $G(\BA)$-intertwining if $A^{(n,r)}_d$ is the leading term of the Laurent expansion.\\

For the spherical Schwartz function $\vphi^0$, we see that
$$\Phi_{\vphi^0}|_{K_G}=\Phi_{\vphi^0}(1)=\vphi^0(0).$$
Thus $\Phi_{\vphi^0}$ is a spherical section and
$$E^{(n,n)}(g,s;\Phi_{\vphi^0})=\vphi^0(0)E^{(n,n)}(g,s).$$

\subsection{Regularized theta integral}\label{Section 3.2}
Now we begin to define the regularized theta integral. We assume that $n\geq r$ from now on.

First we define an auxiliary Eisenstein series on $H(\BA)$.
Let $P_H$ be the Siegel parabolic subgroup of $H=U(V)$ which stabilizes a totally isotropic subspace of $V$ of rank $r$. Then we have Levi decomposition $P_H=M_H N_H $, where $M_H\cong R_{E/F}GL_r$ is given by
$$M_H(F)=\left\{m(a)=\begin{pmatrix} a & \\ & {}^t\!\bar{a}^{-1}\end{pmatrix} | a\in GL_r(E) \right\},$$
and $N_H$ is given by
$$N_H(F)=\left\{u(b)=\begin{pmatrix} 1 & b\\ & 1 \end{pmatrix}| b\in \SHer_r(F) \right\},$$
where $\SHer_r(F):=\{b\in M_r(E)| {}^t\!\bar{b}=-b\}$. Also we have Iwasawa decomposition $H(\BA)=P_H(\BA) K$, where $K$ is the maximal compact subgroup of $H(\BA)$ defined in Section \ref{Section 2.6}.

For $h=u(b) m(a)k\in H(\BA)$, let
$$|a(h)|=\|\det a\|,$$
and let
$$\Psi(h,s)=|a(h)|^{s+\frac{r}{2}}.$$
Define an Eisenstein series on $H(\BA)$ by
$$E(h,s)=\sum_{\gamma \in P_H(F)\bs H(F)}\Psi(\gamma h,s), \quad \quad \text{for $h\in H(\BA)$}. $$

It is known that $E(h,s)$ has a simple pole at $s=r/2$ with constant residue $\kappa$.

We fix Haar measures on various groups as follows. For the group $H=U(V)$, we take the Haar measure on $H(\BA)$ such that $\vol(H(F)\bs H(\BA))=1$. For the parabolic subgroup $P_H=M_H N_H$ where $M_H\cong R_{E/F}GL_r$, we take the left Haar measure $d_l p=d^{\times}a ~du$ on $P_H(\BA)$, where $d^{\times}a$ is the Tamagawa measure on $GL_r(\BA_E)$, and $du$ is the Haar measure on $N_H(\BA)$ such that
$\vol(N_H(F)\bs N_H(\BA))=1$. We define a constant $c_K$ such that
$$ dh=c_K ~d_l p ~dk.$$
In fact, we will show later that $c_K=\kappa$, where $\kappa$ is the constant residue of $E(h,s)$ at $s=r/2$. See the remark after Proposition \ref{Prop 5.2}.\\

For $\vphi \in \CS(V^n(\BA))$ and $g\in G(\BA)$, define the integral
$$I^{(n,r)}(g,s;\vphi)=\frac{1}{\kappa}\int_{H(F)\bs H(\BA)}\theta(g,h;\vphi)E(h,s)dh,$$
where the measure $dh$ is such that $\vol(H(F)\bs H(\BA))=1$.

Fix $\vphi \in \CS(V^n(\BA))$. Choose a good place $v$ for $\vphi$ (see \cite{Ich} p.249), and let $z=\alpha_{n,r-1,1}\in \CH(G_v// K_{G_v})$ be the element of the spherical Hecke algebra of $G_v$ as defined in \cite{Ich} pp.248-249. Then $\theta(g,h;\omega(z)\vphi)$ is rapidly decreasing on $H(F)\bs H(\BA)$ for all $g\in G(\BA)$ (see \cite{Ich} Proposition 2.4).

We have
$$\begin{aligned}
&\quad I^{(n,r)}(g,s;\omega(z)\vphi)\\
&=\frac{1}{\kappa}\int_{P_H(F)\bs H(\BA)}\theta(g,h;\omega(z)\vphi)|a(h)|^{s+\frac{r}{2}}dh\\
&=\frac{c_K}{\kappa}\int_{P_H(F)\bs P_H(\BA)}\int_K \theta(g,pk;\omega(z)\vphi)dk |a(p)|^{s+\frac{r}{2}}d_l p\\
&=\int_{P_H(F)\bs P_H(\BA)}\theta(g,p;(\omega(z)\vphi)_K)|a(p)|^{s+\frac{r}{2}}d_l p \quad \quad (\text{since $c_K=\kappa$})\\
&=\int_{GL_r(E)\bs GL_r(\BA_E)}\int_{N_H(F)\bs N_H(\BA)}\theta(g,um(a);\omega(z)\vphi_K)du \|\det a\|^{s-\frac{r}{2}}d^{\times}a\\
&=\int_{GL_r(E)\bs GL_r(\BA_E)}\theta_{N_H}(g,m(a);\omega(z)\vphi_K)\|\det a\|^{s-\frac{r}{2}}d^{\times}a,
\end{aligned}$$
where $$\vphi_K =\int_K \omega(k)\vphi dk$$ is the $K$-invariant projection of $\vphi$, with the measure $dk$ on $K$ taken to be such that $\vol(K)=1$; and
$$\theta_{N_H}(g,h;\vphi)=\int_{N_H(F)\bs N_H(\BA)}\theta(g,uh;\vphi) du$$ is the $N_H$-constant term of $\theta(g,h;\vphi)$, with the measure $du$ taken to be such that $\vol(N_H(F)\bs N_H(\BA))=1$.

Similar to \cite{KR2} Proposition 5.5.1, we have the following
\begin{prop}
$$\theta_{N_H}(g,h;\omega(z)\hat{\vphi})=\sum_{\gamma \in P_r(F)\bs G(F)}\sum_{\alpha \in GL_r(E)}\omega(\gamma g, m(\alpha)h)\omega(z)\hat{\vphi}(w_0),$$
where $w_0=(0_{r\times n}, 1_r, 0_{r\times (n-r)})$.
\end{prop}
\begin{proof}
In the mixed model,
$$\omega(u(b))\hat{\vphi}(x)=\psi(\frac{1}{2}\tr(b\pair{x,x}))\hat{\vphi}(x).$$
We have
$$\begin{aligned}
&\quad \theta_{N_H}(g,h;\omega(z)\hat{\vphi})\\
&=\int_{N_H(F)\bs N_H(\BA)}\theta(g,uh;\omega(z)\hat{\vphi})du\\
&=\int_{N_H(F)\bs N_H(\BA)} \sum_{x\in \CW^r(F)}\omega(g,uh)\omega(z)\hat{\vphi}(x)du\\
&=\int_{\SHer_r(F)\bs \SHer_r(\BA)}\sum_{x\in \CW^r(F)}\psi(\frac{1}{2}\tr(b\pair{x,x}))\omega(g,h)\omega(z)\hat{\vphi}(x)db\\
&=\sum_{x\in \CW^r(F)} \omega(g,h)\omega(z)\hat{\vphi}(x) \int_{\SHer_r(F)\bs \SHer_r(\BA)}\psi(\frac{1}{2}\tr(b\pair{x,x})) db\\
&=\sum_{x\in \CW^r(F), \pair{x,x}=0} \omega(g,h)\omega(z)\hat{\vphi}(x)\\
&=\sum_{x\in \CW^r(F), \pair{x,x}=0, \rank(x)=r}\omega(g,h)\omega(z)\hat{\vphi}(x)
\end{aligned}$$
where we have used the fact that $\omega(z)\hat{\vphi}(x)$ vanishes on $x$ with $\rank(x)<r$.

Note that the set $\{x\in \CW^r(F)|\pair{x,x}=0, \rank(x)=r\}$ is equal to $\{\alpha w_0 \gamma|\alpha \in GL_r(E), \gamma \in P_r(F)\bs G(F)\}$, where $w_0=(0_{r\times n}, 1_r, 0_{r\times (n-r)})$.

So we have
$$\begin{aligned}
&\quad \theta_{N_r}(g,h;\omega(z)\hat{\vphi})\\
&=\sum_{x\in \CW^r(F), \pair{x,x}=0, \rank(x)=r}\omega(g,h)\omega(z)\hat{\vphi}(x)\\
&=\sum_{\gamma \in P_r(F)\bs G(F)}\sum_{\alpha \in GL_r(E)}\omega(g,h)\omega(z)\hat{\vphi}({}^t\!\bar{\alpha} w_0 \gamma)\\
&=\sum_{\gamma \in P_r(F)\bs G(F)}\sum_{\alpha \in GL_r(E)}\omega(\gamma g, m(\alpha)h)\omega(z)\hat{\vphi}(w_0).
\end{aligned}$$

\end{proof}

From the above proposition, we see that
$$\begin{aligned}
&\quad I^{(n,r)}(g,s;\omega(z)\vphi)\\
&=\int_{GL_r(E)\bs GL_r(\BA_E)}\sum_{\gamma \in P_r(F)\bs G(F)}\sum_{\alpha \in GL_r(E)}\omega(\gamma g,m(\alpha)m(a))\omega(z)\vphi_{K}(w_0)\|\det a\|^{s-\frac{r}{2}} d^{\times}a\\
&=\int_{GL_r(E)\bs GL_r(\BA_E)}\sum_{\gamma \in P_r(F)\bs G(F)}\sum_{\alpha \in GL_r(E)}\omega(\gamma g,m(\alpha)m(a))\omega(z)\hat{\vphi}_{K}(w_0)\|\det a\|^{s-\frac{r}{2}} d^{\times}a\\
&=\sum_{\gamma \in P_r(F)\bs G(F)}\int_{ GL_r(\BA_E)}\omega(\gamma g, m(a))\omega(z)\hat{\vphi}_{K}(w_0)\|\det a\|^{s-\frac{r}{2}} d^{\times}a\\
&=\sum_{\gamma \in P_r(F)\bs G(F)}\int_{ GL_r(\BA_E)}\omega(\gamma g)\omega(z)\hat{\vphi}_{K}({}^t\!\bar{a} w_0)\|\det a\|^{s+n-\frac{r}{2}} d^{\times}a.
\end{aligned}$$\\

Now for $\Re(s)\gg 0$, let
$$f^{(n,r)}(g,s;\vphi)=\int_{ GL_r(\BA_E)}\omega(g)\hat{\vphi}({}^t\!\bar{a} w_0)\|\det a\|^{s+n-\frac{r}{2}} d^{\times}a,$$
where $d^{\times}a$ is the Tamagawa measure on $GL_r(\BA_E)$.

Then $f^{(n,r)}(g,s;\vphi)$ is a non-holomorphic section in $I^{n,r}(s)$, and
$$I^{(n,r)}(g,s;\omega(z)\vphi)=\sum_{\gamma \in P_r(F)\bs G(F)}f^{(n,r)}(\gamma g,s;\omega(z)\vphi_K).$$

\begin{lem}
For $\Re(s)\gg 0$,
$$f^{(n,r)}(g,s;\omega(z)\vphi)=P_z(s)f^{(n,r)}(g,s;\vphi),$$
where $$P_z(s)=\prod_{j=1}^r(C_{s-\frac{r+1}{2}+j}-C_{n-r+\frac{1}{2}}),$$
and $C_s=q_v^s+q_v^{-s}$, where $v$ is the chosen good place for $\vphi$.
\end{lem}
\begin{proof}
Similar to \cite{KR2} Lemma 5.5.3. See also \cite{Ich} p.249.
\end{proof}

For our fixed $\vphi \in \CS(V^n(\BA))$, define the {\sl regularized theta integral}
$$\CE^{(n,r)}(g,s;\vphi)=\frac{1}{P_z(s)}I^{(n,r)}(g,s;\omega(z)\vphi).$$

Then
$$\begin{aligned}
&\CE^{(n,r)}(g,s;\vphi)=\frac{1}{\kappa \cdot P_z(s)}\int_{H(F)\bs H(\BA)}\theta(g,h;\omega(z)\vphi)E(h,s)dh\\
&\quad \quad \quad \quad \quad \quad =\sum_{\gamma \in P_r(F)\bs G(F)}f^{(n,r)}(\gamma g,s;\vphi_K)
\end{aligned}$$
is an Eisenstein series associated to a non-holomorphic section in $I^{n,r}(s)$.\\

Let $z'\in \CH(H_v//K_v)$ be the element in the spherical Hecke algebra of $H_v$ corresponding to $z$ under Satake isomorphism. Then $z'$ acts on the space $\CA(H)$ of automorphic forms on $H$, and it is self-adjoint for the Petersson inner product on $\CA(H)$, i.e.
$$\pair{z' * f_1,f_2}_{Pet}=\pair{f_1, z' * f_2}_{Pet}$$
for $f_1, f_2\in \CA(H)$, where
$$\pair{f_1,f_2}_{Pet}=\int_{H(F)\bs H(\BA)}f_1(h) f_2(h)dh.$$
Moreover, we have
$$ z' * E(h,s)=P_z(s)E(h,s).$$

The Eisenstein series $E(h,s)$ has  a simple pole at $s=\rho_r=r/2$. Taking note of the definition of $P_z(s)$, we see that the regularized theta integral $\CE^{(n,r)}(g,s;\vphi)$ has a double pole at $s=\rho_r$ if $r\leq n \leq 2r-1$, and it has a simple pole at $s=\rho_r$ if $n\geq 2r$.\\

We write the Laurent expansion of $\CE^{(n,r)}(g,s;\vphi)$ at $s=\rho_r=r/2$ as
$$\CE^{(n,r)}(g,s;\vphi)=\sum_{d\geq -2}^{\infty}B^{(n,r)}_d(\vphi)(g)(s-\rho_r)^d.$$
Note that each $B^{(n,r)}_d$ defines a map from $\CS(V^n(\BA))$ to $\CA(G)$, which is $G(\BA)$-intertwining. This assertion can be checked as follows. We have
$$\begin{aligned}
\CE^{(n,r)}(g,s;\vphi)
&=\frac{1}{P_z(s)}I^{(n,r)}(g,s;\omega(z)\vphi)\\
&=\frac{1}{\kappa \cdot P_z(s)}\int_{[H]}\theta(g,h;\omega(z)\vphi)E(h,s)dh\\
&=\int_{[H]}\theta(g,h;\omega(z)\vphi)\cdot \frac{E(h,s)}{\kappa \cdot P_z(s)}dh.
\end{aligned}$$
Write the Laurent expansion of $\frac{E(h,s)}{\kappa \cdot P_z(s)}$ at $s=\rho_r$ as
$$\frac{E(h,s)}{\kappa \cdot P_z(s)}=\sum_{d=-2}^{\infty}C_d(h)(s-\rho_r)^d.$$
Then
$$ B^{(n,r)}_d(\vphi)(g)=\int_{[H]}\theta(g,h;\omega(z)\vphi)C_d(h)dh.$$
The assertion then follows from
$$\theta(gg',h;\omega(z)\vphi)=\theta(g,h;\omega(z)\omega(g')\vphi)$$
for all $g,g'\in G(\BA)$.\\

\begin{lem}\label{Lem 3.4}
For the spherical Schwartz function $\vphi^0$, we have
$$\CE^{(n,r)}(g,s;\vphi^0)=|D_E|^{-r(s+n-\frac{r}{2})/2} \prod_{i=1}^{r} \frac{\xi_E(s+n-\frac{r}{2}-i+1)}{\xi_E(i)} E^{(n,r)}(g,s),$$
where $E^{(n,r)}(g,s)$ is the spherical Eisenstein series.
\end{lem}
\begin{proof}
Note that
$$\CE^{(n,r)}(g,s;\vphi^0)=\sum_{\gamma \in P_r(F)\bs G(F)}f^{(n,r)}(\gamma g,s;\vphi^0_K).$$
But $\vphi^0_K=\vphi^0$, so $f^{(n,r)}(-,s;\vphi^0_K)=f^{(n,r)}(-,s;\vphi^0)$ is a spherical section in $I^{n,r}(s)$ for $\Re(s)\gg 0$. By the definition of $\vphi^0$, we see from \cite{GI} Lemma 7.3 that
$$f^{(n,r)}(1,s;\vphi^0)=|D_E|^{-r(s+n-\frac{r}{2})/2} \prod_{i=1}^{r} \frac{\xi_E(s+n-\frac{r}{2}-i+1)}{\xi_E(i)}.$$
The desired result then follows.
\end{proof}
{\bf Remark.} Actually, we should have
$$\CE^{(n,r)}(g,s;\vphi^0)=\frac{c_K}{\kappa}|D_E|^{-r(s+n-\frac{r}{2})/2} \prod_{i=1}^{r} \frac{\xi_E(s+n-\frac{r}{2}-i+1)}{\xi_E(i)} E^{(n,r)}(g,s).$$
But as we shall show later that $c_K=\kappa$ (see the remark after Proposition \ref{Prop 5.2}), we omit $c_K/\kappa$ in the above identity.

\subsection{The regularized Siegel-Weil formula (first term identity)}
In this section, we state the first term identity of the regularized Siegel-Weil formula in the range $n\geq 2r$. See \cite{Ich} Theorems 4.1 and 4.2, and \cite{Yam} Theorem 2.2.

\begin{thm}\label{Thm 3.5}
Suppose $\vphi \in S(V^n(\BA))$.

(i) Assume that $n\geq 2r+1$. Then there is a standard section $\Phi \in I^{n,n}(s)$ such that
$$ B^{(n,r)}_{-1}(\vphi)=A^{(n,n-r)}_{-1}(\Phi).$$
Moreover, if $\vphi$ is spherical, then $\Phi$ can be chosen to be a spherical section.

(ii) Assume that $n=2r$. Then
$$ B^{(2r,r)}_{-1}(\vphi)=\frac{1}{2} A^{(2r,r)}_0(\vphi).$$
\end{thm}
\begin{proof}
(i) \cite{Yam} Theorem 2.2 says that
$$E^{(n,n)}(g,\rho_{n,r};\Phi_{\vphi})=2 I(g,\vphi),$$
where
$$ I(g,\vphi)=c_z^{-1} I(g,\omega(z)\vphi)=\Res_{s=\frac{r}{2}} \CE^{(n,r)}(g,s;\vphi)$$
in the notation of \cite{Ich} p.251.

Recall the functional equation
$$ E^{(n,n)}(g,s;\Phi_{\vphi})=\frac{a(s)}{b(s)}E^{(n,n)}(g,-s;M^0(s)\Phi_{\vphi}),$$
where
$$\begin{aligned}
&M^0(s)=\frac{b(s)}{a(s)}M(s),\\
&a(s)=\prod_{i=1}^n \xi_F(2s-n+i,\eta^{i-1}),\\
&b(s)=\prod_{i=1}^n \xi_F(2s+n+1-i,\eta^{i-1}).
\end{aligned}$$
Since $M^0(s)\Phi_{\vphi}$ is a holomorphic section,
$E^{(n,n)}(g,-s;M^0(s)\Phi_{\vphi})$ has a simple pole at $s=\rho_{n,r}$.
Note that $a(s)$ is holomorphic and non-zero at $s=\rho_{n,r}=(2r-n)/2$, and $b(s)$ has a simple pole at $s=\rho_{n,r}$.

So $E^{(n,n)}(g,s;\Phi_{\vphi})$ is holomorphic at $s=\rho_{n,r}$, and
$$ E^{(n,n)}(g,\rho_{n,r};\Phi_{\vphi})= -\frac{a(\rho_{n,r})}{\Res_{s=\rho_{n,r}}b(s)}\Res_{s=-\rho_{n,r}}E^{(n,n)}(g,s;M^0(-s)\Phi_{\vphi}).$$
Note that
$$\Res_{s=\frac{r}{2}}\CE^{(n,r)}(g,s;\vphi)=I(g,\vphi).$$
If we take $\Phi$ to be the section in $I^{n,n}(s)$ given by
$$\Phi(g,s)=-\frac{1}{2}\frac{a(\rho_{n,r})}{\Res_{s=\rho_{n,r}}b(s)}M^0(-s)\Phi_{\vphi}(g,s),$$
then
$$\Res_{s=\frac{r}{2}}\CE^{(n,r)}(g,s;\vphi)=\Res_{s=\frac{n-2r}{2}}E^{(n,n)}(g,s;\Phi).$$

If $\vphi=\vphi^0$ is spherical, then $\Phi_{\vphi}$ is spherical, and $M^0(-s)\Phi_{\vphi}$ is spherical by the standard Gindikin-Karpelevich argument (see \cite{Yam} $\S 3$), and hence $\Phi$ is spherical.

(ii) This is just \cite{Ich} Theorem 4.2. Note that
$$\Res_{s=\frac{r}{2}}\CE^{(2r,r)}(g,s;\vphi)=I(g,\vphi).$$
\end{proof}

\section{Spherical second term identity}\label{Section 4}
In this section, we shall prove a spherical second term identity of the regularized Siegel-Weil formula in the range $r+1\leq n \leq 2r-1$, which relates the terms $A^{(n,r)}_0(\vphi^0)$ with $B^{(n,r)}_{-1}(\vphi^0)$ for the spherical Schwartz function $\vphi^0$.

\subsection{Spherical Eisenstein series}
Assume that $1\leq r\leq n$.

For each $s_0 \in \BC$, write the Laurent expansion of the spherical Eisenstein series $E^{(n,r)}(g,s)$ at $s=s_0$ as
$$E^{(n,r)}(g,s)=\sum_{d\gg -\infty}(s-s_0)^d \cdot E_d^{(n,r)}(g,s_0).$$

By Lemma \ref{Lem 3.4}, we see that the spherical Eisenstein series $E^{(n,r)}(g,s)$ has a double pole at $s=r/2$ if $r+1\leq n \leq 2r-1$, and it has a simple pole at $s=r/2$ if $n\geq 2r$. Henceforth, we call $n\geq 2r$ {\sl the first term range}, $r+1\leq n \leq 2r-1$ {\sl the second term range}, and we call $n=2r$ {\sl the boundary case}.

Let $Q=P_1$ be the parabolic subgroup of $G=U(\CW)=U(n,n)$ which stabilizes the line $\{(0_n,a,0_{n-1}|a\in E\}$ in $\CW$. Then $Q=M_1 N_1$, where the Levi factor $M_1\cong R_{E/F}\BG_m \times U(n-1,n-1)$. Consider the constant term $E^{(n,r)}_Q$ of the spherical Eisenstein series $E^{(n,r)}(g,s)$ along $Q$, which is defined by
$$E_{Q}^{(n,r)}(g,s)=\int_{N_1(F)\bs N_1(\BA)}E^{(n,r)}(ug,s)du,$$
where $N_1$ is the unipotent radical of $Q$, and the measure $du$ is such that $\vol(N_1(F)\bs N_1(\BA))=1$.
Its restriction to $M_1\cong R_{E/F}GL_1 \times U(n-1,n-1)$ can be computed as follows.

\begin{prop}\label{Prop 4.1}
Suppose $n\geq 2$ and $1\leq r \leq n$. The constant term $E^{(n,r)}_Q$ of the spherical Eisenstein series $E^{(n,r)}(s)$, as an automorphic form on $\BA_E^{\times}\times U(n-1,n-1)(\BA)$, can be expressed as follows:
\begin{itemize}
\item If $1<r<n$, then
$$\begin{aligned}
E^{(n,r)}_{Q}((a,g),s)
&=\|a\|^{s+n-\frac{r}{2}}E^{(n-1,r-1)}\left(g,s+\frac{1}{2}\right)\\
& \quad +\|a\|^{r} E^{(n-1,r)}(g,s)F^{(n,r)}(s)\\
& \quad +\|a\|^{-s+n-\frac{r}{2}}  E^{(n-1,r-1)}\left(g,s-\frac{1}{2}\right)G^{(n,r)}(s),
\end{aligned}$$
where
$$\begin{aligned}
& F^{(n,r)}(s)=\frac{\xi_E(s+n-\frac{3r}{2})}{\xi_E(s+n-\frac{r}{2})},\\
&G^{(n,r)}(s)=\frac{\xi_E(2s)}{\xi_E(2s+r-1)} \frac{\xi_F(2s+r-1)}{\xi_F(2s+r)} \frac{\xi_E(s-n+\frac{3r}{2})}{\xi_E(s+n-\frac{r}{2})}.
\end{aligned}$$

\item If $r=n$, then
$$\begin{aligned}
E^{(n,n)}_{Q}((a,g),s)
&=\|a\|^{s+\frac{n}{2}}E^{(n-1,n-1)}\left(g,s+\frac{1}{2}\right)\\
&\quad +\|a\|^{-s+\frac{n}{2}}E^{(n-1,n-1)}\left(g,s-\frac{1}{2}\right)H^{(n)}(s).
\end{aligned}$$
where
$$H^{(n)}(s)=\frac{\xi_E(2s)}{\xi_E(2s+n-1)}\frac{\xi_F(2s+n-1)}{\xi_F(2s+n)}.$$

\item If $r=1$, then
$$\begin{aligned}
&E^{(n,1)}_{Q}((a,g),s)=\|a\|^{s+n-\frac{1}{2}}+\|a\|  E^{(n-1,1)}(g,s)\frac{\xi_E(s+n-\frac{3}{2})}{\xi_E(s+n-\frac{1}{2})}\\
&\quad \quad \quad \quad \quad \quad \quad \quad +\|a\|^{-s+n-\frac{1}{2}}\frac{\xi_F(2s)}{\xi_F(2s+1)}
\frac{\xi_E(s-n+\frac{3}{2})}{\xi_E(s+n-\frac{1}{2})}.\\
\end{aligned}$$

\end{itemize}
Here $\xi_F(s)$ (resp. $\xi_E(s)$) is the normalized complete zeta function of $F$ (resp. of $E$), $(a,g)\in \BA_E^{\times}\times U(n-1,n-1)(\BA)\cong M_1(\BA)$, and $\|a\|$ is the absolute value of $a\in \BA_E^{\times}$.
\end{prop}
\begin{proof}
This can be proved by the same method as in \cite{J1} Chapter 3. See the Appendix (Section \ref{Appendix}) or \cite{X2} for the detailed proof.
\end{proof}

For $s_0\in \BC$, we write the Laurent expansion of $F^{(n,r)}(s)$ at $s=s_0$ as
$$F^{(n,r)}(s)=\sum_{d\gg -\infty}(s-s_0)^d F^{(n,r)}_d(s_0),$$
and similarly for $G^{(n,r)}(s)$ and $H^{(n)}(s)$.

For $E^{(n,r)}(g,s)$ and $\|a\|$ we simply write $E^{(n,r)}(s)$ and $\|~\|$ respectively, and we write the Laurent expansion of $E^{(n,r)}(s)$ at $s=s_0$ as
$$E^{(n,r)}(s)=\sum_{d\gg -\infty}(s-s_0)^d E^{(n,r)}_d(s_0).$$

\subsection{A lemma}
We record the following lemma which will be frequently used in our later computations. This is an analog of \cite{GT} Lemma 4.2.

\begin{lem}\label{Lem 4.2}
Let $E_1,\ldots, E_k$ and $F_1,\ldots,F_{\ell}$ be automorphic forms on $U(n,n)(\BA)$. Let $r_1,\ldots, r_k$ and $s_1,\ldots, s_{\ell}$ be real numbers such that all the $r_i$'s are distinct and all the $s_j$'s are distinct but some of the $r_i$'s might be the same as some of the $s_j$'s. If
$$\sum_{i=1}^k \|~\|^{r_i}E_i +\sum_{j=1}^{\ell}\|~\|^{s_j}\log\|~\|F_j=0$$
as an automorphic form on $\BA_E^{\times}\times U(n,n)(\BA)$, then all the $E_i$'s and $F_j$'s are zero.
\end{lem}
\begin{proof}
Fix an embedding $\BR^{+}\subset \BA_E^{\times}$ by choosing one of the archimedean places of $E$. For a fixed $g\in U(n,n)(\BA)$, we have
$$\sum_{i=1}^k \|~\|^{r_i}E_i(g) +\sum_{j=1}^{\ell}\|~\|^{s_j}\log\|~\|F_j(g)=0$$
as a function on $\BA^{\times}_E$, and by restriction to $\BR^{+}$, we can regard it as a functions on $\BR^{+}$. But it is known that the functions $\|~\|^{r_i}$ and $\|~\|^{s_j}\log\|~\|$, as functions on $\BR^{+}$, are linearly independent over $\BC$. Thus all the $E_i(g)$ and $F_j(g)$ are zero. Since the choice of $g$ is arbitrary, all the $E_i$'s and $F_j$'s are zero.
\end{proof}

\subsection{Spherical first term identity for 1st term range}
In this section, we study the spherical first term identity of the regularized Siegel-Weil formula for $n\geq 2r$.

\begin{prop}\label{Prop 4.3}
[Spherical first term identity for 1st term range]
Assume that $n\geq 2r+1$. Then there exists a non-zero constant $c_{n,r}$ such that
$$E_{-1}^{(n,r)}\left(\frac{r}{2}\right)=c_{n,r}E_{-1}^{(n,n)}\left(\frac{n-2r}{2}\right),$$
where $c_{n,r}$ is given explicitly by
$$c_{n,r}=\prod_{i=1}^r \frac{\xi_F(i,\eta^{i-1})}{\xi_F(r+i,\eta^{r+i})} \times \prod_{i=1}^r \frac{\xi_E(n-2r+i)}{\xi_E(n-i+1)} \times \prod_{i=0}^{n-r-1}\frac{\xi_F(2n-2r-i,\eta^i)}{\xi_F(i+1,\eta^i)}.$$
\end{prop}
\begin{proof}
The existence of $c_{n,r}$ is immediate from Theorem \ref{Thm 3.5}(i). It remains to compute $c_{n,r}$ explicitly. Note that $n\geq 3$ by assumption. First we compute $c_{n,1}$.
We have
$$E^{(n,1)}_{-1,Q}\left(\frac{1}{2}\right)=c_{n,1}E^{(n,n)}_{-1,Q}\left(\frac{n}{2}-1\right).$$
By Proposition \ref{Prop 4.1}, we obtain
$$\begin{aligned}
&\|~\|E_{-1}^{(n-1,1)}\left(\frac{1}{2}\right)\frac{\xi_E(n-1)}{\xi_E(n)}+\|~\|^{n-1}\frac{\xi_F(1)\xi_E(2-n)}{2\xi_F(2)\xi_E(n)}\\
&=c_{n,1}\left[ \|~\|^{n-1}E^{(n-1,n-1)}_{-1}\left(\frac{n-1}{2}\right)+\|~\| \cdot \Res_{s=\frac{n}{2}-1}E^{(n-1,n-1)}\left(s-\frac{1}{2}\right)H^{(n)}(s)\right].
\end{aligned}$$

Now Lemma \ref{Lem 4.2} allows us to extract all the terms containing $\|~\|^{n-1}$, and we obtain
$$\frac{\xi_F(1)\xi_E(2-n)}{2\xi_F(2)\xi_E(n)}=c_{n,1} E^{(n-1,n-1)}_{-1}\left(\frac{n-1}{2}\right).$$
This implies that $E^{(n-1,n-1)}_{-1}\left(\frac{n-1}{2}\right)$, as an automorphic form on $\BA_E^{\times}\times U(n-2,n-2)(\BA)$, is a constant function. Indeed,
$$E^{(n,n)}_{-1}\left(\frac{n}{2}\right)=\frac{1}{2}\prod_{i=0}^{n-1}\frac{\xi_F(i+1,\eta^i)}{\xi_F(2n-i,\eta^i)}.$$
See \cite{Ich} p.267.

Thus
$$c_{n,1}=\frac{\xi_E(n-1)}{\xi_E(n)}\frac{\xi_F(1)}{\xi_F(2)}\prod_{i=0}^{n-2}\frac{\xi_F(2n-2-i,\eta^i)}{\xi_F(i+1,\eta^i)}.$$

Now we compute $c_{n,r}$ for $r\geq 2$.
We have
$$E^{(n,r)}_{-1,Q}\left(\frac{r}{2}\right)=c_{n,r}E^{(n,n)}_{-1,Q}\left(\frac{n}{2}-r \right).$$
By Proposition \ref{Prop 4.1}, we obtain
$$\begin{aligned}
&\|~\|^n E_{-1}^{(n-1,r-1)}\left(\frac{r+1}{2}\right)+\|~\|^{r} \Res_{s=\frac{r}{2}}\left\{E^{(n-1,r)}(s)F^{(n,r)}(s)\right\}\\
&\quad \quad +\|~\|^{n-r}\Res_{s=\frac{r}{2}}\left\{E^{(n-1,r-1)}\left(s-\frac{1}{2}\right)G^{(n,r)}(s)\right\}\\
&=c_{n,r}\left[\|~\|^{n-r}E_{-1}^{(n-1,n-1)}\left(\frac{n+1}{2}-r\right)+\|~\|^r \Res_{s=\frac{n}{2}-r}\left\{E^{(n-1,n-1)}\left(s-\frac{1}{2}\right)H^{(n)}(s)\right\}\right].
\end{aligned}$$
Note that $n-r\neq n$ and $n-r\neq r$. By Lemma \ref{Lem 4.2} we can extract the terms containing $\|~\|^{n-r}$ and obtain
$$\Res_{s=\frac{r}{2}}\left\{E^{(n-1,r-1)}\left(s-\frac{1}{2}\right)G^{(n,r)}(s)\right\}=c_{n,r} E_{-1}^{(n-1,n-1)}\left(\frac{n+1}{2}-r\right),$$
or
$$E^{(n-1,r-1)}_{-1}\left(\frac{r-1}{2}\right)G^{(n,r)}\left(\frac{r}{2}\right)=c_{n,r} E_{-1}^{(n-1,n-1)}\left(\frac{n+1}{2}-r \right).$$
But
$$E^{(n-1,r-1)}_{-1}\left(\frac{r-1}{2}\right)=c_{n-1,r-1}E_{-1}^{(n-1,n-1)}\left(\frac{n+1}{2}-r \right),$$
so we obtain
$$\begin{aligned}
c_{n,r}
&=c_{n-1,r-1} G^{(n,r)}\left(\frac{r}{2}\right)\\
&=\frac{\xi_F(r)\xi_F(r,\eta)}{\xi_F(2r-1,\eta)\xi_F(2r)}\frac{\xi_E(n-2r+1)}{\xi_E(n)}c_{n-1,r-1}.
\end{aligned}$$
Thus
$$c_{n,r}=\prod_{i=1}^r \frac{\xi_F(i,\eta^{i-1})}{\xi_F(r+i,\eta^{r+i})} \times \prod_{i=1}^r \frac{\xi_E(n-2r+i)}{\xi_E(n-i+1)} \times \prod_{i=0}^{n-r-1}\frac{\xi_F(2n-2r-i,\eta^i)}{\xi_F(i+1,\eta^i)}.$$

\end{proof}

Before proceeding further, we record the following auxiliary lemma about the poles of the spherical Eisenstein series.
\begin{lem}\label{Lem 4.4}
(i) Suppose $n\geq r$. Then $E^{(n,r)}_{-2}(s_0)=0$ for $s_0\geq (r+1)/2$.

(ii) Suppose $n\geq 2r$. Then $E^{(n,r)}_{-2}\left(\frac{r}{2}\right)=0$.

(iii) Suppose $n\geq 2r+1$. Then $E^{(n,r)}_{-1}\left(\frac{r+2}{2}\right)=0$.

\end{lem}
\begin{proof}
By Lemma \ref{Lem 3.4}, we have
$$\CE^{(n,r)}(g,s)=|D_E|^{-r(s+n-\frac{r}{2})/2} \prod_{i=1}^r \frac{\xi_E(s+n-\frac{r}{2}-i+1)}{\xi_E(i)} \cdot E^{(n,r)}(g,s).$$

Recall the spherical regularized theta integral is given by
$$\CE^{(n,r)}(g,s)=\frac{1}{\kappa \cdot P_z(s)}\int_{[H]} \theta(g,h;\omega(z)\vphi^0)E(h,s)dh,$$
where
$$P_z(s)=\prod_{j=1}^r(C_{s-\frac{r+1}{2}+j}-C_{n-r+\frac{1}{2}}),$$
and $C_s=q^s+q^{-s}$.

The spherical Eisenstein series $E(h,s)$ is holomorphic in $\Re(s)>r/2$ and has a simple pole at $s=r/2$.

(i) Suppose $s_0\geq (r+1)/2$. Then $P_z(s)$ has at most a simple zero at $s_0$, and thus $\CE^{(n,r)}(s)$ has at most a simple pole at $s_0$. Since $\prod_{i=1}^r \xi_E(s+n-\frac{r}{2}-i+1)$ is holomorphic and non-zero at $s_0$, we see that $E^{(n,r)}(s)$ has at most a simple pole at $s_0$, and hence $E^{(n,r)}_{-2}(s_0)=0$.

(ii) If $n\geq 2r$, then $P_z(s)$ is non-zero at $s=r/2$, but $E(h,s)$ has a simple pole at $s=r/2$, thus the regularized theta integral $\CE^{(n,r)}(s)$ has a simple pole at $s=r/2$. Since $\prod_{i=1}^r \xi_E(s+n-\frac{r}{2}-i+1)$ is holomorphic and non-zero at $s=r/2$, we see that the spherical Eisenstein series $E^{(n,r)}(s)$ has a simple pole at $s=r/2$ if $n\geq 2r$, and hence $E^{(n,r)}_{-2}(\frac{r}{2})=0$.

(iii) In this case, $P_z(s)$ is nonzero at $s=(r+2)/2$, and $\prod_{i=1}^r \xi_E(s+n-\frac{r}{2}-j+1)$ is holomorphic and non-zero at $s=(r+2)/2$, so $E^{(2r+1,r)}(s)$ is holomorphic at $s=(r+2)/2$, and hence $E^{(2r+1,r)}_{-1}(\frac{r+2}{2})=0$.
\end{proof}

Now we consider the spherical first term identity on the boundary.
\begin{prop}\label{Prop 4.5}
[Spherical first term identity on the boundary]
Assume that $n=2r$. Then there is a constant $c_r$ such that
$$E_{-1}^{(2r,r)}\left(\frac{r}{2}\right)=c_r E_0^{(2r,2r)}(0),$$
where
$$c_r=\frac{1}{2}\prod_{i=1}^r \frac{\xi_E(i)}{\xi_E(r+i)}.$$
\end{prop}
\begin{proof}
The existence of $c_r$ is immediate from Theorem \ref{Thm 3.5} (ii). Now we compute $c_r$ explicitly.
First assume that $r>1$. By Proposition \ref{Prop 4.1}, we have
$$\begin{aligned}
E^{(2r+1,r)}_{-1,Q}(\frac{r}{2})
&=\|~\|^{2r+1}E^{(2r,r-1)}_{-1}\left(\frac{r+1}{2}\right)\\
&\quad +\|~\|^r \Res_{s=\frac{r}{2}}\left\{E^{(2r,r)}(s)F^{(2r+1,r)}(s)\right\}\\
&\quad +\|~\|^{r+1}\Res_{s=\frac{r}{2}}\left\{E^{(2r,r-1)}\left(s-\frac{1}{2}\right)G^{(2r+1,r)}(s)\right\}.
\end{aligned}$$

On the other hand, we have, by Proposition \ref{Prop 4.3} and Proposition \ref{Prop 4.1}, that
$$\begin{aligned}
E^{(2r+1,r)}_{-1,Q}\left(\frac{r}{2}\right)
&=c_{2r+1,r}E^{(2r+1,2r+1)}_{-1,Q}\left(\frac{1}{2}\right)\\
&=c_{2r+1,r} \left[\|~\|^{r+1}E_{-1}^{2r,2r}(1)+\|~\|^r \Res_{s=\frac{1}{2}}\left\{E^{(2r,2r)}\left(s-\frac{1}{2}\right)H^{(2r+1)}(s)\right\} \right].
\end{aligned}$$

Comparing the terms containing $\|~\|^r$, we obtain
$$\Res_{s=\frac{r}{2}}\left\{E^{(2r,r)}(s)F^{(2r+1,r)}(s)\right\}=c_{2r+1,r}\Res_{s=\frac{1}{2}}\left\{ E^{(2r,2r)}\left(s-\frac{1}{2} \right)H^{(2r+1)}(s)\right\}.$$

Note that
$$\Res_{s=\frac{r}{2}}\left\{E^{(2r,r)}(s)F^{(2r+1,r)}(s)\right\}=E^{(2r,r)}_{-1}\left(\frac{r}{2} \right)F^{(2r+1,r)}\left(\frac{r}{2} \right),$$
and
$$\Res_{s=\frac{1}{2}}\left\{E^{(2r,2r)}\left(s-\frac{1}{2}\right)H^{(2r+1)}(s)\right \}=E^{(2r,2r)}_{0}(0)H^{(2r+1)}_{-1}\left(\frac{1}{2} \right).$$

So we have
$$E^{(2r,r)}_{-1}\left(\frac{r}{2}\right)F^{(2r+1,r)}\left(\frac{r}{2}\right)=c_{2r+1,r}E^{(2r,2r)}_{0}(0)H^{(2r+1)}_{-1} \left(\frac{1}{2}\right),$$
namely
$$c_r F^{(2r+1,r)}\left(\frac{r}{2}\right)=c_{2r+1,r}H^{(2r+1)}_{-1}\left(\frac{1}{2}\right).$$

Since
$$F^{(2r+1,r)}\left(\frac{r}{2}\right)=\frac{\xi_E(r+1)}{\xi_E(2r+1)},$$
and
$$H^{(2r+1)}_{-1}\left(\frac{1}{2}\right)=\frac{1}{2}\frac{\xi_E(1)\xi_F(2r+1)}{\xi_E(2r+1)\xi_F(2r+2)},$$
we have
$$c_r=\frac{c_{2r+1,r}}{2}\frac{\xi_E(1)\xi_F(2r+1)}{\xi_E(r+1)\xi_F(2r+2)}.$$

But
$$c_{2r+1,r}=\frac{\xi_F(2r+2)}{\xi_F(2r+1)}\prod_{i=2}^r \frac{\xi_E(i)}{\xi_E(r+i)},$$
so
$$c_r=\frac{1}{2}\prod_{i=1}^r \frac{\xi_E(i)}{\xi_E(r+i)}.$$

Now we consider the case $r=1$.

We have
$$E^{(3,1)}_{-1,Q}\left(\frac{1}{2}\right)=\|~\|E^{(2,1)}_{-1}\left(\frac{1}{2}\right)\frac{\xi_E(2)}{\xi_E(3)}+\text{other terms}.$$
On the other hand,
$$E^{(3,1)}_{-1,Q}\left(\frac{1}{2}\right)=c_{3,1}E^{(3,3)}_{-1,Q}\left(\frac{1}{2}\right),$$
where
$$c_{3,1}=\frac{\xi_F(4)}{\xi_F(3)}.$$
But
$$E^{(3,3)}_{-1,Q}\left(\frac{1}{2}\right)=\|~\|E^{(2,2)}_0(0)H^{(3)}_{-1}\left(\frac{1}{2}\right)+\text{other terms},$$
where
$$H^{(3)}_{-1}\left(\frac{1}{2}\right)=\frac{1}{2}\frac{\xi_E(1)}{\xi_E(3)}\frac{\xi_F(3)}{\xi_F(4)}.$$
Comparing the terms containing $\|~\|$ in the above two expressions of $E^{(3,1)}_{-1,Q}(\frac{1}{2})$, we obtain
$$E^{(2,1)}_{-1}\left(\frac{1}{2}\right)\frac{\xi_E(2)}{\xi_E(3)}=c_{3,1}E^{(2,2)}_0(0)H^{(3)}_{-1}\left(\frac{1}{2}\right),$$
i.e.
$$E^{(2,1)}_{-1}\left(\frac{1}{2}\right)=\frac{1}{2}\frac{\xi_E(1)}{\xi_E(2)} E^{(2,2)}_0(0)=c_1 E^{(2,2)}_0(0),$$
where
$$c_1=\frac{1}{2}\frac{\xi_E(1)}{\xi_E(2)}.$$
\end{proof}

{\bf Remark}. The explicit value of $c_r$ can also be obtained by using Lemma \ref{Lem 3.4} and Theorem \ref{Thm 3.5}(ii).

\subsection{Idea of the proof of the first and second term identities for 2nd term range}
In the following, we give a brief description of the proof of the spherical first and second term identities for 2nd term range. The idea of the proof uses a clever inductive method which is due to Gan-Takeda \cite{GT}.

First we consider the spherical first term identity for 2nd term range.

Consider the case $n=2r-1$.
It follows from Lemma \ref{Lem 4.4} (ii) that
$$ E^{(2r,r)}_{-2,Q}\left(\frac{r}{2}\right)=0.$$
We use Proposition \ref{Prop 4.1} to compute the term $E^{(2r,r)}_{-2,Q}(\frac{r}{2})$, and then we obtain the spherical first term identity for $n=2r-1$ immediately:
$$ E^{(2r-1,r)}_{-2}\left(\frac{r}{2}\right)=d_{2r-1,r} E^{(2r-1,2r-1)}_{-1}\left(\frac{1}{2}\right).$$

In general, for $r+1\leq n \leq 2r-1$, let $j=2r-1-n$. Then $0\leq j \leq r-2$. We use induction on $j$. The case $j=0$ is the spherical first term identity for $n=2r-1$. For $j=1$, if we take the constant term along $Q$ of the identity for $j=0$, then we obtain
$$ E^{(2r-1,r)}_{-2,Q}\left(\frac{r}{2}\right)=d_{2r-1,r} E^{(2r-1,2r-1)}_{-1,Q}\left(\frac{1}{2}\right).$$
We compute the terms $E^{(2r-1,r)}_{-2,Q}(\frac{r}{2})$ and $E^{(2r-1,2r-1)}_{-1,Q}(\frac{1}{2})$ by using Proposition \ref{Prop 4.1}. Then we can use Lemma \ref{Lem 4.2} to obtain the spherical first term identity for $n=2r-2$.

The case for general $j$ is similar: take constant term of the identity for $j-1$, and use Proposition \ref{Prop 4.1} and Lemma \ref{Lem 4.2}. See Proposition \ref{Prop 4.6} for the full statement.\\

Next we consider the spherical second term identity for 2nd term range.

First consider the case $n=2r-1$.
Taking the constant term along $Q$ of the spherical first term identity for $n=2r$, we obtain
$$E_{-1,Q}^{(2r,r)}\left(\frac{r}{2}\right)=c_r E_{0,Q}^{(2r,2r)}(0).$$
We compute the terms $E_{-1,Q}^{(2r,r)}(\frac{r}{2})$ and $E_{0,Q}^{(2r,2r)}(0)$ by Proposition \ref{Prop 4.1}. Then we can use Lemma \ref{Lem 4.2} to obtain
$$\begin{aligned}
& E^{(2r-1,r)}_{-2}\left(\frac{r}{2}\right)\cdot a_1 +E^{(2r-1,r)}_{-1}\left(\frac{r}{2}\right) \cdot a_2\\
& \quad \quad +E^{(2r-1,r-1)}_{-1}\left(\frac{r-1}{2}\right) \cdot a_3 + E^{(2r-1,r-1)}_0\left(\frac{r-1}{2}\right) \cdot a_4\\
&=E^{(2r-1,2r-1)}_0\left(\frac{1}{2}\right) \cdot a_5 +E^{(2r-1,2r-1)}_{-1}\left(\frac{1}{2}\right)\cdot a_6,
\end{aligned}$$
where each $a_i$ is an explicit non-zero constant.
Using the following first term identities
$$\begin{aligned}
&E^{(2r-1,r-1)}_{-1}\left(\frac{r-1}{2}\right)=c_{2r-1,r-1}E^{(2r-1,2r-1)}_{-1}\left(\frac{1}{2}\right),\\
&E^{(2r-1,r)}_{-2}\left(\frac{r}{2}\right)=d_{2r-1,r}E^{(2r-1,2r-1)}_{-1}\left(\frac{1}{2}\right),
\end{aligned}$$
we can simplify the above expression and obtain the spherical second term identity for $n=2r-1$ (see Proposition \ref{Prop 4.7}).

In general, for $r+1\leq n \leq 2r-1$, let $j=2r-1-n$. Then $0\leq j \leq r-2$. We use induction on $j$ to obtain the spherical second term identity. The case $j=0$ is the spherical second term identity for $n=2r-1$. For $j=1$, if we take the constant term along $Q$ of the identity for $j=0$, we obtain
$$\begin{aligned}
&E_{-1,Q}^{(2r-1,r)}\left(\frac{r}{2}\right) \cdot b_1 + E_{0,Q}^{(2r-1,r-1)}\left(\frac{r-1}{2}\right) \cdot b_2\\
&\quad \quad \quad  =E_{0,Q}^{(2r-1,2r-1)}\left(\frac{1}{2}\right) \cdot b_3+ E_{-1,Q}^{(2r-1,2r-1)}\left(\frac{1}{2}\right)\cdot b_4,
\end{aligned}$$
where each $b_i$ is an explicit non-zero constant.
Now we compute the four terms $E_{-1,Q}^{(2r-1,r)}(\frac{r}{2})$, $E_{0,Q}^{(2r-1,r-1)}$, $E_{0,Q}^{(2r-1,2r-1)}(\frac{1}{2})$, and $E_{-1,Q}^{(2r-1,2r-1)}(\frac{1}{2})$ by using Proposition \ref{Prop 4.1}. Then we can use Lemma \ref{Lem 4.2} to obtain the spherical second term identity for $j=1$.

The proof of the spherical second term identity for general $j$ is similar: take the constant term along $Q$ of the identity for $j-1$, simplify the resulting expression by using Proposition \ref{Prop 4.1}, and use Lemma \ref{Lem 4.2}. See Theorem \ref{Thm 4.8} for the full statement.

\subsection{Spherical first term identity for 2nd term range}
In this section, we study the spherical first term identity for 2nd term range, which will be needed in the derivation of the spherical second term identity.
\begin{prop}\label{Prop 4.6}
[Spherical first term identity for 2nd term range]
Assume that $r+1\leq n \leq 2r-1$. Then
$$E^{(n,r)}_{-2}\left(\frac{r}{2}\right)=d_{n,r}E^{(n,n)}_{-1}\left(\frac{2r-n}{2}\right),$$
where
$$\begin{aligned}
&d_{2r-1,r}=c_{2r-1,r-1}\frac{\xi_E(1)\xi_F(2r-1)}{\xi_E(2r-1)\xi_F(2r)},\\
&d_{n-1,r}=d_{n,r}\frac{\xi_E(n)}{\xi_E(n-r)}, \quad \text{for $r+2 \leq n\leq 2r-1$}.
\end{aligned}$$
\end{prop}
\begin{proof}
Note that $r\geq 2$ by assumption.

First we show the case $n=2r-1$.
We start with the identity
$$E^{(2r,r)}_{-2,Q}\left(\frac{r}{2}\right)=0, $$
which follows immediately from Lemma \ref{Lem 4.4}(ii).

By Proposition \ref{Prop 4.1}, we have
$$\begin{aligned}
&E^{(2r,r)}_{-2,Q}\left(\frac{r}{2}\right)=\|~\|^r E^{(2r-1,r)}_{-2}\left(\frac{r}{2}\right)F^{(2r,r)}\left(\frac{r}{2}\right)\\
&\quad \quad \quad \quad \quad +\|~\|^r E^{(2r-1,r-1)}_{-1}\left(\frac{r-1}{2}\right)G^{(2r,r)}_{-1}\left(\frac{r}{2}\right),
\end{aligned}$$
where we have used the fact that $E^{(2r-1,r-1)}_{-2}(\frac{r+1}{2})=0$ and $E^{(2r-1,r-1)}_{-2}(\frac{r-1}{2})=0$ (Lemma \ref{Lem 4.4} (i),(ii)).

Then we have
$$0= E^{(2r-1,r)}_{-2}\left(\frac{r}{2}\right)F^{(2r,r)}\left(\frac{r}{2}\right)+E^{(2r-1,r-1)}_{-1}\left(\frac{r-1}{2} \right)G^{(2r,r)}_{-1}\left(\frac{r}{2}\right).$$
So
$$\begin{aligned}
&\quad E^{(2r-1,r)}_{-2}\left(\frac{r}{2}\right)\\
&=-E^{(2r-1,r-1)}_{-1}\left(\frac{r-1}{2}\right)G^{(2r,r)}_{-1}\left(\frac{r}{2} \right) F^{(2r,r)}\left(\frac{r}{2}\right)^{-1}\\
&=-c_{2r-1,r-1}E^{(2r-1,2r-1)}_{-1}\left(\frac{1}{2}\right) G^{(2r,r)}_{-1}\left(\frac{r}{2} \right) F^{(2r,r)}\left(\frac{r}{2}\right)^{-1}.
\end{aligned}$$
Thus
$$E^{(2r-1,r)}_{-2}\left(\frac{r}{2}\right)=d_{2r-1,r}E^{(2r-1,2r-1)}_{-1}\left(\frac{1}{2} \right),$$
where
$$d_{2r-1,r}=-c_{2r-1,r-1}G^{(2r,r)}_{-1}\left(\frac{r}{2} \right) F^{(2r,r)}\left(\frac{r}{2} \right)^{-1}.$$
But
$$\begin{aligned}
&G^{(2r,r)}_{-1}\left(\frac{r}{2} \right)=\frac{\xi_E(r)\xi_F(2r-1)\xi_E(0)}{\xi_E(2r-1)\xi_F(2r)\xi_E(2r)},\\
&F^{(2r,r)}\left(\frac{r}{2} \right)=\frac{\xi_E(r)}{\xi_E(2r)}.
\end{aligned}$$
So
$$\begin{aligned}
d_{2r-1,r}
&=-c_{2r-1,r-1}\frac{\xi_E(0)\xi_F(2r-1)}{\xi_E(2r-1)\xi_F(2r)}\\
&=c_{2r-1,r-1}\frac{\xi_E(1)\xi_F(2r-1)}{\xi_E(2r-1)\xi_F(2r)}.
\end{aligned}$$

Next we show the case $r+1 \leq n\leq 2r-2$. Write
$$n=2r-1-j, \quad \text{where $1\leq j \leq r-2$}.$$
We use induction on $j$. First consider the base step $j=1$. Taking the constant term along $Q$ of the identity for $2r-1$, we obtain
$$E^{(2r-1,r)}_{-2,Q}\left(\frac{r}{2} \right)=d_{2r-1,r}E^{(2r-1,2r-1)}_{-1,Q}\left(\frac{1}{2} \right).$$
By Proposition \ref{Prop 4.1}, we have
$$\begin{aligned}
E^{(2r-1,r)}_{-2,Q}\left(\frac{r}{2} \right)
&=\|~\|^{2r-1}E^{(2r-2,r-1)}_{-2}\left(\frac{r+1}{2} \right)\\
&\quad +\|~\|^{r}E^{(2r-2,r)}_{-2}\left(\frac{r}{2} \right)F^{(2r-1,r)}\left(\frac{r}{2} \right)\\
&\quad +\|~\|^{r-1}E^{(2r-2,r-1)}_{-1}\left(\frac{r-1}{2} \right) G^{(2r-1,r)}_{-1}\left(\frac{r}{2} \right),
\end{aligned}$$
where we have used the fact that $E^{(2r-2,r-1)}_{-2}(\frac{r-1}{2})=0$ (Lemma \ref{Lem 4.4} (ii)).

But
$$\begin{aligned}
&E^{(2r-2,r-1)}_{-2}\left(\frac{r+1}{2} \right)=0 \quad(\text{Lemma \ref{Lem 4.4} (i)}),\\
&E^{(2r-2,r-1)}_{-1}\left(\frac{r-1}{2} \right)=c_{r-1}E^{(2r-2,2r-2)}_{0}(0) \quad \text{(Proposition \ref{Prop 4.5})}.
\end{aligned}$$
So we obtain
$$\begin{aligned}
E^{(2r-1,r)}_{-2,Q}\left(\frac{r}{2} \right)
&=\|~\|^{r}E^{(2r-2,r)}_{-2}\left(\frac{r}{2} \right)F^{(2r-1,r)}\left(\frac{r}{2} \right)\\
& \quad +\|~\|^{r-1}c_{r-1}E^{(2r-2,2r-2)}_{0}(0)G^{(2r-1,r)}_{-1}\left(\frac{r}{2} \right).
\end{aligned}$$
On the other hand, by Proposition \ref{Prop 4.1}, we have
$$\begin{aligned}
E^{(2r-1,2r-1)}_{-1,Q}\left(\frac{1}{2}\right)
&=\|~\|^r E^{(2r-2,2r-2)}_{-1}(1)\\
&\quad +\|~\|^{r-1} E^{(2r-1,2r-1)}_{0}(0)H^{(2r-1)}_{-1}\left(\frac{1}{2}\right).
\end{aligned}$$
Thus
$$\begin{aligned}
&\quad d_{2r-1,r}\left[\|~\|^r E^{(2r-2,2r-2)}_{-1}(1)+\|~\|^{r-1} E^{(2r-1,2r-1)}_{0}(0)H^{(2r-1)}_{-1}\left(\frac{1}{2} \right) \right]\\
&=\|~\|^{r}E^{(2r-2,r)}_{-2}\left(\frac{r}{2} \right)F^{(2r-1,r)}\left(\frac{r}{2} \right) +\|~\|^{r-1}c_{r-1}E^{(2r-2,2r-2)}_{0}(0)G^{(2r-1,r)}_{-1}\left(\frac{r}{2} \right).
\end{aligned}$$
By Lemma \ref{Lem 4.2}, we can extract the terms containing $\|~\|^r$ in the above identity, and we obtain
$$d_{2r-1,r}E^{(2r-2,2r-2)}_{-1}(1)=E^{(2r-2,r)}_{-2}\left(\frac{r}{2} \right)F^{(2r-1,r)}\left(\frac{r}{2} \right),$$
i.e.
$$E^{(2r-2,r)}_{-2}\left(\frac{r}{2} \right)=d_{2r-2,r}E^{(2r-2,2r-2)}_{-1}(1),$$
where
$$\begin{aligned}
d_{2r-2,r}
&=d_{2r-1,r}F^{(2r-1,r)}\left(\frac{r}{2} \right)^{-1}\\
&=d_{2r-1,r}\frac{\xi_E(2r-1)}{\xi_E(r-1)}.
\end{aligned}$$
This shows the base step $j=1$.

Now consider the induction step. Assume that $2r-j-2\geq r+1$ and the identity holds for $2r-j-1$ for some $j\geq 1$. We now show that the identity holds for $2r-j-2$. Taking the constant term along $Q$ of the identity for $2r-j-1$, we have
$$E^{(2r-1-j,r)}_{-2,Q}\left(\frac{r}{2} \right)=d_{2r-j-1,r}E^{(2r-j-1,2r-j-1)}_{-1,Q}\left(\frac{j+1}{2}\right).$$
Now we compute the two terms in the above identity.

For the term on the LHS, we have, by Proposition \ref{Prop 4.1}, that
$$\begin{aligned}
E^{(2r-1-j,r)}_{-2,Q}\left(\frac{r}{2} \right)
&=\|~\|^{2r-j-1}E^{(2r-j-2,r-1)}_{-2}\left(\frac{r+1}{2} \right)\\
&\quad +\|~\|^r E^{(2r-j-2,r)}_{-2}\left(\frac{r}{2} \right) F^{(2r-j-1,r)}\left(\frac{r}{2} \right)\\
&\quad +\|~\|^{r-j-1} E^{(2r-j-2,r-1)}_{-2}\left(\frac{r-1}{2} \right)G^{(2r-j-1,r)}\left(\frac{r}{2} \right).
\end{aligned}$$
But
$$E^{(2r-j-2,r-1)}_{-2}\left(\frac{r+1}{2} \right)=0 \quad (\text{Lemma \ref{Lem 4.4} (i)}).$$
So
$$\begin{aligned}
E^{(2r-1-j,r)}_{-2,Q}\left(\frac{r}{2} \right)
&=\|~\|^r E^{(2r-j-2,r)}_{-2}\left(\frac{r}{2} \right) F^{(2r-j-1,r)}\left(\frac{r}{2} \right)\\
&\quad +\|~\|^{r-j-1} E^{(2r-j-2,r-1)}_{-2}\left(\frac{r-1}{2} \right) G^{(2r-j-1,r)}\left(\frac{r}{2} \right).
\end{aligned}$$

On the other hand, by Proposition \ref{Prop 4.1}, we have
$$\begin{aligned}
E^{(2r-j-1,2r-j-1)}_{-1,Q}\left(\frac{j+1}{2}\right)
&=\|~\|^r E^{(2r-j-2,2r-j-2)}_{-1}\left(\frac{j+2}{2}\right)\\
&\quad +\|~\|^{r-j-1}E^{(2r-j-2,2r-j-2)}_{-1}\left(\frac{j}{2}\right)H^{(2r-j-1)}\left(\frac{j+1}{2}\right).
\end{aligned}$$

So
$$\begin{aligned}
&\quad \|~\|^r E^{(2r-j-2,r)}_{-2}\left(\frac{r}{2} \right) F^{(2r-j-1,r)}\left(\frac{r}{2} \right)\\
&\quad \quad +\|~\|^{r-j-1} E^{(2r-j-2,r-1)}_{-2}\left(\frac{r-1}{2} \right)G^{(2r-j-1,r)}\left(\frac{r}{2} \right)\\
&=d_{2r-j-1,r}\|~\|^r \left[E^{(2r-j-2,2r-j-2)}_{-1}\left(\frac{j+2}{2}\right)\right.\\
&\quad \quad \quad \quad \quad \quad \quad \quad \left.+\|~\|^{r-j-1}E^{(2r-j-2,2r-j-2)}_{-1}\left(\frac{j}{2}\right)H^{(2r-j-1)}\left(\frac{j+1}{2}\right)\right].
\end{aligned}$$
By Lemma \ref{Lem 4.2}, we can extract the terms containing $\|~\|^r$ in the above identity and obtain
$$E^{(2r-j-2,r)}_{-2}\left(\frac{r}{2} \right)F^{(2r-j-1,r)}\left(\frac{r}{2} \right)=d_{2r-j-1,r}E^{(2r-j-2,2r-j-2)}_{-1}\left(\frac{j+2}{2} \right),$$
i.e.
$$E^{(2r-j-2,r)}_{-2}\left(\frac{r}{2} \right)=d_{2r-j-2,r}E^{(2r-j-2,2r-j-2)}_{-1}\left(\frac{j+2}{2} \right),$$
where
$$ \begin{aligned}
d_{2r-j-2,r}
&=d_{2r-j-1,r}F^{(2r-j-1,r)}\left(\frac{r}{2} \right)^{-1}\\
&=d_{2r-j-1,r}\frac{\xi_E(2r-j-1)}{\xi_E(r-j-1)}.
\end{aligned}$$
This completes the proof.
\end{proof}

\subsection{Spherical second term identity for 2nd term range}
In this section, we will derive the spherical second term identity for 2nd term range. First we consider the case $n=2r-1$, which can be regarded as the first step of the induction.
\begin{prop}\label{Prop 4.7}
[Spherical second term identity for $n=2r-1$]
Suppose $r\geq 2$. Then
$$\begin{aligned}
&\quad E^{(2r-1,r)}_{-1}(\frac{r}{2})F^{(2r,r)}_0\left(\frac{r}{2} \right)+E^{(2r-1,r-1)}_{0}\left(\frac{r-1}{2} \right)G^{(2r,r)}_{-1}\left(\frac{r}{2} \right)\\
& =2 c_r E^{(2r-1,2r-1)}_0\left(\frac{1}{2} \right)+\gamma_0 E^{(2r-1,2r-1)}_{-1}\left(\frac{1}{2} \right),
\end{aligned}$$
where $\gamma_0$ is some constant which depends only on $r$.
\end{prop}
\begin{proof}
This is similar to the proof of \cite{GT} Proposition 4.7.

We start with the identity
$$E^{(2r,r)}_{-1,Q}\left(\frac{r}{2} \right)=c_r E^{(2r,2r)}_{0,Q}(0),$$
which is obtained by taking the constant term along $Q$ of the spherical first term identity on the boundary (Proposition \ref{Prop 4.5}).

By Proposition \ref{Prop 4.1}, we have
$$\begin{aligned}
E^{(2r,r)}_Q(s)
&=\|~\|^{s+\frac{3r}{2}}E^{(2r-1,r-1)}\left(s+\frac{1}{2} \right)\\
&\quad +\|~\|^{r}E^{(2r-1,r)}(s)F^{(2r,r)}(s)\\
&\quad +\|~\|^{-s+\frac{3r}{2}}E^{(2r-1,r-1)}\left(s-\frac{1}{2} \right)G^{(2r,r)}(s).
\end{aligned}$$
Note that
$$E^{(2r-1,r-1)}_{-1}\left(\frac{r+1}{2}\right)=0 \quad (\text{Lemma \ref{Lem 4.4} (iii)}).$$
So
$$\begin{aligned}
E^{(2r,r)}_{-1,Q}\left(\frac{r}{2} \right)
&=\|~\|^{r}E^{(2r-1,r)}_{-2}\left(\frac{r}{2} \right) F^{(2r,r)}_1\left(\frac{r}{2} \right)\\
&\quad +\|~\|^{r}E^{(2r-1,r)}_{-1}\left(\frac{r}{2} \right)F^{(2r,r)}_0\left(\frac{r}{2}\right)\\
&\quad +\|~\|^{r}E^{(2r-1,r-1)}_{-1}\left(\frac{r-1}{2}\right)G^{(2r,r)}_0\left(\frac{r}{2}\right)\\
&\quad +\|~\|^{r}E^{(2r-1,r-1)}_{0}\left(\frac{r-1}{2}\right)G^{(2r,r)}_{-1}\left(\frac{r}{2}\right)\\
&\quad -\|~\|^{r}\log \|~\| E^{(2r-1,r-1)}_{-1}\left(\frac{r-1}{2}\right)G^{(2r,r)}_{-1}\left(\frac{r}{2} \right),
\end{aligned}$$
where we have used that facts that $E^{(2r-1,r)}(s)$ has a double pole at $s=r/2$, $E^{(2r-1,r-1)}(s)$ has a simple pole at $s=(r-1)/2$, $F^{(2r,r)}(s)$ is holomorphic at $s=r/2$, and $G^{(2r,r)}(s)$ has a simple pole at $s=r/2$.

On the other hand, by Proposition \ref{Prop 4.1}, we have
$$\begin{aligned}
E^{(2r,2r)}_Q(s)
&=\|~\|^{s+r}E^{(2r-1,2r-1)}\left(s+\frac{1}{2}\right)\\
&\quad +\|~\|^{-s+r}E^{(2r-1,2r-1)}\left(s-\frac{1}{2}\right)H^{(2r)}(s).
\end{aligned}$$
So
$$\begin{aligned}
E^{(2r,2r)}_{0,Q}(0)
&=\|~\|^r E^{(2r-1,2r-1)}_0\left(\frac{1}{2} \right)+\|~\|^{r}\log \|~\| E^{(2r-1,2r-1)}_{-1}\left(\frac{1}{2} \right)\\
&\quad +\|~\|^{r} E^{(2r-1,2r-1)}_{0}\left(-\frac{1}{2} \right)H^{(2r)}_0(0)\\
&\quad +\|~\|^{r} E^{(2r-1,2r-1)}_{1}\left(-\frac{1}{2} \right)H^{(2r)}_{-1}(0)\\
&\quad -\|~\|^{r}\log \|~\| E^{(2r-1,2r-1)}_{0}\left(-\frac{1}{2} \right)H^{(2r)}_{-1}(0),
\end{aligned}$$
where we have used that facts that $E^{(2r-1,2r-1)}(s)$ has a simple pole at $s=1/2$ and is holomorphic at $s=-1/2$ (which we will show immediately), and $H^{(2r)}(s)$ has a simple pole at $s=0$.

Recall the functional equation
$$E^{(2r-1,2r-1)}(s)=\beta_{2r-1}(s)E^{(2r-1,2r-1)}(-s),$$
where
$$\beta_{n}(s)=\prod_{i=1}^n \frac{\xi_F(2s-n+i,\eta^{i-1})}{\xi_F(2s+n+1-i,\eta^{i-1})}.$$
Note that $\beta_{2r-1}(s)$ has a simple zero at $s=-1/2$, and $E^{(2r-1,2r-1)}(s)$ has a simple pole at $s=1/2$, so $E^{(2r-1,2r-1)}(s)$ is holomorphic at $s=-1/2$. Hence
$$\begin{aligned}
&E^{(2r-1,2r-1)}_0\left(-\frac{1}{2} \right)=-\beta_{2r-1,1}\left(-\frac{1}{2} \right)E^{(2r-1,2r-1)}_{-1}\left(\frac{1}{2} \right),\\
&E^{(2r-1,2r-1)}_1\left(-\frac{1}{2} \right)=\beta_{2r-1,1}\left(-\frac{1}{2} \right)E^{(2r-1,2r-1)}_{0}\left(\frac{1}{2} \right)-\beta_{2r-1,2}\left(-\frac{1}{2} \right)E^{(2r-1,2r-1)}_{-1}\left(\frac{1}{2} \right),
\end{aligned}$$
where we write the Laurent expansion of $\beta_n(s)$ at $s_0$ as
$$\beta_n(s)=\sum_{d\gg -\infty}(s-s_0)^d \beta_{n,d}(s_0).$$
So
$$\begin{aligned}
&E^{(2r,2r)}_{0,Q}(0)=\|~\|^r E^{(2r-1,2r-1)}_0\left(\frac{1}{2} \right)+\|~\|^{r}\log \|~\| E^{(2r-1,2r-1)}_{-1}\left(\frac{1}{2} \right)\\
&-\|~\|^{r}\beta_{2r-1,1}\left(-\frac{1}{2} \right)E^{(2r-1,2r-1)}_{-1}\left(\frac{1}{2} \right)H^{(2r)}_0(0)\\
& +\|~\|^{r} \left[\beta_{2r-1,1}\left(-\frac{1}{2} \right)E^{(2r-1,2r-1)}_{0}\left(\frac{1}{2} \right)-\beta_{2r-1,2}\left(-\frac{1}{2} \right)E^{(2r-1,2r-1)}_{-1}\left(\frac{1}{2} \right) \right]H^{(2r)}_{-1}(0)\\
&-\|~\|^{r}\log \|~\| E^{(2r-1,2r-1)}_{0}\left(-\frac{1}{2} \right)H^{(2r)}_{-1}(0).
\end{aligned}$$

By Lemma \ref{Lem 4.2}, we can extract the terms containing only $\|~\|^r$ in the identity
$$E^{(2r,r)}_{-1,Q}\left(\frac{r}{2} \right)=c_r E^{(2r,2r)}_{0,Q}(0).$$
Then we obtain
$$\begin{aligned}
&\quad E^{(2r-1,r)}_{-2}\left(\frac{r}{2} \right)F^{(2r,r)}_1\left(\frac{r}{2} \right)+E^{(2r-1,r)}_{-1}\left(\frac{r}{2} \right)F^{(2r,r)}_0\left(\frac{r}{2} \right)\\
&\quad +E^{(2r-1,r-1)}_{-1}\left(\frac{r-1}{2} \right)G^{(2r,r)}_0\left(\frac{r}{2} \right)+E^{(2r-1,r-1)}_{0}\left(\frac{r-1}{2} \right)G^{(2r,r)}_{-1}\left(\frac{r}{2} \right)\\
& =c_r \left\{ E^{(2r-1,2r-1)}_0\left(\frac{1}{2}\right)-\beta_{2r-1,1}\left(-\frac{1}{2} \right)E^{(2r-1,2r-1)}_{-1}\left(\frac{1}{2} \right)H^{(2r)}_0(0)\right.\\
&\quad \quad + \left. \left[\beta_{2r-1,1}\left(-\frac{1}{2} \right)E^{(2r-1,2r-1)}_{0}\left(\frac{1}{2} \right)-\beta_{2r-1,2}\left(-\frac{1}{2} \right) E^{(2r-1,2r-1)}_{-1}\left(\frac{1}{2} \right) \right]H^{(2r)}_{-1}(0)\right\}\\
&=c_r\left[1+\beta_{2r-1,1}\left(-\frac{1}{2} \right)H^{(2r)}_{-1}(0) \right]E^{(2r-1,2r-1)}_0\left(\frac{1}{2} \right)\\
&\quad +c_r\left[-\beta_{2r-1,1}\left(-\frac{1}{2} \right)H^{(2r)}_0(0)-\beta_{2r-1,2}\left(-\frac{1}{2} \right) H^{(2r)}_{-1}(0) \right]E^{(2r-1,2r-1)}_{-1}\left(\frac{1}{2} \right).
\end{aligned}$$

Using the first term identities
$$\begin{aligned}
&E^{(2r-1,r-1)}_{-1}\left(\frac{r-1}{2} \right)=c_{2r-1,r-1}E^{(2r-1,2r-1)}_{-1}\left(\frac{1}{2} \right),\\
&E^{(2r-1,r)}_{-2}\left(\frac{r}{2} \right)=d_{2r-1,r}E^{(2r-1,2r-1)}_{-1}\left(\frac{1}{2} \right),
\end{aligned}$$
we can write the above identity as
$$\begin{aligned}
&\quad E^{(2r-1,r)}_{-1}\left(\frac{r}{2} \right)F^{(2r,r)}_0\left(\frac{r}{2} \right)+E^{(2r-1,r-1)}_{0}\left(\frac{r-1}{2} \right)G^{(2r,r)}_{-1}\left(\frac{r}{2} \right)\\
&=c_r\left[1+\beta_{2r-1,1}\left(-\frac{1}{2} \right)H^{(2r)}_{-1}(0) \right]E^{(2r-1,2r-1)}_0\left(\frac{1}{2} \right)\\
&\quad +c_r \left[-\beta_{2r-1,1}\left(-\frac{1}{2} \right)H^{(2r)}_0(0)-\beta_{2r-1,2}\left(-\frac{1}{2} \right) H^{(2r)}_{-1}(0)\right.\\
&\quad \quad \left. -d_{2r-1,r}F^{(2r,r)}_1\left(\frac{r}{2} \right)-c_{2r-1,r-1}G^{(2r,r)}_0\left(\frac{r}{2}\right)\right] E^{(2r-1,2r-1)}_{-1}\left(\frac{1}{2} \right),
\end{aligned}$$
i.e.
$$\begin{aligned}
&\quad E^{(2r-1,r)}_{-1}\left(\frac{r}{2} \right)F^{(2r,r)}_0\left(\frac{r}{2} \right)+E^{(2r-1,r-1)}_{0}\left(\frac{r-1}{2} \right)G^{(2r,r)}_{-1}\left(\frac{r}{2} \right)\\
&=c_r\left[1+\beta_{2r-1,1}\left(-\frac{1}{2} \right)H^{(2r)}_{-1}(0) \right]E^{(2r-1,2r-1)}_0\left(\frac{1}{2} \right)+\gamma_0 E^{(2r-1,2r-1)}_{-1}\left(\frac{1}{2} \right),
\end{aligned}$$
where $\gamma_0$ is a constant.

Finally, we claim that
$$1+\beta_{2r-1,1}\left(-\frac{1}{2} \right)H^{(2r)}_{-1}(0)=2.$$

First, we have
$$H^{(2r)}_{-1}(0)=\frac{1}{2}\frac{\xi_E(0)}{\xi_F(2r-1,\eta)\xi_F(2r)}.$$
Secondly, since
$$\beta_{2r-1}(s)=\frac{\xi_F(2s)}{\xi_F(2s+1)} \prod_{i=1}^{2r-2}\frac{\xi_F(2s-2r+1+i,\eta^{i-1})}{\xi_F(2s+2r-i,\eta^{i-1})},$$
we have
$$\begin{aligned}
\beta_{2r-1,1}\left(-\frac{1}{2} \right)
&= \frac{\xi_F(-1)}{\Res_{s=-1/2}\xi_F(2s+1)} \prod_{i=1}^{2r-2}\frac{\xi_F(-2r+i,\eta^{i-1})}{\xi_F(2r-1-i,\eta^{i-1})}\\
&=2 \frac{\xi_F(-1)}{\xi_F(0)}\frac{\xi_F(2r)\xi_F(2r-1,\eta)}{\xi_F(2)\xi_F(1,\eta)}\\
&=2 \frac{\xi_F(2r)\xi_F(2r-1,\eta)}{\xi_E(0)},
\end{aligned}$$
where we have used the facts that $\xi_F(-1)=\xi_F(2)$, $\xi_F(0)=-\xi_F(1)$, and $\xi_E(1)=-\xi_E(0)$ by the functional equation $\xi(s)=\xi(1-s)$.
Thus
$$1+\beta_{2r-1,1}\left(-\frac{1}{2} \right)H^{(2r)}_{-1}(0)=2.$$
\end{proof}

Now we can show the general second term identity for $r+1\leq n \leq 2r-1$ by induction on $j=2r-1-n$.
\begin{thm}\label{Thm 4.8}
[Spherical second term identity]
Suppose $r\geq 2$. For $0\leq j \leq r-2$, the following identity holds:
$$\begin{aligned}
&E^{(2r-j-1,r)}_{-1}\left(\frac{r}{2}\right)\prod_{i=0}^{j}F^{(2r-i,r)}\left(\frac{r}{2}\right)\\
&\quad \quad +E^{(2r-j-1,r-j-1)}_{0}\left(\frac{r-j-1}{2}\right)\prod_{i=0}^{j}G^{(2r-i,r-i)}\left(\frac{r-i}{2}\right)\\
&=2c_r E^{(2r-j-1,2r-j-1)}_0\left(\frac{j+1}{2}\right)+\gamma_j E^{(2r-j-1,2r-j-1)}_{-1}\left(\frac{j+1}{2}\right),
\end{aligned}$$
where $\gamma_j$ is a constant, and we write $G^{(2r,r)}(\frac{r}{2})$ for $G^{(2r,r)}_{-1}(\frac{r}{2})$.

Note that for $0\leq i \leq j$, $F^{(2r-i,r)}(s)$ is holomorphic at $s=r/2$, and
$$ F^{(2r-i,r)}\left(\frac{r}{2}\right)=\frac{\xi_E(r-i)}{\xi_E(2r-i)}.$$
For $1\leq i \leq j$, $G^{(2r-i,r-i)}(s)$ is holomorphic at $s=(r-i)/2$, and
$$ G^{(2r-i,r-i)}\left(\frac{r-i}{2}\right)=\frac{\xi_E(r-i)\xi_E(-i)}{\xi_F(2r-2i-1,\eta)\xi_F(2r-2i)\xi_E(2r-i)}.$$
\end{thm}
\begin{proof}
The proof is similar to that of \cite{GT} Theorem 4.8, and is done by induction on $j$.

The case $j=0$ is the spherical second term identity for $n=2r-1$ which we have already shown.

Next we consider the case $j=1$ (so $r\geq 3$). Looking at the constant term along $Q$ of the identity for $j=0$, we obtain
$$\begin{aligned}
&E^{(2r-1,r)}_{-1,Q}\left(\frac{r}{2} \right)F^{(2r,r)}_0\left(\frac{r}{2} \right)+E^{(2r-1,r-1)}_{0,Q}\left(\frac{r-1}{2} \right)G^{(2r,r)}_{-1}\left(\frac{r}{2} \right)\\
&=2c_r E^{(2r-1,2r-1)}_{0,Q}\left(\frac{1}{2} \right)+\gamma_0 E^{(2r-1,2r-1)}_{-1,Q}\left(\frac{1}{2} \right). \quad \quad \quad \quad \quad \quad (1)
\end{aligned}$$

Now we compute the four terms in $(1)$ using Proposition \ref{Prop 4.1}.

For the term $E^{(2r-1,r)}_{-1,Q}(\frac{r}{2})$, we have
$$\begin{aligned}
&E^{(2r-1,r)}_{-1,Q}\left(\frac{r}{2}\right)=\|~\|^{2r-1}E^{(2r-2,r-1)}_{-1}\left(\frac{r+1}{2} \right)\\
& \quad \quad +\|~\|^r \left[ E^{(2r-2,r)}_{-2}\left(\frac{r}{2} \right)F^{(2r-1,r)}_1\left(\frac{r}{2} \right)+E^{(2r-2,r)}_{-1}\left(\frac{r}{2} \right)F^{(2r-1,r)}_0\left(\frac{r}{2} \right) \right]\\
& \quad \quad +\|~\|^{r-1}E^{(2r-2,r-1)}_{-1}\left(\frac{r-1}{2} \right)G^{(2r-1,r)}_0\left(\frac{r}{2} \right)+\|~\|^{r-1}E^{(2r-2,r-1)}_{0}\left(\frac{r-1}{2} \right)G^{(2r-1,r)}_{-1}\left(\frac{r}{2} \right)\\
& \quad \quad -\|~\|^{r-1}\log \|~\| E^{(2r-2,r-1)}_{-1}\left(\frac{r-1}{2} \right)G^{(2r-1,r)}_{-1}\left(\frac{r}{2} \right),
\end{aligned}$$
where we have used the facts that $E^{(2r-2,r)}(s)$ has a double pole at $s=r/2$, $E^{(2r-2,r-1)}(s)$ has a simple pole at $s=(r-1)/2$, $F^{(2r-1,r)}(s)$ is holomorphic at $s=r/2$, and $G^{(2r-1,r)}(s)$ has a simple pole at $s=r/2$.

For the term $E^{(2r-1,r-1)}_{0,Q}(\frac{r-1}{2})$, we have
$$\begin{aligned}
&E^{(2r-1,r-1)}_{0,Q}\left(\frac{r-1}{2} \right)=\|~\|^{2r-1}E^{(2r-2,r-2)}_0\left(\frac{r}{2} \right)\\
&\quad \quad \quad +\|~\|^{r-1}E^{(2r-2,r-1)}_0\left(\frac{r-1}{2} \right)F^{(2r-1,r-1)}\left(\frac{r-1}{2} \right)\\
&\quad \quad \quad +\|~\|^r E^{(2r-2,r-2)}_0\left(\frac{r-2}{2} \right)G^{(2r-1,r-1)}\left(\frac{r-1}{2} \right)\\
&\quad \quad \quad -\|~\|^r \log \|~\| E^{(2r-2,r-2)}_{-1}\left(\frac{r-2}{2} \right)G^{(2r-1,r-1)}\left(\frac{r-1}{2}\right),
\end{aligned}$$
where we have used the facts that $F^{(2r-1,r-1)}(s)$ is holomorphic at $s=(r-1)/2$, $E^{(2r-2,r-2)}(s)$ has a simple pole at $s=(r-2)/2$, and $G^{(2r-1,r-1)}(s)$ is holomorphic at $s=(r-1)/2$.

For the term $E^{(2r-1,2r-1)}_{0,Q}(\frac{1}{2})$, we have
$$\begin{aligned}
E^{(2r-1,2r-1)}_{0,Q}\left(\frac{1}{2} \right)
&=\|~\|^r E^{(2r-2,2r-2)}_0(1)+\|~\|^{r-1}E^{(2r-2,2r-2)}(0)H^{(2r-1)}_0\left(\frac{1}{2} \right)\\
&\quad -\|~\|^{r-1} \log \|~\| E^{(2r-2,2r-2)}(0)H^{(2r-1)}_{-1}\left(\frac{1}{2} \right),
\end{aligned}$$
where we used the facts that $E^{(2r-2,2r-2)}(s)$ is holomorphic at $s=0$, and $H^{(2r-1)}(s)$ has a simple pole at $s=1/2$.

For the term $E^{(2r-1,2r-1)}_{-1,Q}(\frac{1}{2})$, we have
$$\begin{aligned}
E^{(2r-1,2r-1)}_{-1,Q}\left(\frac{1}{2} \right)
&=\|~\|^r E^{(2r-2,2r-2)}_{-1}(1)\\
&\quad +\|~\|^{r-1}E^{(2r-2,2r-2)}(0)H^{(2r-1)}_{-1}\left(\frac{1}{2} \right),
\end{aligned}$$
where we used the facts that $E^{(2r-2,2r-2)}(s)$ has a simple pole at $s=1$, and $H^{(2r-1)}(s)$ has a simple pole at $s=1/2$.

We put the above expressions into $(1)$, and extract the terms containing only $\|~\|^r$ by Lemma \ref{Lem 4.2}. Then we obtain
$$\begin{aligned}
&E^{(2r-2,r)}_{-1}\left(\frac{r}{2} \right)F^{(2r-1,r)}_0\left(\frac{r}{2} \right)F^{(2r,r)}_0\left(\frac{r}{2} \right)+E^{(2r-2,r)}_{-2}\left(\frac{r}{2} \right)F^{(2r-1,r)}_1\left(\frac{r}{2} \right)F^{(2r,r)}_0\left(\frac{r}{2} \right)\\
&\quad +E^{(2r-2,r-2)}_0\left(\frac{r-2}{2} \right)G^{(2r-1,r-1)}\left(\frac{r-1}{2} \right)G^{(2r,r)}_{-1}\left(\frac{r}{2} \right)\\
&=2c_r E^{(2r-2,2r-2)}_0(1)+ \gamma_0 E^{(2r-2,2r-2)}_{-1}(1).
\end{aligned}$$
Note that by Proposition \ref{Prop 4.6} (spherical first term identity for 2nd term range), we have
$$E^{(2r-2,r)}_{-2}\left(\frac{r}{2} \right)=d_{2r-2,r}E^{(2r-2,2r-2)}_{-1}(1).$$
So
$$\begin{aligned}
&E^{(2r-2,r)}_{-1}\left(\frac{r}{2} \right)F^{(2r-1,r)}_0\left(\frac{r}{2} \right)F^{(2r,r)}_0\left(\frac{r}{2} \right)\\
&\quad \quad +E^{(2r-2,r-2)}_0\left(\frac{r-2}{2} \right)G^{(2r-1,r-1)}\left(\frac{r-1}{2} \right)G^{(2r,r)}_{-1}\left(\frac{r}{2} \right)\\
&=2c_r E^{(2r-2,2r-2)}_0(1)+ \gamma_1 E^{(2r-2,2r-2)}_{-1}(1),
\end{aligned}$$
where
$$\gamma_1=\gamma_0+d_{2r-2,r}F^{(2r-1,r)}_1\left(\frac{r}{2} \right)F^{(2r,r)}_0\left(\frac{r}{2} \right).$$
This shows that the identity holds for $j=1$.

Now we consider the general case. Assume the identity holds for a given $j\geq 1$. Looking at the constant term along $Q$ of the identity for $j$, we obtain
$$\begin{aligned}
&E^{(2r-j-1,r)}_{-1,Q}\left(\frac{r}{2} \right)\prod_{i=0}^{j}F^{(2r-i,r)}_0\left(\frac{r}{2} \right)+E^{(2r-j-1,r-j-1)}_{0,Q}\left(\frac{r-j-1}{2} \right)
\prod_{i=0}^{j}G^{(2r-i,r-i)}\left(\frac{r-i}{2} \right)\\
&=2c_r E^{(2r-j-1,2r-j-1)}_{0,Q}\left(\frac{j+1}{2} \right)+\gamma_j E^{(2r-j-1,2r-j-1)}_{-1,Q}\left(\frac{j+1}{2} \right). \quad \quad \quad \quad \quad (2)
\end{aligned}$$

We compute the four terms in $(2)$ by using Proposition \ref{Prop 4.1}.

We have
$$\begin{aligned}
&\quad E^{(2r-j-1,r)}_{-1,Q}\left(\frac{r}{2} \right)\\
&=\|~\|^{2r-j-1}E^{(2r-j-2,r-1)}_{-1}\left(\frac{r+1}{2} \right)\\
&\quad +\|~\|^{r}E^{(2r-j-2,r)}_{-1}\left(\frac{r}{2} \right)F^{(2r-j-1,r)}_0\left(\frac{r}{2} \right)\\
&\quad +\|~\|^{r}E^{(2r-j-2,r)}_{-2}\left(\frac{r}{2} \right)F^{(2r-j-1,r)}_1\left(\frac{r}{2} \right)\\
&\quad +\|~\|^{r-j-1}E^{(2r-j-2,r-1)}_{-1}\left(\frac{r-1}{2} \right)G^{(2r-j-1,r)}_0\left(\frac{r}{2} \right)\\
&\quad +\|~\|^{r-j-1}E^{(2r-j-2,r-1)}_{-2}\left(\frac{r-1}{2} \right)G^{(2r-j-1,r)}_1\left(\frac{r}{2} \right)\\
&\quad -\|~\|^{r-j-1}\log \|~\|E^{(2r-j-2,r-1)}_{-2}\left(\frac{r-1}{2} \right)G^{(2r-j-1,r)}_0\left(\frac{r}{2} \right),
\end{aligned}$$
and
$$\begin{aligned}
&\quad E^{(2r-j-1,r-j-1)}_{0,Q}\left(\frac{r-j-1}{2} \right)\\
&=\|~\|^{2r-j-1}E^{(2r-j-2,r-j-2)}_{0}\left(\frac{r-j}{2} \right)\\
&\quad +\|~\|^{r-j-1}E^{(2r-j-2,r-j-1)}_{0}\left(\frac{r-j-1}{2} \right)F^{(2r-j-1,r-j-1)}_0\left(\frac{r-j-1}{2} \right)\\
&\quad +\|~\|^{r-j-1}E^{(2r-j-2,r-j-1)}_{-1}\left(\frac{r-j-1}{2} \right)F^{(2r-j-1,r-j-1)}_1\left(\frac{r-j-1}{2} \right)\\
&\quad +\|~\|^{r}E^{(2r-j-2,r-j-2)}_{0}\left(\frac{r-j-2}{2} \right)G^{(2r-j-1,r-j-1)}_0\left(\frac{r-j-1}{2} \right)\\
&\quad +\|~\|^{r}E^{(2r-j-2,r-j-2)}_{-1}\left(\frac{r-j-2}{2} \right)G^{(2r-j-1,r-j-1)}_1\left(\frac{r-j-1}{2} \right)\\
&\quad -\|~\|^{r}\log \|~\|E^{(2r-j-2,r-j-2)}_{-1}\left(\frac{r-j-2}{2} \right)G^{(2r-j-1,r-j-1)}_0\left(\frac{r-j-1}{2} \right),
\end{aligned}$$
and
$$\begin{aligned}
&\quad E^{(2r-j-1,2r-j-1)}_{0,Q}\left(\frac{j+1}{2} \right)\\
&=\|~\|^{r}E^{(2r-j-2,2r-j-2)}_{0}\left(\frac{j+2}{2} \right)\\
&\quad +\|~\|^{r}\log\|~\|E^{(2r-j-2,2r-j-2)}_{-1}\left(\frac{j+2}{2} \right)\\
&\quad +\|~\|^{r-j-1}E^{(2r-j-2,2r-j-2)}_{0}]\left(\frac{r}{2} \right)H^{2r-j-1}_0\left(\frac{j+1}{2} \right)\\
&\quad +\|~\|^{r-j-1}E^{(2r-j-2,2r-j-2)}_{-1}\left(\frac{r}{2} \right)H^{2r-j-1}_1\left(\frac{j+1}{2} \right)\\
&\quad -\|~\|^{r-j-1}\log\|~\|E^{(2r-j-2,2r-j-2)}_{-1}\left(\frac{r}{2} \right)H^{2r-j-1}_0\left(\frac{j+1}{2} \right),
\end{aligned}$$
and
$$\begin{aligned}
&\quad E^{(2r-j-1,2r-j-1)}_{-1,Q}\left(\frac{j+1}{2} \right)\\
&=\|~\|^r E^{(2r-j-2,2r-j-2)}_{-1}\left(\frac{j+2}{2} \right)\\
&\quad +\|~\|^{r-j-1}E^{(2r-j-1,2r-j-1)}_{-1}\left(\frac{j}{2} \right)H^{(2r-j-1)}_0\left(\frac{j+1}{2} \right).
\end{aligned}$$
We put the above four expressions into $(2)$ and extract the terms containing only $\|~\|^r$ by Lemma \ref{Lem 4.2}. Then we obtain
$$\begin{aligned}
&\left[ E^{(2r-j-2,r)}_{-1}\left(\frac{r}{2}\right)F^{(2r-j-1,r)}_0\left(\frac{r}{2} \right)+E^{(2r-j-2,r)}_{-2}\left(\frac{r}{2} \right)F^{(2r-j-1,r)}_1\left(\frac{r}{2} \right)\right]\prod_{i=0}^{j}F^{(2r-i,r)}_0\left(\frac{r}{2} \right)\\
&+\left[E^{(2r-j-2,r-j-2)}_{0}\left(\frac{r-j-2}{2} \right)G^{(2r-j-1,r-j-1)}_0\left(\frac{r-j-1}{2} \right)\right.\\
&\quad
\left.+E^{(2r-j-2,r-j-2)}_{-1}\left(\frac{r-j-2}{2} \right)G^{(2r-j-1,r-j-1)}_1\left(\frac{r-j-1}{2} \right) \right]\prod_{i=0}^{j}G^{(2r-i,r-i)}\left(\frac{r-i}{2}\right)\\
&=2c_r E^{(2r-j-2,2r-j-2)}_{0}\left(\frac{j+2}{2} \right)+\gamma_j E^{(2r-j-2,2r-j-2)}_{-1}\left(\frac{j+2}{2} \right).
\end{aligned}$$
By Proposition \ref{Prop 4.6} and Proposition \ref{Prop 4.3}, we have
$$\begin{aligned}
&E^{(2r-j-2,r)}_{-2}\left(\frac{r}{2} \right)=d_{2r-j-2,r}E^{(2r-j-2,2r-j-2)}_{-1}\left(\frac{j+2}{2} \right),\\
&E^{(2r-j-2,r-j-2)}_{-1}\left(\frac{r-j-2}{2} \right)=c_{2r-j-2,r-j-2}E^{(2r-j-2,2r-j-2)}_{-1}\left(\frac{j+2}{2} \right).
\end{aligned}$$
So we obtain
$$\begin{aligned}
&E^{(2r-j-2,r)}_{-1}\left(\frac{r}{2} \right)\prod_{i=0}^{j+1}F^{(2r-i,r)}_0\left(\frac{r}{2} \right)\\
&\quad \quad +E^{(2r-j-2,r-j-2)}_{0}\left(\frac{r-j-2}{2} \right)\prod_{i=0}^{j+1}G^{(2r-i,r-i)}\left(\frac{r-i}{2} \right)\\
&=2c_r E^{(2r-j-2,2r-j-2)}_0\left(\frac{j+2}{2} \right)+\gamma_{j+1} E^{(2r-j-2,2r-j-2)}_{-1}\left(\frac{j+2}{2} \right),
\end{aligned}$$
where
$$\begin{aligned}
&\gamma_{j+1}=\gamma_j-d_{2r-j-2,r}\prod_{i=0}^{j}F^{(2r-i,r)}_0\left(\frac{r}{2} \right)\\
&\quad \quad \quad -c_{2r-j-2,r-j-2}G^{(2r-j-1,r-j-1)}_1\left(\frac{r-j-1}{2} \right)\prod_{i=0}^{j}G^{(2r-i,r-i)}\left(\frac{r-i}{2} \right).
\end{aligned}$$
\end{proof}

\section{Weak second term identity}\label{Section 5}
In this section, we will prove the {\sl weak} second term identity for 2nd term range.

\subsection{First term identity for 1st term range}
Assume $n\geq 2r$. Fix $\vphi \in S(V^n(\BA))$. We write the Laurent expansion of the regularized theta integral at $s=\rho_r=r/2$ as
$$\CE^{(n,r)}(g,s;\vphi)=\frac{B^{(n,r)}_{-1}(\vphi)(g)}{s-\rho_r}+B^{(n,r)}_0(\vphi)(g)+O(s-\rho_r).$$
On the other hand, we write the Laurent expansion of the Siegel Eisenstein series at $s=\rho_{n,r}=(2r-n)/2$ as
$$E^{(n,n)}(g,s;\Phi_{\vphi})=\frac{A^{(n,r)}_{-1}(\vphi)(g)}{s-\rho_{n,r}}+A^{(n,r)}_0(\vphi)(g)+O(s-\rho_{n,r}).$$

For $n\geq 2r+1$, we have the following first term identity due to Ichino, which is a special case of \cite{Ich} Theorem 4.1.

\begin{prop}\label{Prop 5.1}
Assume that $n\geq 2r+1$. Then for all $\vphi \in S(V_c^n(\BA))$, there is a non-zero constant $a_{n,r}$ independent of $\vphi$ such that
$$A_{-1}^{(n,n-r)}(\vphi)=a_{n,r}B_{-1}^{(n,r)}(\mathrm{Ik}(\vphi)).$$
The constant $a_{n,r}$ is given by
$$\begin{aligned}
&a_{n,r}=|D_E|^{\frac{n(2r-n)}{2}}\prod_{i=0}^{n-1} \frac{\xi_F(i+1-2r,\eta^{i})}{\xi_F(2n-2r-i,\eta^{i})}.
\end{aligned}$$
\end{prop}
\begin{proof}
Consider the space $V_c$, which is split and of dimension $2(n-r)$. Then \cite{Ich} Theorem 4.1 says that
$$A_{-1}^{(n,n-r)}(\vphi)(g)=\Res_{s=\rho_{n,n-r}}E^{(n,n)}(g,s;\Phi_{\vphi})=c_{K_c} I(g,\mathrm{Ik}(\vphi)),$$
here $c_{K_c}$ is the constant such that
$$ dh=c_{K_{c}} d_l p ~dk,$$
where $dh$ is the measure on $H_{c}(\BA)$ such that $\vol(H_{c}(F)\bs H_{c}(\BA))=1$, $d_l p=dm~dn$ is left Haar measure on $P_{H_c}(\BA)$, $K_c$ is a maximal compact subgroup of $H_c(\BA)$ as in Section \ref{Section 2.6}, $dk$ is the Haar measure on $K_{c}$ such that $\vol(K_{c})=1$.

But $$I(g,\mathrm{Ik}(\vphi))= \Res_{s=\rho_r}\CE^{(n,r)}(g,s;\mathrm{Ik}(\vphi))=B_{-1}^{(n,r)}(\mathrm{Ik}(\vphi))(g).$$
So we have
$$A_{-1}^{(n,n-r)}(\vphi)=c_{K_c} B_{-1}^{(n,r)}(\mathrm{Ik}(\vphi)).$$

We can compute $a_{n,r}=c_{K_c}$ as follows.

First note that
$$E^{(n,n)}(g,s;\Phi_{\vphi_c^0})=\vphi_c^0(0)E^{(n,n)}(g,s).$$
Then we have
$$\begin{aligned}
A^{(n,n-r)}_{-1}(\vphi^0_c)
&=\vphi^0_c(0)E^{(n,n)}_{-1}\left(\frac{n-2r}{2} \right)\\
&=\vphi^0_c(0)c_{n,r}^{-1} E^{(n,r)}_{-1}\left(\frac{r}{2} \right),
\end{aligned}$$
by Proposition \ref{Prop 4.3}.

On the other hand, by Lemma \ref{Lem 3.4}, we have
$$\begin{aligned}
B^{(n,r)}_{-1}(\mathrm{Ik}(\vphi^0_c))
&=B^{(n,r)}_{-1}(\vphi^0)\\
&=|D_E|^{-\frac{rn}{2}}\prod_{i=1}^r \frac{\xi_E(n-i+1)}{\xi_E(i)} E^{(n,r)}_{-1}\left(\frac{r}{2} \right).
\end{aligned}$$
Combining, we obtain
$$A^{(n,n-r)}_{-1}(\vphi^0_c)=\vphi^0_c(0) |D_E|^{\frac{rn}{2}}c_{n,r}^{-1}\prod_{i=1}^r \frac{\xi_E(i)}{\xi_E(n-i+1)} B^{(n,r)}_{-1}(\mathrm{Ik}^{(n,r)}(\vphi^0_c)).$$
But $\vphi^0_c(0)|D_E|^{\frac{rn}{2}}=|D_E|^{-n(n-r)/2}|D_E|^{\frac{rn}{2}}=|D_E|^{\frac{n(2r-n)}{2}}$, and
$$c_{n,r}=\prod_{i=0}^{n-1} \frac{\xi_F(2n-2r-i,\eta^i)}{\xi_F(i+1-2r,\eta^i)}\prod_{i=1}^{r}\frac{\xi_E(i)}{\xi_E(n-i+1)}.$$
So
$$\begin{aligned}
&a_{n,r}=\vphi^0_c(0)|D_E|^{\frac{rn}{2}}c_{n,r}^{-1}\prod_{i=1}^r \frac{\xi_E(i)}{\xi_E(n-i+1)}\\
&\quad \quad =|D_E|^{\frac{n(2r-n)}{2}}\prod_{i=0}^{n-1} \frac{\xi_F(i+1-2r,\eta^{i})}{\xi_F(2n-2r-i,\eta^{i})}.
\end{aligned}$$
\end{proof}

Next consider the boundary case $n=2r$. The first term identity in this case, also due to Ichino \cite{Ich}, is just Theorem \ref{Thm 3.5} (ii). We record it here for completeness.
\begin{prop}\label{Prop 5.2}
Assume that $n=2r$. Then for all $\vphi \in S(V^n(\BA))$, we have
$$A_0^{(2r,r)}(\vphi)=2B_{-1}^{(2r,r)}(\vphi).$$
\end{prop}

{\bf Remark.} We now show that $c_K=\kappa$.

By Proposition \ref{Prop 4.5}, we have
$$\begin{aligned}
A^{(2r,r)}_0(\vphi^0)=E^{(2r,2r)}(0;\Phi_{\vphi^0})=\vphi^0(0) E^{(2r,2r)}_0(0)=\vphi^0(0)c_r^{-1}E^{(2r,r)}_{-1}(\frac{r}{2}).
\end{aligned}$$

On the other hand, by Lemma \ref{Lem 3.4}, we have
$$\begin{aligned}
B^{(2r,r)}_{-1}(\vphi^0)
&=\Res_{s=\frac{r}{2}}\CE^{(2r,r)}(s;\vphi^0)\\
&=\frac{c_K}{\kappa}|D_E|^{-r^2}\prod_{i=1}^r \frac{\xi_E(2r-i+1)}{\xi_E(i)}E^{(2r,r)}_{-1}\left(\frac{r}{2} \right).
\end{aligned}$$

So we have
$$\begin{aligned}
A^{(2r,r)}_0(\vphi^0)
&=\frac{\kappa}{c_K}\vphi^0(0)|D_E|^{r^2}c_r^{-1}\prod_{i=1}^r \frac{\xi_E(i)}{\xi_E(2r-i+1)}B^{(2r,r)}_{-1}(\vphi^0)\\
&=\frac{\kappa}{c_K} \cdot 2B^{(2r,r)}_{-1}(\vphi^0).
\end{aligned}$$
Since $A^{(2r,r)}_0(\vphi^0)=2 B^{(2r,r)}_{-1}(\vphi^0)$ by the above proposition, we obtain $c_K=\kappa$.

\subsection{Weak first term identity for 2nd term range}
Suppose $r+1\leq n \leq 2r-1$. We can extend the first term identity for the spherical function $\vphi^0$ to all $\vphi \in S(V^n(\BA))^{\circ}$, where $S(V^n(\BA))^{\circ}$ is the $G(\BA)$-span of the spherical Schwartz function $\vphi^0$ in $S(V^n(\BA))$.
\begin{prop}\label{Prop 5.3}
[Weak first term identity for 2nd term range]
Suppose $r+1\leq n \leq 2r-1$. Then for $\vphi \in S(V^n(\BA))^{\circ}$, we have
$$A^{(n,r)}_{-1}(\vphi)=b_{n,r}B^{(n,r)}_{-2}(\vphi),$$
where
$$b_{n,,r}=d_{n,r}^{-1}\prod_{i=1}^r \frac{\xi_E(i)}{\xi_E(n-i+1)}.$$
\end{prop}
\begin{proof}
We have
$$\begin{aligned}
A^{(n,r)}_{-1}(\vphi^0)
&=\vphi^0(0)E^{(n,n)}_{-1}\left(\frac{2r-n}{2} \right)\\
&=\vphi^0(0)d_{n,r}^{-1}E^{(n,r)}_{-2}\left(\frac{r}{2} \right),
\end{aligned}$$
by Proposition \ref{Prop 4.6}.

On the other hand, by Lemma \ref{Lem 3.4}, we have
$$B^{(n,r)}_{-2}(\vphi^0)=|D_E|^{-\frac{rn}{2}}\prod_{i=1}^r \frac{\xi_E(n-i+1)}{\xi_E(i)} E^{(n,r)}_{-2}\left(\frac{r}{2} \right).$$
So we obtain
$$\begin{aligned}
A^{(n,r)}_{-1}(\vphi^0)
&=\vphi^0(0)|D_E|^{\frac{rn}{2}}d_{n,r}^{-1}\prod_{i=1}^r \frac{\xi_E(i)}{\xi_E(n-i+1)}B^{(n,r)}_{-2}(\vphi^0)\\
&=d_{n,r}^{-1}\prod_{i=1}^r \frac{\xi_E(i)}{\xi_E(n-i+1)}B^{(n,r)}_{-2}(\vphi^0)\\
&=b_{n,r}B^{(n,r)}_{-2}(\vphi^0).
\end{aligned}$$
As $A^{(n,r)}_{-1}$ and $B^{(n,r)}_{-2}$ are $G(\BA)$-intertwining, we have, for $\vphi \in S(V^n(\BA))^{\circ}$, that
$$A^{(n,r)}_{-1}(\vphi)=b_{n,r}B^{(n,r)}_{-2}(\vphi).$$
\end{proof}

\subsection{Weak second term identity for 2nd term range}
Suppose $r+1\leq n \leq 2r-1$, i.e. we are in the second term range. We can also extend the second term identity for the spherical function $\vphi^0$ to $\vphi\in S(V^n(\BA))^{\circ}$. Note that there are four terms $E^{(n,r)}_{-1}(\frac{r}{2})$, $E^{(n,n)}_0(\frac{2r-n}{2})$, $E^{(n,n)}_{-1}(\frac{2r-n}{2})$ and $E^{(n,n-r)}_0(\frac{n-r}{2})$ in the spherical second term identity. Now we analyze these four terms.

For the term $E^{(n,n)}_{-1}(\frac{2r-n}{2})$, we have
$$ A^{(n,r)}_{-1}(\vphi^0)=\vphi^0(0) E^{(n,n)}_{-1}\left(\frac{2r-n}{2}\right).$$
Note that the map
$$\begin{aligned}
&A^{(n,r)}_{-1}: S(V^n(\BA)) \ra \CA(G)\\
&\quad \quad \quad \quad \quad \quad \quad\vphi \mapsto A^{(n,r)}_{-1}(\vphi)
\end{aligned}$$
is $G(\BA)$-intertwining, where $\CA(G)$ is the space of automorphic forms on $G(\BA)$.

For the term $E^{(n,n)}_0(\frac{2r-n}{2})$, we have
$$A^{(n,r)}_0(\vphi^0)=\vphi^0(0)E^{(n,n)}_0\left(\frac{2r-n}{2}\right).$$
The map
$$\begin{aligned}
& S(V^n(\BA)) \ra \CA(G) \\
&\quad \quad \quad \quad\vphi \mapsto A^{(n,r)}_0(\vphi)
\end{aligned}$$
is not $G(\BA)$-intertwining, but it is $G(\BA)$-intertwining modulo $\CR$, where $\CR=\mathrm{Im}(A^{(n,r)}_{-1})$
is the image of $A^{(n,r)}_{-1}$. In other words, the composite
$$\begin{aligned}
& S(V^n(\BA)) \ra \CA(G) \ra \CA(G)/\CR\\
&\quad \quad \quad \quad \vphi \mapsto A^{(n,r)}_0(\vphi) \mod \CR
\end{aligned}$$
is a $G(\BA)$-intertwining map.

For the term $E^{(n,r)}_{-1}(\frac{r}{2})$, we have, by Lemma \ref{Lem 3.4}, that
$$B^{(n,r)}_{-1}(\vphi^0)= |D_E|^{-\frac{rn}{2}}\prod_{i=1}^r \frac{\xi_E(n-i+1)}{\xi_E(i)} E^{(n,r)}_{-1}\left(\frac{r}{2}\right).$$
Note that the map $B^{(n,r)}_{-1}: \CS(V^n(\BA))\ra \CA(G)$ is $G(\BA)$-intertwining.

Finally, we need to take care of the complementary term $E^{(n,n-r)}_0(\frac{n-r}{2})$.
Since $n<2r$, we have Ikeda's map
$$\mathrm{Ik}=\mathrm{Ik}^{(n,n-r)}:\CS(V^n(\BA)) \ra \CS(V_c^n(\BA)).$$
For $\vphi \in \CS(V^n(\BA))$, we can define regularized theta integral $\CE^{(n,n-r)}(g,s;\mathrm{Ik}(\vphi))$ associated to $\mathrm{Ik}(\vphi)\in \CS(V_c^n(\BA))$ as in Section \ref{Section 3.2}. It has a simple pole at $s=(n-r)/2$, and we write its Laurent expansion there as
$$\CE^{(n,n-r)}(g,s;\mathrm{Ik}(\vphi))=\sum_{d=-1}^{\infty}B^{(n,n-r)}_d(\mathrm{Ik}(\vphi))(g)\left(s-\frac{n-r}{2}\right)^d.$$
Then $B^{(n,n-r)}_0 \circ \mathrm{Ik}: \CS(V^n(\BA))\ra \CA(G)$ defines a $G(\BA)$-intertwining map.
By Lemma \ref{Lem 3.4} we have
$$B^{(n,n-r)}_0(\mathrm{Ik}(\vphi^0))=|D_E|^{-\frac{n(n-r)}{2}}\prod_{i=1}^{n-r} \frac{\xi_E(n-i+1)}{\xi_E(i)} E^{(n,n-r)}_0\left(\frac{n-r}{2} \right).$$

Recall the spherical second term identity for 2nd term range (Theorem \ref{Thm 4.8}):
$$\begin{aligned}
&E^{(n,r)}_{-1}\left(\frac{r}{2} \right)\prod_{i=0}^{2r-n-1}F^{(2r-i,r)}_0\left(\frac{r}{2} \right)\\
&\quad +E^{(n,n-r)}_0\left(\frac{n-r}{2} \right)\prod_{i=0}^{2r-n-1}G^{(2r-i,r-i)}\left(\frac{r-i}{2}\right)\\
&=2 c_r E^{(n,n)}_0\left(\frac{2r-n}{2} \right) + \gamma_{2r-n-1} E^{(n,n)}_{-1}\left(\frac{2r-n}{2}\right),
\end{aligned}$$
which can be rewritten as
$$\begin{aligned}
&|D_E|^{\frac{rn}{2}}\cdot \prod_{i=1}^r \frac{\xi_E(i)}{\xi_E(n-i+1)}\cdot B^{(n,r)}_{-1}(\vphi^0) \cdot \prod_{i=0}^{2r-n-1}F^{(2r-i,r)}_0\left(\frac{r}{2}\right)\\
&\quad +|D_E|^{\frac{(n-r)n}{2}}\cdot \prod_{i=1}^{n-r}\frac{\xi_E(i)}{\xi_E(n-i+1)} \cdot B^{(n,n-r)}_0(\mathrm{Ik}(\vphi^0))
\cdot \prod_{i=0}^{2r-n-1}G^{(2r-i,r-i)}\left(\frac{r-i}{2}\right)\\
&\equiv 2c_r \vphi^0(0)^{-1} A^{(n,r)}_0(\vphi^0)\mod \CR.
\end{aligned}$$
We can simplify the above as
$$\begin{aligned}
&B^{(n,r)}_{-1}(\vphi^0)+|D_E|^{\frac{n(n-2r)}{2}}\cdot \prod_{i=n-r+1}^r \frac{\xi_E(i)}{\xi_E(n-i+1)}\\
&\quad \quad \quad \quad \quad \cdot \prod_{i=0}^{2r-n-1}\left\{G^{(2r-i,r-i)}\left(\frac{r-i}{2}\right) F^{(2r-i,r)}\left(\frac{r}{2}\right)^{-1}\right\} \cdot B^{(n,n-r)}_0(\mathrm{Ik}(\vphi^0))\\
&\equiv 2c_r \vphi^0(0)^{-1} |D_E|^{-\frac{rn}{2}} \prod_{i=1}^r \frac{\xi_E(n-i+1)}{\xi_E(i)} \cdot \prod_{i=0}^{2r-n-1}F^{(2r-i,r)}\left(\frac{r}{2}\right)^{-1}\cdot  A^{(n,r)}_0(\vphi^0) \mod \CR.
\end{aligned}$$

Note that
$$\begin{aligned}
&F^{(2r-i,r)}\left(\frac{r}{2}\right)=\frac{\xi_E(r-i)}{\xi_E(2r-i)},\\
&G^{(2r-i,r-i)}\left(\frac{r-i}{2}\right)=\frac{\xi_E(r-i)\xi_E(-i)}{\xi_F(2r-2i-1,\eta)\xi_F(2r-2i)\xi_E(2r-i)},\\
&c_r=\frac{1}{2}\prod_{i=1}^r \frac{\xi_E(i)}{\xi_E(r+i)},\\
&\vphi^0(0)=|D_E|^{-rn/2},
\end{aligned}$$
and
$$\prod_{i=n-r+1}^r \frac{\xi_E(i)}{\xi_E(n-i+1)}=1.$$
So we have
$$\begin{aligned}
&\quad B^{(n,r)}_{-1}(\vphi^0)\\
&\quad +|D_E|^{\frac{n(n-2r)}{2}} \prod_{i=0}^{2r-n-1}\left(\frac{\xi_E(-i)}{\xi_F(2r-2i-1,\eta)\xi_F(2r-2i)}\right) B^{(n,n-r)}_0(\mathrm{Ik}(\vphi^0))\\
&\equiv  A^{(n,r)}_0(\vphi^0) \mod \CR.
\end{aligned}$$

Since $B^{(n,r)}_{-1}, B^{(n,n-r)}_0$ are $G(\BA)$-intertwining, and $A^{(n,r)}_0$ are $G(\BA)$-intertwining modulo $\CR=\mathrm{Im}(A^{(n,r)}_{-1})$, we obtain the following result.
\begin{thm}\label{Thm 5.4}
[Weak second term identity]
Suppose $r+1\leq n \leq 2r-1$. For $\vphi \in S(V^n(\BA))^{\circ}$, we have
$$\begin{aligned}
&\quad B^{(n,r)}_{-1}(\vphi)+|D_E|^{\frac{n(n-2r)}{2}} \prod_{i=0}^{2r-n-1}\left(\frac{\xi_E(-i)}{\xi_F(2r-2i-1,\eta)\xi_F(2r-2i)}\right) B^{(n,n-r)}_0(\mathrm{Ik}(\vphi))\\
& \equiv A^{(n,r)}_0(\vphi) \mod \mathrm{Im}(A^{(n,r)}_{-1}),
\end{aligned}$$
where $\mathrm{Im}(A^{(n,r)}_{-1})$ is the image of the $A^{(n,r)}_{-1}$ map.
\end{thm}

\section{Inner product formulae}\label{Section 6}
In this section we consider the inner product formulae for theta lifts from unitary groups.

\subsection{Doubling method}
Let $W$ be the space of row vectors $M_{1,n}(E)$ with a non-degenerate skew-hermitian form $\pair{~,~}_W$. Let $-W$ be the space $M_{1,n}(E)$ with a skew-hermtian form which is negative to the form on $W$, i.e. $\pair{x,y}_{-W}=-\pair{x,y}_W$ for $x,y\in M_{1,n}(E)$. Then we can regard $\CW$ as the direct sum $W\oplus(-W)$, with form given by $\pair{(w_1,w_2),(w_1',w_2')}=\pair{w_1,w_1'}_W+\pair{w_2,w_2'}_{-W}$ for $w_1,w_1'\in W$ and $w_2,w_2'\in -W$.

There is a natural embedding
$$i: U(W)\times U(-W) \ra U(\CW),$$
given by $(w_1,w_2)\cdot i(g_1,g_2)=(w_1 g_1, w_2 g_2)$. Note that the group acts on the space on the right.

We have the following important result due to Piatetski-Shapiro--Rallis \cite{PSR} and Li \cite{Li}, which relates the $L$-function with an integral of Siegel Eisenstein series. We write $I_n(s)=I^{n,n}(s)$ for simplicity.
\begin{thm}\label{Thm 6.1}
[Doubling method]
Let $\pi=\otimes' \pi_v$ be a cuspidal automorphic representation of $G(\BA)$, where $G=U(W)$. Suppose $\Phi(s)=\otimes' \Phi_v(s) \in I_n(s)$ and $f_1=\otimes' f_{1,v}, f_2=\otimes' f_{2,v} \in \pi$ are factorizable. Then
$$ \begin{aligned}
&\quad \int_{[G\times G]}f_1(g_1)\overline{f_2(g_2)}E^{(n,n)}(i(g_1,g_2),s;\Phi)dg_1 dg_2\\
&=\frac{1}{d_n^S(s)}L^S(s+\frac{1}{2},\pi) \prod_{v\in S}Z_v(s,f_{1,v},f_{2,v},\Phi_v,\pi_v),
\end{aligned}$$
where
\begin{itemize}
\item $S$ is any finite set of places of $F$ such that if $v\notin S$, then $v$ is finite and everything is unramified at $v$;
\item $$Z_v(s,f_{1,v},f_{2,v},\Phi_v,\pi_v)=\int_{G_v}\Phi_v(i(g,1),s)\pair{\pi_v(g)f_{1,v},f_{2,v}}dg;$$
\item $$d_n^S(s)=\prod_{v\notin S}d_{n,v}(s),$$ where $d_{n,v}(s)=\prod_{0 \leq i \leq n-1}L_{F_v}(2s+n-i,\eta_v^i)$ for $v\notin S$;
\item $$L^S(s,\pi)=\prod_{v\notin S}L_v(s,\pi_v),$$ where $L_v(s,\pi_v)=L_v(s,BC(\pi_v))$ for $v\notin S$, and $BC(\pi_v)$ is the base change of $\pi_v$ to $GL_n(E_v)$.
\end{itemize}
\end{thm}
\begin{proof}
We have, by the basic identity of Piatetski-Shapiro and Rallis (\cite{PSR} p.3), that for $\Re(s)$ sufficiently large,
$$ \begin{aligned}
&\quad \int_{[G\times G]}f_1(g_1)\overline{f_2(g_2)}E^{(n,n)}(i(g_1,g_2),s;\Phi_{\vphi})dg_1 dg_2\\
&=\int_{G(\BA)}\Phi_{\vphi}(i(g,1),s)\pair{\pi(g)f_1,f_2}dg.
\end{aligned}$$
But we have Euler product
$$ \int_{G(\BA)}\Phi_{\vphi}(i(g,1),s)\pair{\pi(g)f_1,f_2}dg=\prod_v Z_v(s,f_{1,v},f_{2,v},\Phi_v,\pi_v).$$

The local zeta integral $Z_v(s,f_{1,v},f_{2,v},\Phi_v,\pi_v)$ in the unramified case has been computed by J-S.Li \cite{Li}, and his result is as follows.
\begin{prop}
Let $v\nmid 2$ be a finite place of $F$. Suppose $E_v/F_v$ is unramified if $v$ is inert, $\pi_v$ is unramified, $f_{1,v}^0$ and $f_{2,v}^0$ are $G(\CO_v)$-fixed vectors in $\pi_v$ with $\pair{f_{1,v}^0, f_{2,v}^0}=1$, and $\Phi_v^0$ is the spherical section in $I_{n,v}(s)$ with $\Phi_v^0(1)=1$.
Then
$$\begin{aligned}
& Z_v(s, f_{1,v}^0, f_{2,v}^0, \Phi_v^0, \pi_v)=\frac{L_v(s+\frac{1}{2},\pi_v)}{d_{n,v}(s)},
\end{aligned}$$
where
$$ d_{n,v}(s)=\prod_{0 \leq i \leq n-1}L_{F_v}(2s+n-i,\eta_v^i).$$
\end{prop}
\begin{proof}
See \cite{Li} Theorem 3.1. See also \cite{LR} $\S 7$, Proposition 3 and Remark 3.
\end{proof}

For a finite place $v$ of $F$, we say {\sl everything is unramified at $v$} if all the conditions in the above proposition are satisfied.
Let $S$ be a finite set of places of $F$ such that if $v\notin S$, then $v$ is finite and everything is unramified at $v$. Then
$$ \begin{aligned}
&\quad \int_{[G\times G]}f_1(g_1)\overline{f_2(g_2)}E^{(n,n)}(i(g_1,g_2),s;\Phi)dg_1 dg_2\\
&=\prod_{v\notin S}\frac{L_v(s+\frac{1}{2},\pi_v)}{d_{n,v}(s)} \prod_{v\in S}Z_v(s,f_{1,v},f_{2,v},\Phi_v,\pi_v)\\
&=\frac{L^S(s+\frac{1}{2},\pi)}{d_n^S(s)}\prod_{v\in S}Z_v(s,f_{1,v},f_{2,v},\Phi_v,\pi_v).
\end{aligned}$$
\end{proof}

\subsection{Inner product formulae}
First we introduce some notation. Let $W$ be a skew-hermitian space over $E$ of dimension $n$. Let $V$ be the split hermitian space over $E$ of dimension $2r$ as before. Then $U(W)$ and $U(V)$ form a dual pair. Let $V=V^{+}+V^{-}$ be a complete polarization (i.e. a direct sum of two totally isotropic subspaces of rank $r$). Write $X=V^{+}\otimes W$. Then the Weil representation $\omega$ of $U(W)(\BA)\times U(V)(\BA)$ acts on the Schwartz space $\CS(X(\BA))$. In fact, to construct Weil representation of $U(W)(\BA)\times U(V)(\BA)$, we need two auxiliary Hecke characters $\chi_1$ and $\chi_2$ of $E$ (see \cite{HKS,K1}). For simplicity, we choose the characters as follows: we take $\chi_1=1$ to be the trivial Hecke character of $E$ which is used to construct the Weil representation of $U(W)(\BA)$ on $\CS(X(\BA))$, and $\chi_2$ a Hecke character of $E$ such that $\chi_2|_{\BA^{\times}}=\eta^n$ which is used to construct the Weil representation of $U(V)(\BA)$ on $\CS(X(\BA))$.

Note that the Weil representation of $U(\CW)(\BA) \times U(V)(\BA)$ acts on the space $\CS(V^n(\BA))$.
There is a $U(W)(\BA)\times U(W)(\BA)$-intertwining map
$$\sigma: \CS(X(\BA))\hat{\otimes} \CS(X(\BA)) \ra \CS(V^n(\BA)),$$
where we view $\CS(V^n(\BA))$ as a representation of $U(W)(\BA)\times U(W)(\BA)$ via the embedding
$$i: U(W)\times U(-W) \ra U(\CW).$$

Now we consider the global theta lifts from $G=U(W)$ to $H=U(V)$.

Let $\pi$ be a cuspidal automorphic representation of $U(W)(\BA)$. The global {\sl theta lift} of $\pi$ to $U(V)$ is the space
$$\Theta_{2r}(\pi)=\{\theta_{2r}(f,\phi)|f\in \pi, \phi \in \CS(X(\BA))\},$$
where $\theta_{2r}(f,\phi)$ is the function on $H(\BA)$ given by
$$\theta_{2r}(f,\phi)(h)=\int_{[G]}\theta(g,h;\phi)f(g)dg, \quad \text{for $h\in H(\BA)$},$$
and
$$\theta(g,h;\phi)=\sum_{x\in X(F)}\omega(g,h)\phi(x)$$
is the theta kernel function. We call $\theta_{2r}(f,\phi)$ the theta lift of $f$ to $U(V)(\BA)$ (with respect to $\phi$). \\

There are two important properties of the global theta lifts:
\begin{itemize}
\item If $\Theta_{2r_0}(\pi)\neq 0$ for some $r_0$, then $\Theta_{2r}(\pi)\neq 0$ for all $r\geq r_0$ (tower property).
\item If $r$ is the smallest integer with $\Theta_{2r}(\pi)\neq 0$, then $\Theta_{2r}(\pi)$ is in the space of cusp forms.
\end{itemize}

If $\Theta_{2r}(\pi)$ is in the space of cusp forms, then we can consider the inner product
$$\pair{\theta_{2r}(f_1,\phi_1), \theta_{2r}(f_2,\phi_2)}:=\int_{[H]}\theta_{2r}(f_1,\phi_1)(h)\overline{\theta_{2r}(f_2,\phi_2)(h)}dh$$
for $f_1, f_2\in \pi$ and $\phi_1, \phi_2\in \CS(X(\BA))$.

If $\pair{\theta_{2r}(f_1,\phi_1), \theta_{2r}(f_2,\phi_2)}\neq 0$ for some choice of $f_1, f_2 \in \pi$ and $\phi_1, \phi_2\in \CS(X(\BA))$, then $\Theta_{2r}(\pi)\neq 0$. This gives a simple way to determine the non-vanishing of global theta lifts. S.Rallis \cite{R} first used this idea to study global theta lift from a symplectic group to an orthogonal group in the convergent range, and he found a formula which relates the inner product with special values of the partial L-function. Accordingly, this type of formula is often called {\sl Rallis inner product formula}. J-S. Li \cite{Li} extended the Rallis inner product formula to a general dual pair of type I in the convergent range.

In our case, we have the following expression of the inner product (see \cite{GT} Proposition 6.1 for the orthogonal case).

\begin{prop}\label{Prop 6.3}
Suppose $n\geq r$. Let $\pi$ be a cuspidal automorphic representation of $U(W)(\BA)$ with $\Theta_{2(r-1)}(\pi)=0$, so that $\Theta_{2r}(\pi)$ is in the space of cusp forms (possibly zero). Let $f_1, f_2\in \pi$.

(i) For $\phi_1, \phi_2 \in \CS(X(\BA))$, the inner product $\pair{\theta_{2r}(f_1,\phi_1), \theta_{2r}(f_2,\phi_2)}$ is given by
$$\begin{aligned}
&\quad \pair{\theta_{2r}(f_1,\phi_1), \theta_{2r}(f_2,\phi_2)}\\
&=\Res_{s=\frac{r}{2}}\int_{[G\times G]}f_1(g_1)\overline{f_2(g_2)}\CE^{(n,r)}(i(g_1,g_2),s;\sigma(\phi_1\otimes \overline{\phi_2}))dg_1 dg_2\\
&=\int_{[G\times G]}f_1(g_1)\overline{f_2(g_2)}B^{(n,r)}_{-1}(\sigma(\phi_1\otimes \overline{\phi_2}))(i(g_1,g_2))dg_1 dg_2.
\end{aligned}$$
(ii) More generally, for $\vphi=\sum_{i}\sigma(\phi_{1,i}\otimes \overline{\phi_{2,i}})\in \CS(V^n(\BA))$ with $\phi_{1,i}, \phi_{2,i}\in \CS(X(\BA))$, we have
$$\begin{aligned}
&\quad \sum_{i}\pair{\theta_{2r}(f_1, \phi_{1,i}), \theta_{2r}(f_2,\phi_{2,i})}\\
&=\int_{[G\times G]}f_1(g_1)\overline{f_2(g_2)}B^{(n,r)}_{-1}(\vphi)(i(g_1,g_2))dg_1 dg_2.
\end{aligned}$$
\end{prop}
\begin{proof}
Note that (ii) follows from (i) immediately. So we only need to prove (i). Note that $E(h,s)$ has a constant residue $\kappa$ at $s=r/2$.
Write $\sigma(\phi_1\otimes \overline{\phi_2})=\vphi$. Then
$$\begin{aligned}
&\quad \int_{[G\times G]}f_1(g_1)\overline{f_2(g_2)}\CE^{(n,r)}(i(g_1,g_2),s;\vphi)dg_1 dg_2\\
&=\frac{1}{\kappa P_z(s)}\int_{[G\times G]}f_1(g_1)\overline{f_2(g_2)} \left(\int_{[H]}\theta(i(g_1,g_2),h;\omega(z)\vphi)E(h,s)dh\right)dg_1 dg_2\\
&=\frac{1}{\kappa P_z(s)}\int_{[H]}\left( \int_{[G\times G]}f_1(g_1)\overline{f_2(g_2)} \theta(i(g_1,g_2),h;\omega(z)\vphi)dg_1 dg_2\right)E(h,s)dh\\
&=\frac{1}{\kappa P_z(s)}\int_{[H]}\left( \int_{[G\times G]}f_1(g_1)\overline{f_2(g_2)} \theta(i(g_1,g_2),h;\vphi)dg_1 dg_2\right)z' *E(h,s)dh\\
&=\frac{1}{\kappa} \int_{[H]}\left( \int_{[G\times G]}f_1(g_1)\overline{f_2(g_2)} \theta(i(g_1,g_2),h;\vphi)dg_1 dg_2\right) E(h,s)dh\\
&=\frac{1}{\kappa}\int_{[H]} \theta_{2r}(f_1,\phi_1)(h)\overline{\theta_{2r}(f_2,\phi_2)(h)}E(h,s)dh,
\end{aligned}$$
where we have used the Poisson summation formula to show $\theta(i(g_1,g_2),h;\vphi)=\theta(g_1,h;\phi_1)\theta(g_2,h;\overline{\phi_2})$ for $\vphi=\sigma(\phi_1\otimes \overline{\phi_2})$, and have used the adjointness of the operator $z'$ and the fact that $z' * E(h,s)=P_z(s) E(h,s)$.

Taking residue at $s=r/2$ in both sides of the above identity, we obtain
$$\begin{aligned}
&\quad \int_{[H]} \theta_{2r}(f_1,\phi_1)(h)\overline{\theta_{2r}(f_2,\phi_2)(h)}dh\\
&=\int_{[G\times G]}f_1(g_1)\overline{f_2(g_2)}B^{(n,r)}_{-1}(\vphi)(i(g_1,g_2)dg_1 dg_2,
\end{aligned}$$
which is just the desired result.
\end{proof}

Now we study the inner product formulae. First we consider the boundary case $n=2r$.

\begin{thm}\label{Thm 6.4}
[Inner product formula on the boundary]
Suppose $n=2r$. Let $\pi$ be a cuspidal automorphic representation of $U(W)(\BA)$ with $\Theta_{2(r-1)}(\pi)=0$. Let $\vphi \in S(V^n(\BA))$ be such that $\Phi_{\vphi}$ is factorizable as $\otimes' \Phi_v$. Write $\vphi=\sum_i \sigma(\phi_{1,i}\otimes \overline{\phi_{2,i}})$ with $\phi_{1,i}, \phi_{2,i}\in \CS(X(\BA))$. Then for $f_1, f_2\in \pi$, we have
$$\begin{aligned}
&\quad \sum_i \pair{\theta_{2r}(f_1,\phi_{1,i}), \theta_{2r}(f_2, \phi_{2,i})}\\
&=\frac{1}{2}\int_{[G\times G]}f_1(g_1)\overline{f_2(g_2)}A^{(2r,r)}_0(\vphi)(i(g_1,g_2))dg_1 dg_2\\
&=\frac{1}{2d_{2r}^S(0)}\mathrm{Val}_{s=0}\left(L^S(s+\frac{1}{2},\pi)\prod_{v\in S}Z_v(s,f_{1,v},f_{2,v},\Phi_v,\pi_v)\right),
\end{aligned}$$
where $\mathrm{Val}_{s=0}$ means the constant term of the Laurent expansion at $s=0$, and $S$ can be chosen to be any finite set of places of $F$ such that if $v\notin S$, then $v$ is finite and everything is unramified at $v$.
\end{thm}
\begin{proof}
We have
$$\begin{aligned}
&\quad \sum_i \pair{\theta_{2r}(f_1,\phi_{1,i}), \theta_{2r}(f_2, \phi_{2,i})}\\
&=\int_{[G\times G]}f_1(g_1)\overline{f_2(g_2)}B^{(2r,r)}_{-1}(\vphi)(i(g_1,g_2))dg_1 dg_2\\
&=\frac{1}{2} \int_{[G\times G]}f_1(g_1)\overline{f_2(g_2)}A^{(2r,r)}_{0}(\vphi)(i(g_1,g_2))dg_1 dg_2,
\end{aligned}$$
by Proposition \ref{Prop 6.3} and Proposition \ref{Prop 5.2}.

By the doubling method (Theorem \ref{Thm 6.1}), we have
$$\begin{aligned}
&\quad \int_{[G\times G]}f_1(g_1)\overline{f_2(g_2)}A^{(2r,r)}_0(\vphi)(i(g_1,g_2))dg_1 dg_2\\
&=\frac{1}{d_{n}^S(\rho_{n,r})}\mathrm{Val}_{s=\rho_{n,r}}\left(L^S(s+\frac{1}{2},\pi)\prod_{v\in S}Z_v(s,f_{1,v},f_{2,v},\Phi_v,\pi_v) \right).
\end{aligned}$$
Then the desired result follows.
\end{proof}

Next we consider the inner product formula for $n\geq 2r+1$.

\begin{thm}\label{Thm 6.5}
[Inner product formula for 1st term range]
Suppose $n\geq 2r+1$. Let $\pi$ be a cuspidal automorphic representation of $U(W)(\BA)$ with $\Theta_{2(r-1)}(\pi)=0$. Let $\vphi_c \in S(V_c^n(\BA))$ be such that $\Phi_{\vphi_c}$ is factorizable as $\otimes' \Phi_v$. Write $\mathrm{Ik}(\vphi_c)=\sum_i \sigma(\phi_{1,i}\otimes \overline{\phi_{2,i}})$ with $\phi_{1,i}, \phi_{2,i}\in \CS(X(\BA))$. Then for $f_1, f_2\in \pi$, we have
$$\begin{aligned}
&\quad \sum_i \pair{\theta_{2r}(f_1,\phi_{1,i}), \theta_{2r}(f_2, \phi_{2,i})}\\
&=\frac{1}{a_{n,r}d_{n}^S(-\rho_{n,r})}\Res_{s=-\rho_{n,r}}\left(L^S(s+\frac{1}{2},\pi)\prod_{v\in S}Z_v(s,f_{1,v},f_{2,v},\Phi_v,\pi_v)\right),
\end{aligned}$$
where $S$ can be chosen to be any finite set of places of $F$ such that if $v\notin S$, then $v$ is finite and everything is unramified at $v$.
\end{thm}
\begin{proof}
We have
$$\begin{aligned}
&\quad \sum_i \pair{\theta_{2r}(f_1,\phi_{1,i}), \theta_{2r}(f_2, \phi_{2,i})}\\
&=\int_{[G\times G]}f_1(g_1)\overline{f_2(g_2)}B^{(n,r)}_{-1}(\mathrm{Ik}(\vphi_c))(i(g_1,g_2))dg_1 dg_2\\
&=\frac{1}{a_{n,r}}\int_{[G\times G]}f_1(g_1)\overline{f_2(g_2)}A^{(n,n-r)}_{-1}(\vphi_c)(i(g_1,g_2))dg_1 dg_2,
\end{aligned}$$
by Proposition \ref{Prop 6.3} and Proposition \ref{Prop 5.1}.

By the doubling method (Theorem \ref{Thm 6.1}), we have
$$\quad \begin{aligned}
&\quad \int_{[G\times G]}f_1(g_1)\overline{f_2(g_2)}A^{(n,n-r)}_{-1}(\vphi_c)(i(g_1,g_2))dg_1 dg_2\\
&=\frac{1}{d_n^S(-\rho_{n,r})}\Res_{s=-\rho_{n,r}}\left(L^S(s+\frac{1}{2},\pi)\prod_{v\in S}Z_v(s,f_{1,v},f_{2,v},\Phi_v,\pi_v) \right)
\end{aligned}$$
The desired result then follows.
\end{proof}

Finally, we consider the inner product formula for the second term range $r+1 \leq n \leq 2r-1$. First we have the following proposition (see \cite{GT} Proposition 6.3 and {\sl Notes added in proof}).

\begin{prop}\label{Prop 6.6}
Suppose $r+1\leq n \leq 2r-1$. Let $\pi$ be a cuspidal automorphic representation of $U(W)(\BA)$ with $\Theta_{2(r-1)}(\pi)=0$. Let $f_1, f_2\in \pi$.

(i) If $\vphi=\sum_{i}\sigma(\phi_{1,i}\otimes \overline{\phi_{2,i}})\in S(V^n(\BA))^{\circ}$ with $\phi_{1,i}, \phi_{2,i}\in \CS(X(\BA))$ , then
$$\begin{aligned}
&\quad \sum_{i}\pair{\theta_{2r}(f_1,\phi_{1,i}), \theta_{2r}(f_2,\phi_{2,i})}\\
&=\int_{[G\times G]}f_1(g_1)\overline{f_2(g_2)}A^{(n,r)}_0(\vphi)(i(g_1,g_2))dg_1 dg_2,
\end{aligned}$$
where we write $G=U(W)$.

(ii) More generally, if
$\vphi=\sum_{i}\sigma(\phi_{1,i}\otimes \overline{\phi_{2,i}})\in S(V^n(\BA))$ with $\phi_{1,i}, \phi_{2,i}\in \CS(X(\BA))$, then
$$\begin{aligned}
&\quad \sum_{i}\pair{\theta_{2r}(f_1,\phi_{1,i}), \theta_{2r}(f_2,\phi_{2,i})}\\
&=\int_{[G\times G]}f_1(g_1)\overline{f_2(g_2)}A^{(n,r)}_0(\vphi)(i(g_1,g_2))dg_1 dg_2.
\end{aligned}$$
\end{prop}
\begin{proof}
(i) By Proposition \ref{Prop 6.3} and Theorem \ref{Thm 5.4}, we have
$$\begin{aligned}
&\quad \sum_{i}\pair{\theta_{2r}(f_1,\phi_{1,i}), \theta_{2r}(f_2,\phi_{2,i})}\\
&=\int_{[G\times G]}f_1(g_1)\overline{f_2(g_2)}B^{(n,r)}_{-1}(\vphi)(i(g_1,g_2))dg_1 dg_2\\
&=\int_{[G\times G]}f_1(g_1)\overline{f_2(g_2)}A^{(n,r)}_0(\vphi)(i(g_1,g_2))dg_1 dg_2\\
&\quad -c\int_{[G\times G]}f_1(g_1)\overline{f_2(g_2)} B^{(n,n-r)}_0(\mathrm{Ik}(\vphi))(i(g_1,g_2))dg_1 dg_2\\
&\quad +C \int_{[G\times G]}f_1(g_1)\overline{f_2(g_2)} A^{(n,r)}_{-1}(\vphi')(i(g_1,g_2))dg_1 dg_2,
\end{aligned}$$
where $c$ and $C$ are non-zero constants, and $\vphi'$ is another Schwartz function.

We now show that
$$\begin{aligned}
&(1) \qquad \int_{[G\times G]}f_1(g_1)\overline{f_2(g_2)} B^{(n,n-r)}_0(\mathrm{Ik}(\vphi))(i(g_1,g_2))dg_1 dg_2=0,\\
&(2) \qquad \int_{[G\times G]}f_1(g_1)\overline{f_2(g_2)} A^{(n,r)}_{-1}(\vphi')(i(g_1,g_2))dg_1 dg_2=0.
\end{aligned}$$
Both of these are implied by the assumption $\Theta_{2r-2}(\pi)=0$.

First we show the the first identity $(1)$.

Let $V_c=V_c^{+}+V_c^{-}$ be a complete polarization of the complementary space $V_c$. Note that there is a natural isomorphism
$$ \sigma' : \CS((V_c^{+}\otimes W)(\BA))\hat{\otimes} \CS((V_c^{+}\otimes W)(\BA)) \ra \CS(V_c^n(\BA)).$$

Write
$$ \mathrm{Ik}(\vphi)=\sum_j \sigma'(\phi_{1,j}'\otimes \overline{\phi_{2,j}'}),$$
where $\phi_{1,j}', \phi_{2,j}' \in \CS((V_c^{+}\otimes W)(\BA))$.

Write $\vphi_j'=\sigma'(\phi_{1,j}'\otimes \overline{\phi_{2,j}'})$. Then as in the proof of Proposition \ref{Prop 6.3}, we have
$$\begin{aligned}
&\quad \int_{[G\times G]}f_1(g_1)\overline{f_2(g_2)}\CE^{(n,n-r)}(i(g_1,g_2),s;\vphi_j')dg_1 dg_2\\
&=\frac{1}{\kappa'} \int_{[H_c}\theta_{2(n-r)}(f_1,\phi_{1,j}')(h) \overline{\theta_{2(n-r)}(f_2,\phi_{2,j}')(h)}E'(h,s)dh,
\end{aligned}$$
where $H_c=U(V_c)$, $E'(h,s)$ is the auxiliary Eisenstein series on $H_c(\BA)$, $\kappa'$ is the constant residue of $E'(h,s)$ at $s=(n-r)/2$. Since $\Theta_{2(r-1)}(\pi)=0$ and $n-r\leq r-1$, we have $\Theta_{2(n-r)}(\pi)=0$ by the tower property of theta lifts. Thus $\theta_{2(n-r)}(f_1,\phi_{1,j}')=0$ and $ \theta_{2(n-r)}(f_2,\phi_{2,j}')=0$.  Hence
$$ \int_{[G\times G]}f_1(g_1)\overline{f_2(g_2)}\CE^{(n,n-r)}(i(g_1,g_2),s;\vphi_j')dg_1 dg_2=0.$$
Since $\mathrm{Ik}(\vphi)=\sum_j \vphi_j'$, we have
$$ \int_{[G\times G]}f_1(g_1)\overline{f_2(g_2)}\CE^{(n,n-r)}(i(g_1,g_2),s;\mathrm{Ik}(\vphi))dg_1 dg_2=0.$$
In particular,
$$ \int_{[G\times G]}f_1(g_1)\overline{f_2(g_2)}B^{(n,n-r)}_0(\mathrm{Ik}(\vphi))(i(g_1,g_2))dg_1 dg_2=0.$$

Next we show the second identity $(2)$.

Since $n\leq 2r-1$, we have $n\geq 2(n-r)+1$. By Proposition \ref{Prop 5.1}, we have
$$\begin{aligned}
&\quad \int_{[G\times G]}f_1(g_1)\overline{f_2(g_2)} A^{(n,r)}_{-1}(\vphi')(i(g_1,g_2))dg_1 dg_2\\
&=a_{n,n-r} \int_{[G\times G]}f_1(g_1)\overline{f_2(g_2)}B^{(n,n-r)}_{-1}(\mathrm{Ik}(\vphi'))(i(g_1,g_2))dg_1 dg_2.
\end{aligned}$$
Write $\mathrm{Ik}(\vphi')=\sum_j \sigma'(\phi_{1,j}'' \otimes \overline{\phi_{2,j}''})$. Then by Proposition \ref{Prop 6.3} we have
$$ \begin{aligned}
&\quad \sum_j \pair{\theta_{2(n-r)}(f_1,\phi_{1,j}''), \theta_{2(n-r)}(f_2,\phi_{2,j}'')}\\
&=\int_{[G\times G]}f_1(g_1)\overline{f_2(g_2)}B^{(n,n-r)}_{-1}(\mathrm{Ik}(\vphi'))(i(g_1,g_2))dg_1 dg_2.
\end{aligned}$$
The LHS of the above identity is zero since $\theta_{2(n-r)}(f_1,\phi_{1,j}'')=\theta_{2(n-r)}(f_2,\phi_{2,j}'')=0$ by the fact that $\Theta_{2(n-r)}(\pi)=0$. Thus we have
$$ \int_{[G\times G]}f_1(g_1)\overline{f_2(g_2)} A^{(n,r)}_{-1}(\vphi')(i(g_1,g_2))dg_1 dg_2=0.$$

(ii) Both sides of the desired identity factor through the quotient
$$ S(V^n(\BA)) \ra S(V^n(\BA))_{H(\BA)},$$
where the quotient map is given by
$$\vphi \mapsto \Phi_{\vphi}\in I_n(\rho_{n,r}).$$
The quotient space $S(V^n(\BA))_{H(\BA)}$ is generated by the image of the spherical function $\vphi^0$ as a $G(\BA)$-module, since $S(V_v^n)_{H_v}\cong R_n(V_v)$ is generated by the spherical section $\Phi_{\vphi^0_v}$ for each $v$. In other words, the subspace $S(V^n(\BA))^{\circ}$ is mapped surjectively onto this quotient. Thus to show the identity, it suffices to show the desired identity on the subspace $S(V^n(\BA))^{\circ}$, which we have already done in (i). Thus the result follows.
\end{proof}

This proposition implies the following inner product formula for 2nd term range.
\begin{thm}\label{Thm 6.7}
[Inner product formula for 2nd term range]
Suppose $r+1\leq n \leq 2r-1$. Let $\pi$ be a cuspidal automorphic representation of $U(W)(\BA)$ with $\Theta_{2(r-1)}(\pi)=0$. Let $\vphi \in S(V^n(\BA))$ be such that $\Phi_{\vphi}$ is factorizable as $\otimes'\Phi_v$. Write $\vphi=\sum_{i}\sigma(\phi_{1,i}\otimes \overline{\phi_{2,i}})$ with $\phi_{1,i}, \phi_{2,i}\in \CS(X(\BA))$. Then for $f_1, f_2\in \pi$, we have
$$\begin{aligned}
&\quad \sum_{i}\pair{\theta_{2r}(f_1, \phi_{1,i}), \theta_{2r}(f_2,\phi_{2,i})}\\
&=\frac{1}{d_n^S(\rho_{n,r})}\mathrm{Val}_{s=\rho_{n,r}}\left(L^S(s+\frac{1}{2},\pi)\prod_{v\in S}Z_v(s,f_{1,v},f_{2,v},\Phi_v,\pi_v) \right),
\end{aligned}$$
where $\mathrm{Val}_{s=s_0}$ means the constant term of the Laurent expansion at $s=s_0$, and $S$ can be chosen to be any finite set of places of $F$ such that if $v\notin S$, then $v$ is finite and everything is unramified at $v$.
\end{thm}
\begin{proof}
By Proposition \ref{Prop 6.6}, we have
$$\begin{aligned}
&\quad \sum_{i}\pair{\theta_{2r}(f_1, \phi_{1,i}), \theta_{2r}(f_2,\phi_{2,i})}\\
&=\int_{[G\times G]}f_1(g_1)\overline{f_2(g_2)}A^{(n,r)}_0(\vphi)(i(g_1,g_2))dg_1 dg_2.
\end{aligned}$$

By the doubling method (Theorem \ref{Thm 6.1}), we have
$$\begin{aligned}
&\quad \int_{[G\times G]}f_1(g_1)\overline{f_2(g_2)}A^{(n,r)}_0(\vphi)(i(g_1,g_2))dg_1 dg_2\\
&=\frac{1}{d_n^S(\rho_{n,r})}\mathrm{Val}_{s=\rho_{n,r}}\left(L^S(s+\frac{1}{2},\pi)\prod_{v\in S}Z_v(s,f_{1,v},f_{2,v},\Phi_v,\pi_v) \right).
\end{aligned}$$
Then the desired result follows.
\end{proof}

\section{Appendix: Proof of Proposition 4.1}\label{Appendix}
In this section, we will give a proof of Proposition \ref{Prop 4.1} following D.Jiang \cite{J1,J2}.

Recall that $G=G_n=U(n,n)$, and for $1\leq r \leq n$, $P_r=P^n_r$ is the parabolic subgroup of $G_n$ with stabilizes the subspace $\{(0_n, a_1, \ldots, a_r, 0_{n-r})|a_i\in E\}$, and $I^{n,r}(s)=\Ind_{P_r(\BA)}^{G(\BA)}|\cdot|^s$ is the degenerate principal series.

For a holomorphic section $f \in I^{n,r}(s)$, the associated Eisenstein series on $G_n(\BA)$ is defined by
$$E^{(n,r)}(g,s;f)=\sum_{\gamma \in P^n_r(F)\bs G_n(F)}f(\gamma g,s), \quad \quad \text{for $g\in G_n(\BA)$}.$$
It is well-known that this series converges absolutely for $\Re(s)\gg 0$, and has a meromorphic continuation to the whole $s$-plane.

The spherical Eisenstein series $E^{(n,r)}(g,s)$ is the one associated to the spherical section $f^0$ in $I^{n,r}(s)$ with $f^0(1)=1$, i.e. $E^{(n,r)}(g,s)=E^{(n,r)}(g,s;f^0)$.

Write $P^n_1=Q$. For $f \in I^{n,r}(s)$, the constant term of the Eisenstein series $E^{(n,r)}(g,s;f)$ along $Q$ is given by
$$E_{Q}^{(n,r)}(g,s;f)=\int_{N_1(F)\bs N_1(\BA)}E^{(n,r)}(ug,s;f)du,$$
where $N_1$ is the unipotent radical of $Q$ given explicitly by
$$N_1(F)=\left\{u(b,c,d)=\begin{pmatrix} 1& b & c & d \\ &1_{n-1} & ^t\!\bar{d} & \\ & & 1 & \\ & & -^t\!\bar{b} & 1_{n-1} \end{pmatrix}\right\},$$
where $b,d \in M_{1,n-1}(E)$, and $c\in E$ with $c+\bar{c}=-(b ^t\!\bar{d}+d ^t\!\bar{b})$.

Unfolding, we get
$$E_{Q}^{(n,r)}(g,s;f)=\sum_{\gamma \in P_r(F)\bs G(F)/N_1(F)}\int_{N_1^{\gamma}(F)\bs N_1(\BA)}f(\gamma ug,s)du,$$
where $N_1^{\gamma}=N_1\cap \gamma^{-1}P_r \gamma$.

Similar to \cite{J2} Lemma 2.1, we have the following
\begin{lem}
Let
$$w_1=\begin{pmatrix} & & 1& & & & &\\  &1_{r-1} & & & & & &\\ 1 & & & & & & &\\  & & &1_{n-r-1} & & & &\\  & & & & & & 1 &\\  & & & & &1_{r-1} & &\\  & & & & 1 & & &\\  & & & & & & & 1_{n-r-1}\end{pmatrix},$$

$$w_2=\begin{pmatrix} & & 1 & \\ & 1_{n-1}& &  \\ -1 & & & \\ & & & 1_{n-1}\end{pmatrix}.$$
Then we have

(i)
$$
G_n=\begin{cases}
P^n_r Q \cup P^n_r w_1 Q \cup P^n_r w_2 Q &\text{if $1\leq r<n$,}\\
P^n_r Q \cup P^n_r w_2 Q  &\text{if $r=n$.}
\end{cases}$$

(ii)
$$\begin{aligned}
&\text{$P^n_r\bs P^n_r Q/N_1 \cong P^{n-1}_{r-1} \bs G_{n-1}$ if $r>1$,}\\
&P^n_r\bs P^n_r w_1 Q/N_1 \cong P^{n-1}_r \bs G_{n-1},\\
&\text{$P^n_1\bs P^n_1 w_2 Q/N_1=[w_2]$, and $P^n_r\bs P^n_r w_2 Q/N_1 \cong P^{n-1}_{r-1}\bs G_{n-1}$ if $r>1$.}
\end{aligned}$$
\end{lem}

Let $B$ be the standard Borel subgroup of $G=U(n,n)$. Then $B=AN$, where $A\cong (R_{E/F}GL_1)^n$ is the Levi factor, and $N$ is the unipotent radical. For $w=w_1$ or $w_2$, let $N_w=w^{-1} \bar{N}w \cap N$, where $\bar{N}$ is the opposite of $N$.

For $f \in I^{n,r}(s)$, define two intertwining operators as follows:
$$\begin{aligned}
&\CU_1(s)f(g)=\int_{N_{w_1}(\BA)}f(w_1 ug,s)du,\\
&\CU_2(s)f(g)=\int_{N_{w_2}(\BA)}f(w_2 ug,s)du.
\end{aligned}$$

Let $M_1$ be the Levi factor of $Q=P_1$. Then $M_1\cong R_{E/F}GL_1\times G_{n-1}$, and we write an element in $M_1$ as $m_1(a,g)$ with $a\in R_{E/F}GL_1$ and $g\in G_{n-1}$.

We have the following result analogous to \cite{J1} Chapter 3, Lemma 1.0.2.
\begin{lem}
For $m_1(a,g)\in M_1(\BA)$ and $f \in I^{n,r}(s)$, we have

(1) $f(m_1(a,g),s)=\|a\|^{s+\frac{2n-r}{2}}f(m_1(1,g),s).$

(2) $ \CU_1(s)f (m_1(a,g))=\|a\|^{r-1} \CU_1(s)f (m_1(1,g)).$

(3)  $ \CU_2(s)f(m_1(a,g))=\|a\|^{-s+\frac{2n-r}{2}} \CU_2(s)f(m_1(1,g)).$
\end{lem}

For $f \in I^{n,r}(s)$, define a function $i^*f$ on $G_{n-1}(\BA)$ by
$$ i^* f(g)=f(m_1(1,g),s), \quad \text{for $g\in G_{n-1}(\BA)$}.$$

Then we have the following result analogous to \cite{J1} Chapter 3, Proposition 1.0.1.
\begin{lem}
For $n\geq r >1$,

(1) $i^*: I^{n,r}(s)\ra I^{n-1,r-1}(s+\frac{1}{2})$.

(2) $i^* \circ \CU_1(s): I^{n,r}(s) \ra I^{n-1,r}(s)$.

(3) $i^* \circ \CU_2(s): I^{n,r}(s) \ra I^{n-1,r-1}(s-\frac{1}{2})$.
\end{lem}

We can obtain the following expression for the constant term of Eisenstein series. See \cite{J1} pp.76-77 for the symplectic case, and \cite{GT} Proposition 4.1 for the even orthogonal case.
\begin{prop}\label{Prop 7.4}
For $f \in I^{n,r}(s)$ and $m_1(a,g)\in M_1(\BA)$ with $a\in \BA_E^{\times}$ and $g\in G_{n-1}(\BA)$, we have:
\begin{itemize}
\item For $1< r <n$,
$$\begin{aligned}
E^{(n,r)}_{Q}(m_1(a,g),s;f)
&=\|a\|^{s+\frac{2n-r}{2}}E^{(n-1,r-1)}\left(g,s+\frac{1}{2};i^*f \right)\\
& \quad +\|a\|^{r} E^{(n-1,r)}(g,s;i^{*}\circ \CU_1(s) f)\\
& \quad +\|a\|^{-s+\frac{2n-r}{2}}E^{(n-1,r-1)}\left(g,s-\frac{1}{2};i^* \circ \CU_2(s)f \right);
\end{aligned}$$

\item For $r=1$ and $n\geq 2$,
$$\begin{aligned}
E^{(n,1)}_{Q}(m_1(a,g),s;f)
&=\|a\|^{s+\frac{2n-1}{2}}i^*f(g)+\|a\| E^{(n-1,1)}(g,s;i^{*}\circ \CU_1(s) f)\\
&\quad +\|a\|^{-s+\frac{2n-1}{2}}i^* \circ \CU_2(s)f(g).
\end{aligned}$$

\item For $r=n\geq 2$,
$$\begin{aligned}
E^{(n,n)}_{Q}(m_1(a,g),s;f)
&=\|a\|^{s+\frac{n}{2}}E^{(n-1,n-1)}\left(g,s+\frac{1}{2};i^*f \right)\\
&\quad +\|a\|^{-s+\frac{n}{2}}E^{(n-1,n-1)}\left(g,s-\frac{1}{2};i^*\circ \CU_2(s)f \right).
\end{aligned}$$
\end{itemize}
\end{prop}
\begin{proof}
Same as the proof of \cite{J1} Chapter 3, Theorem 1.0.2.
\end{proof}

Recall that our goal is to prove Proposition \ref{Prop 4.1}. To achieve this, we see from the above proposition that we only need to compute the following two integrals
$$c_1(s):=\CU_1(s)f^0(1)=\int_{N_{w_1}(\BA)}f^0(w_1 u,s)du$$
and $$c_2(s):=\CU_2(s)f^0(1)=\int_{N_{w_2}(\BA)}f^0(w_2 u,s)du,$$
where $f^0$ is the spherical section in $I^{n,r}(s)$ with $f^0(1)=1$.

To do this, we use the Gindikin-Karpelevich argument as in \cite{PSR} $\S 5.3$.\\

Let $T$ be the standard maximal $F$-split torus of $G=U(n,n)$. Then $T\cong (GL_1)^n$ and a typical element in $T$ is of the form $t=\mathrm{diag}(a_1,\ldots, a_n, a_1^{-1}, \ldots, a_n^{-1})$, where each $a_i\in GL_1$.

Let $X^*(T)$ be the group of rational characters of $T$. Then $X^*(T)$ is a free $\BZ$-module with basis $\{x_i|1\leq i \leq n\}$, where $x_i \in X^*(T)$ is given by $x_i(t)=a_i$ for $t=\mathrm{diag}(a_1,\ldots, a_n, a_1^{-1}, \ldots, a_n^{-1})\in T$.

Let $\Sigma$ be the set of $F$-roots of $G$ with respect to $T$, $\Sigma^{+}$ the set of positive roots determined by $N$. For a root $\alpha$, let $-\alpha$ be the root given by $(-\alpha)(t)=\alpha(t)^{-1}$ for $t\in T$. For two roots $\alpha_1, \alpha_2$, let $\alpha_1+\alpha_2$ be the root given by $(\alpha_1+\alpha_2)(t)=\alpha_1(t)\alpha_2(t)$. Then
$$\begin{aligned}
&\Sigma=\{\pm(x_i \pm x_j)|1\leq i<j \leq n\}\cup \{\pm 2x_k|1\leq k \leq n\},\\
&\Sigma^{+}=\{(x_i\pm x_j)|1\leq i < j \leq n \}\cup \{2x_k|1\leq k \leq n\}.
\end{aligned}$$

Define the Weyl group $W_G=N_G(T)/C_G(T)$, where $N_G(T)$ is the normalizer of $T$ in $G$, and $C_G(T)$ is the centralizer of $T$ in $G$.

For an element $w$ in the Weyl group $W_G$, let
$$\Sigma^{+}(w)=\{\alpha \in \Sigma^{+}|w^{-1}\alpha \in \Sigma^{-}\},$$
where $w^{-1}\alpha$ is the root given by $(w^{-1}\alpha)(t)=\alpha(wtw^{-1})$ for $t\in T$.

For a positive root $\alpha$, let $G_{\alpha}=C_G(\ker \alpha)$, $N_{\alpha}=N \cap G_{\alpha}$, $A_{\alpha}=A\cap G_{\alpha}$, and let $\bar{N}_{\alpha}$ be the opposite of $N_{\alpha}$. For $w\in W_G$, let $N_{w}=w^{-1}\bar{N} w\cap N$.

To compute $c_1(s)$ and $c_2(s)$, we proceed as in \cite{J2} Chapter 3, $\S 2.4$. First choose an element $w'\in W_{M_r}:=N_{M_r}(T)/C_{M_r}(T)$ which is given by
$$w'=\begin{pmatrix} 0& 1_{r-1} &&&& \\ 1 & 0&&&& \\ && 1_{n-r} &&& \\ &&& 0& 1_{r-1} &\\ &&& 1 &0& \\ &&&&& 1_{n-r} \end{pmatrix}.$$
Let $w_1^*=w' w_1$, $w_2^*=w' w_2$.

Then
$$\begin{aligned}
&c_{1}(s)=\int_{N_{w_1^*}(\BA)}f^0(w_1^* u,s)du=\prod_v \int_{N_{w_1^*}(F_v)}f^0(w_1^* u,s)du,\\
&c_{2}(s)=\int_{N_{w_2^*}(\BA)}f^0(w_2^* u,s)du=\prod_v \int_{N_{w_2^*}(F_v)}f^0(w_2^* u,s)du.
\end{aligned}$$

By \cite{Lai} Proposition 4.4, we have
$$\begin{aligned}
&\int_{N_{w_1^*}(F_v)}f^0(w_1^* u,s)du=\prod_{\alpha \in \Sigma^{+}(w_1^*)}\int_{\bar{N}_{\alpha}(F_v)}f^0_{\alpha}(u,s)du,\\
&\int_{N_{w_2^*}(F_v)}f^0(w_2^* u,s)du=\prod_{\alpha \in \Sigma^{+}(w_2^*)}\int_{\bar{N}_{\alpha}(F_v)}f^0_{\alpha}(u,s)du,
\end{aligned}$$
where $f^0_{\alpha}$ is the restriction of $f^0$ to $G_{\alpha}$.

It remains to compute
$$M_{\alpha,v}(s):=\int_{\bar{N}_{\alpha}(F_v)}f^0_{\alpha}(u,s)du$$
for $\alpha \in \Sigma^{+}(w_1^*)\cup \Sigma^{+}(w_2^*)$.

Note that each $G_{\alpha,v}$ is of semi-simple $F_v$-rank 1, so we can apply the formula in \cite{Lai} $\S 3$ to compute $M_{\alpha,v}(s)$. See the proof of \cite{Tan2} Proposition 2.1 for relevant computations.\\

First we record the following lemma which will be useful in our later computation.
\begin{lem}\label{Lem 7.5}
For $f \in I^{n,r}_v(s)$, we have
$$f \in \Ind_{B_v}^{G_v}(\prod_{1\leq i \leq r}\|a_i\|^{s+i-\frac{r+1}{2}}\prod_{r+1\leq i \leq n}\|a_i\|^{-n+i-\frac{1}{2}}),$$
where $\prod_{1\leq i \leq r}\|a_i\|^{s+i-\frac{r+1}{2}}\prod_{r+1\leq i \leq n}\|a_i\|^{-n+i-\frac{1}{2}}$ denotes the character of the Levi factor $A_v$ of $B_v$ given by
$$\mathrm{diag}(a_1,\ldots, a_n, \bar{a}_1^{-1}, \ldots, \bar{a}_n^{-1})\mapsto \prod_{1\leq i \leq r}\|a_i\|^{s+i-\frac{r+1}{2}}\prod_{r+1\leq i \leq n}\|a_i\|^{-n+i-\frac{1}{2}}.$$
\end{lem}
\begin{proof}
For simplicity, we omit the subscript $v$ from notation. Recall that for a character $\chi$ of the Levi factor $A$ of $B$, the normalized induced representation $\Ind_{B}^{G}(\chi)$ is the space of all smooth functions $f: G \ra \BC$ such that
$$ f(a g)=\delta_B^{\frac{1}{2}}(a)\chi(a)f(g)$$
for all $a\in A$ and $g\in G$,
where $\delta_B$ is the modular character of $B$ given by
$$ a \mapsto \prod_{1\leq i \leq n}\|a_i\|^{2n-2i+1}$$
for $a=\mathrm{diag}(a_1,\ldots, a_n, \bar{a}_1^{-1}, \ldots, \bar{a}_n^{-1})\in A$.

For $a=\mathrm{diag}(a_1, \ldots, a_n, \bar{a}_1^{-1}, \ldots, \bar{a}_n^{-1}) \in A$, we have $a= m_r(a',g') \in M_r\cong GL_r(E) \times G_{n-r}$, where $a'=\mathrm{diag}(a_1, \ldots, a_r) \in GL_r(E)$. Then for $f \in I^{n,r}(s)$, we have
$$\begin{aligned}
f(ag)
&=\|a_1 \cdots a_r\|^{s+n-\frac{r}{2}}f(g)\\
&=\delta_B^{1/2}(a)\left(\delta_B^{-1/2}(a)\|a_1 \cdots a_r\|^{s+n-\frac{r}{2}}f(g)\right)\\
&=\delta_B^{1/2}(a)\left(\prod_{1\leq i \leq r}\|a_i\|^{s+i-\frac{r+1}{2}} \prod_{r+1\leq i \leq n}\|a_i\|^{-n+i-\frac{1}{2}}\right)f(g),
\end{aligned}$$
for all $a\in A$ and $g\in G$.

The desired result then follows.
\end{proof}

Now we compute $c_1(s)$ and $c_2(s)$.
First we deal with $c_1(s)$.

For $w_1^*=w' \cdot w_1$, we have
$$\Sigma^{+}(w_1^*)=\{x_i-x_{r+1}|1\leq i \leq r\}.$$

For $\alpha=x_i-x_{r+1}$ with $1\leq i \leq r$, we have $G_{\alpha,v}\cong GL_2(E_v)\times (E_v^{\times})^{n-2}$. As in the proof of \cite{Tan2} Proposition 2.1, we can reduce our computation to the derived group $G_{\alpha}'$ of $G_{\alpha}$ since it yields the same $M_{\alpha}(s)$. We have $G_{\alpha,v}'\cong SL_2(E_v)$, and $A_{\alpha,v}'$ is the subgroup of $A_v\cong (E_v^{\times})^n$ whose elements have 1 along the diagonals except the $i$-th and the $(r+1)$-th entries where they are $x$ and $x^{-1}$ respectively for some $x\in E_v^{\times}$. By Lemma \ref{Lem 7.5}, we have
$$ f^0_{\alpha,v}\in \Ind_{B_{\alpha,v}'}^{G_{\alpha,v}'}(\|x\|^{s+i-\frac{r+1}{2}} \cdot \|x^{-1}\|^{-n+r+1-\frac{1}{2}})=\Ind_{B_{\alpha,v}'}^{G_{\alpha,v}'} \|x\|^{s+i+n-\frac{3r}{2}-1}.$$
Then by the formula in \cite{Lai} $\S 3.5$, we have
$$M_{\alpha,v}(s)=\vol(\CO_{E_v})\frac{\zeta_{E_v}(s+i+n-\frac{3r}{2}-1)}{\zeta_{E_v}(s+i+n-\frac{3r}{2})},$$
where we take $\vol(\CO_{E_v})=1$ if $v$ is a real place, and we take $\vol(\CO_{E_v})=\vol(\CO_v)^2$ and $\zeta_{E_v}(s)=\zeta_{F_v}(s)^2$ if $v$ is finite split.

So
$$c_{1,v}(s):=\prod_{\alpha \in \Sigma^{+}(w_1^*)} M_{\alpha,v}(s)=
\vol(\CO_{E_v})^r \frac{\zeta_{E_v}(s+n-\frac{3r}{2})}{\zeta_{E_v}(s+n-\frac{r}{2})},$$
where we take $\vol(\CO_{E_v})=1$ if $v$ is a real place.

Thus
$$\begin{aligned}
c_1(s)
&=\prod_{v\nmid \infty}\vol(\CO_{E_v})^r \frac{\zeta_E(s+n-\frac{3r}{2})}{\zeta_E(s+n-\frac{r}{2})}\\
&=|D_E|^{-r/2}\frac{\zeta_E(s+n-\frac{3r}{2})}{\zeta_E(s+n-\frac{r}{2})}\\
&=\frac{\xi_E(s+n-\frac{3r}{2})}{\xi_E(s+n-\frac{r}{2})}.
\end{aligned}$$\\

Next we deal with $c_2(s)$. We have
$$\Sigma^{+}(w_2^*)=
\begin{cases}
\{x_i+x_r|1\leq i \leq r-1\}\cup \{2x_r\}\cup \{ x_r+x_{j}, x_r-x_{j}| r+1 \leq j \leq n\}& \text{if $r\leq n-1$},\\
\{x_i+x_n|1\leq i \leq n-1\}\cup \{2x_n\} & \text{if $r=n$}.
\end{cases}$$

For $\alpha=x_i+x_r$ with $1\leq i \leq r-1$, we check that $G_{\alpha,v}\cong GL_2(E_v)\times (E_v^{\times})^{n-2}$. As before, we can reduce our computation to the derived group $G_{\alpha}'$ of $G_{\alpha}$. We have $G_{\alpha,v}'\cong SL_2(E_v)$ and $A_{\alpha,v}'$ is the subgroup of $A_v\cong (E_v^{\times})^n$ whose elements have 1 along the diagonals except the $i$-th and $r$-th entries where they are $x$ and $\bar{x}$ respectively for some $x\in E_v^{\times}$. By Lemma \ref{Lem 7.5}, we have
$$f^0_{\alpha,v}\in \Ind_{B_{\alpha,v}'}^{G_{\alpha,v}'}(\|x\|^{s+i-\frac{r+1}{2}}\cdot \|\bar{x}\|^{s+r-\frac{r+1}{2}})
=\Ind_{B_{\alpha,v}'}^{G_{\alpha,v}'}\|x\|^{2s+i-1}.$$
Then by the formula in \cite{Lai} $\S 3.5$, we have
$$M_{\alpha,v}(s)=\vol(\CO_{E_v}) \frac{\zeta_{E_v}(2s+i-1)}{\zeta_{E_v}(2s+i)}.$$

For $\alpha=2x_r$, we have $G_{\alpha,v}\cong U(1,1)(F_v)\times (E_v^{\times})^{n-1}$, $G_{\alpha,v}'\cong SL_2(F_v)$, and $A_{\alpha,v}'$ is the subgroup of $A_v\cong (E_v^{\times})^n$ whose elements have 1 along the diagonals except the $r$-th entry where it can be any $x\in F_v^{\times}$.  By Lemma \ref{Lem 7.5} we have
$$\Phi^0_{\alpha}\in \Ind_{B_{\alpha,v}'}^{G_{\alpha,v}'}\|~\|^{s+r-\frac{r+1}{2}}=\Ind_{B_{\alpha,v}'}^{G_{\alpha,v}'}|x|^{2s+r-1}.$$
Then by the formula in \cite{Lai} $\S 3.3$, we have
$$M_{\alpha,v}(s)=\vol(\CO_v)\frac{\zeta_{F_v}(2s+r-1)}{\zeta_{F_v}(2s+r)} .$$

For $\alpha=x_r+x_j$ with $r+1\leq j \leq n$, we have $G_{\alpha,v}\cong GL_2(E_v)\times (E_v^{\times})^{n-2}$,  $G_{\alpha,v}'\cong SL_2(E_v)$, and $A_{\alpha,v}'$ is the subgroup of $A_v\cong (E_v^{\times})^n$ whose elements have 1 along the diagonals except the $r$-th and the $j$-th entries where they are $x$ and $\bar{x}$ respectively for some $x\in E_v^{\times}$. By Lemma \ref{Lem 7.5}, we have
$$ f^0_{\alpha,v}\in \Ind_{B_{\alpha,v}'}^{G_{\alpha,v}'}(\|x\|^{s+r-\frac{r+1}{2}}\cdot \|\bar{x}\|^{-n+j-\frac{1}{2}})
=\Ind_{B_{\alpha,v}'}^{G_{\alpha,v}'}\|x\|^{s+j-n+\frac{r}{2}-1}.$$
Then by the formula in \cite{Lai} $\S 3.5$, we have
$$M_{\alpha,v}(s)=\vol(\CO_{E_v})\frac{\zeta_{E_v}(s+j-n+\frac{r}{2}-1)}{\zeta_{E_v}(s+j-n+\frac{r}{2})}.$$

For $\alpha=x_r-x_j$ with $r+1\leq j \leq n$, we have $G_{\alpha,v}\cong GL_2(E_v)\times (E_v^{\times})^{n-2}$, $G_{\alpha,v}'\cong SL_2(E_v)$, and $A_{\alpha,v}'$ is the subgroup of $A_v\cong (E_v^{\times})^n$ whose elements have 1 along the diagonals except the $r$-th and the $j$-th entries where they are $x$ and $x^{-1}$ respectively for some $x\in E_v^{\times}$. By Lemma \ref{Lem 7.5}, we have
$$ f^0_{\alpha,v}\in \Ind_{B_{\alpha,v}'}^{G_{\alpha,v}'}(\|x\|^{s+r-\frac{r+1}{2}}\cdot \|x^{-1}\|^{-n+j-\frac{1}{2}})
=\Ind_{B_{\alpha,v}'}^{G_{\alpha,v}'}\|x\|^{s-j+n+\frac{r}{2}}.$$
Then by the formula in \cite{Lai} $\S 3.5$, we have
$$M_{\alpha,v}(s)=\vol(\CO_{E_v})\frac{\zeta_{E_v}(s-j+n+\frac{r}{2})}{\zeta_{E_v}(s-j+n+\frac{r}{2}+1)}.$$

Thus
$$\begin{aligned}
c_{2,v}(s):
&=\prod_{\alpha \in \Sigma^{+}(w_2^*)}M_{\alpha,v}(s)\\
&=\vol(\CO_{E_v})^{2n-r-1}\cdot \vol(\CO_v)\frac{\zeta_{E_v}(2s)}{\zeta_{E_v}(2s+r-1)} \frac{\zeta_{F_v}(2s+r-1)}{\zeta_{F_v}(2s+r)} \frac{\zeta_{E_v}(s-n+\frac{3r}{2})}{\zeta_{E_v}(s+n-\frac{r}{2})},
\end{aligned}$$
and
$$\begin{aligned}
c_2(s)
&=\prod_v c_{2,v}(s)\\
&=\frac{\xi_E(2s)}{\xi_E(2s+r-1)} \frac{\xi_F(2s+r-1)}{\xi_F(2s+r)} \frac{\xi_E(s-n+\frac{3r}{2})}{\xi_E(s+n-\frac{r}{2})}.
\end{aligned}$$\\

In summary, we have the following
\begin{prop}\label{Prop 7.6}
Let $f^0$ be the spherical section in $I^{n,r}(s)$ with $f^0(1)=1$. Then
$$\begin{aligned}
& \CU_{1}(s)f^0=c_{1}(s)f^0,\\
& \CU_{2}(s)f^0=c_{2}(s)f^0,
\end{aligned}$$
where
$$\begin{aligned}
&c_{1}(s)=\frac{\xi_E(s+n-\frac{3r}{2})}{\xi_E(s+n-\frac{r}{2})},\\
&c_{2}(s)=\frac{\xi_E(2s)}{\xi_E(2s+r-1)} \frac{\xi_F(2s+r-1)}{\xi_F(2s+r)} \frac{\xi_E(s-n+\frac{3r}{2})}{\xi_E(s+n-\frac{r}{2})}.
\end{aligned}$$

\end{prop}

Proposition \ref{Prop 4.1} then follows immediately from Proposition \ref{Prop 7.4} and Proposition \ref{Prop 7.6}.

\end{document}